\documentclass{article}
\usepackage[latin1]{inputenc}
\usepackage[swedish,english]{babel}
\usepackage{amsmath,amsthm}
\usepackage{mathtools}
\usepackage{amssymb}
\usepackage{bm}
\usepackage{ifpdf}
\ifpdf
\usepackage[pdftex]{graphicx}
\usepackage{epstopdf}
\else
\usepackage[dvips]{graphicx}
\fi
\usepackage{caption}
\usepackage{subcaption}
\usepackage[pdftitle={The PDF title},
pdfauthor={Vidar Stiernstrom},
pdffitwindow=true,
breaklinks=true,
colorlinks=true,
urlcolor=blue,
linkcolor=red,
citecolor=red,
anchorcolor=red]{hyperref}
\usepackage{todonotes}
\usepackage{natbib}
\usepackage{doi}

\newcommand{\reviewerOne}[1]{{#1}}
\newcommand{\reviewerTwo}[1]{{{#1}}}
\newcommand{\changed}[1]{{{#1}}}

\numberwithin{equation}{section}
\numberwithin{table}{section}
\numberwithin{figure}{section}

\theoremstyle{plain}
\newtheorem{theorem}{Theorem}[section]
\newtheorem{lemma}[theorem]{Lemma}

\theoremstyle{definition}

\theoremstyle{remark}
\newtheorem*{remark}{Remark}

\newcommand{\ie}{i.\,e.\,}
\newcommand{\eg}{e.\,g.\,}

\newcommand{\ip}[2]{\left(#1,#2\right)}
\newcommand{\ipf}[2]{\ip{#1}{#2}_{\Gamma \times \T}}

\newcommand{\ipo}[2]{\ip{#1}{#2}_{\Omega \times \T}}
\newcommand{\ipdo}[2]{\ip{#1}{#2}_{\partial \Omega \times \T}}

\newcommand{\jump}[1]{\ensuremath{[\![#1]\!]} }
\newcommand{\dbar}[1]{\bar{\bar{#1}}}
\newcommand{\abs}[1]{\lvert #1 \rvert}

\newcommand{\T}{\mathcal{T}}
\newcommand{\bOmega}{\mathbf{\Omega}}

\newcommand{\bGamma}{\bm{\Gamma}}

\newcommand{\dI}{\partial_{I}}
\newcommand{\dK}{\partial_{K}}

\newcommand{\pdd}[2]{\frac{\partial#1}{\partial#2}}
\newcommand{\ddd}[2]{\frac{\delta#1}{\delta#2}}

\newcommand{\F}{\mathcal{F}}
\renewcommand{\L}{\mathcal{L}}

\newcommand{\brx}{\bar{x}}
\newcommand{\bru}{\bar{u}}
\newcommand{\brtau}{\bar{\tau}}
\newcommand{\brV}{\bar{V}}
\newcommand{\brT}{\bar{T}}
\newcommand{\brn}{\bar{n}}
\newcommand{\brr}{\bar{r}}
\newcommand{\brv}{\bar{v}}
\newcommand{\brR}{\bar{R}}
\newcommand{\brF}{\bar{F}}

\newcommand{\bx}{\mathbf{x}}
\newcommand{\bu}{\mathbf{u}}
\newcommand{\bv}{\mathbf{v}}
\newcommand{\bw}{\mathbf{w}}

\newcommand{\bp}{\mathbf{p}}
\newcommand{\bV}{\mathbf{V}}
\newcommand{\bPsi}{\bm{\Psi}}
\newcommand{\btau}{\bm{\tau}}

\newcommand{\brho}{\bm{\rho}}
\newcommand{\bmu}{\bm{\mu}}
\newcommand{\bZ}{\bm{Z}}

\newcommand{\be}{\mathbf{e}}

\newcommand{\bH}{\mathbf{H}}
\newcommand{\bD}{\mathbf{D}}

\newcommand{\bR}{\mathbf{R}}

\newcommand{\bI}{\mathbf{I}}
\newcommand{\bF}{\mathbf{F}}
\newcommand{\bG}{\mathbf{G}}
\newcommand{\bS}{\mathbf{S}}
\newcommand{\bQ}{\mathbf{Q}}
\newcommand{\bA}{\mathbf{A}}
\newcommand{\bU}{\mathbf{U}}

\newcommand{\bT}{\mathbf{T}}
\newcommand{\bL}{\mathbf{L}}

\title{Adjoint-based inversion for stress and frictional parameters in earthquake modeling}

\author{Vidar Stiernstr\"om\thanks{Department of Information Technology, Uppsala University, SE-751 05 Uppsala, Sweden, \href{mailto:cstierns@stanford.edu}{cstierns@stanford.edu}. Now at Department of Geophysics, Stanford University, Stanford CA 94305, USA}, Martin Almquist\thanks{Department of Information Technology, Uppsala University, SE-751 05 Uppsala, Sweden, \href{mailto:martin.almquist@it.uu.se}{martin.almquist@it.uu.se}}, Eric M. Dunham\thanks{Department of Geophysics, Stanford University, Stanford CA 94305, USA, \href{mailto:edunham@stanford.edu}{edunham@stanford.edu}}}
\date{}
\begin{document}

\maketitle

\begin{abstract}
  We present an adjoint-based optimization method to invert for stress and frictional parameters used in earthquake modeling. The forward problem is linear elastodynamics with nonlinear rate-and-state frictional faults. The misfit functional quantifies the difference between simulated and measured particle displacements or velocities at receiver locations. The misfit may include windowing or filtering operators. We derive the corresponding adjoint problem, which is linear elasticity with linearized rate-and-state friction \changed{and, for forward problems involving fault normal stress changes, nonzero fault opening}, with time-dependent coefficients derived from the forward solution. The gradient of the misfit is efficiently computed by convolving forward and adjoint variables on the fault. The method thus extends the framework of full-waveform inversion to include frictional faults with rate-and-state friction. In addition, we present a space-time dual-consistent discretization of a dynamic rupture problem with a rough fault in antiplane shear, using high-order accurate summation-by-parts finite differences in combination with explicit Runge--Kutta time integration. The dual consistency of the discretization ensures that the discrete adjoint-based gradient is the exact gradient of the discrete misfit functional as well as a consistent approximation of the continuous gradient. Our theoretical results are corroborated by inversions with synthetic data. We anticipate that adjoint-based inversion of seismic and/or geodetic data will be a powerful tool for studying earthquake source processes; it can also be used to interpret laboratory friction experiments.
\end{abstract}

\section{Introduction}
Earthquake cycle and dynamic rupture simulations with fault dynamics governed by rate-and-state friction have emerged as promising tools to better understand the processes governing earthquake nucleation and seismicity. Dynamic source models are complementary to kinematic source models that express the seismic wavefield and solid displacements through a convolution of the specified space-time history of slip with elastic Green's functions. Kinematic models are widely used in slip inversions, which can be set up as linear least squares problems due to the linear relation between slip and the wavefield. However, kinematic inversions only reveal how faults slip, but not why they slip in a certain way. Furthermore, the solutions to the inverse problem can violate certain physical constraints, such as dissipation rather than creation of mechanical energy during frictional sliding. Dynamic source models go beyond kinematic source models through the introduction of friction laws that are based on laboratory experiments. These friction laws involve nonlinear relations between fault shear and normal tractions, slip velocity, and one or more state variables that capture the dependence of frictional strength on the slip history of the interface. After specification of the initial stresses and frictional parameters, dynamic rupture and earthquake cycle models provide both the slip history and wavefield (or quasi-static solid displacement field) as part of the solution to the nonlinear problem. Connecting cycle and dynamic rupture simulations to real data permits determination of the stresses and frictional parameters, which are otherwise difficult or impossible to constrain. This requires augmenting dynamic source models with inversion capabilities. The inversion is a PDE-constrained optimization problem where we seek parameter values that minimize the misfit between model output and data. We propose a gradient-based optimization where the gradient of the misfit is computed with only two simulations: one of the forward problem and one of the adjoint problem. Gradient-based optimization with the adjoint method is the preferred approach in full waveform inversions (FWI) in seismic tomography \changed{and reflection seismology} \cite{Tromp2005, Fichtner2011, virieux_operto_2009}.
   
To date, dynamic rupture inversions have been done either by grid-search algorithms or Bayesian methods that require running thousands to order of a million forward models \cite{Peyrat2004,Gallovic2019a,Gallovic2019b,Premus2022}. Friction inversion for earthquake cycle models has been posed as a sequential data assimilation problem that is solved, for example, using the ensemble Kalman filter \cite{vanDinther2019,Hirahara2019}. Frictional properties are treated as part of the state vector and updated in time. However, these parameters should not be time-dependent and we seek an alternative formulation that respects this.

In this work, we present an adjoint-based optimization framework for inverting rate-and-state frictional parameters and initial stresses from geodetic and seismic observations. The method presented in this work extends FWI to include nonlinear rate-and-state friction laws. Similar to FWI, the adjoint equations are the same linear elasticity equations as the forward equations, but with fault dynamics governed by what resembles linearized rate-and-state friction \changed{and, for forward problems involving fault normal stress changes, a nonzero fault opening condition}. The adjoint friction law\changed{, state evolution equation, and the opening condition} include time-dependent coefficients that are functions of \changed{slip velocity, state, and normal stress} from the forward problem. The framework is presented for dynamic elasticity but may be straightforwardly adapted to quasi-static or quasi-dynamic settings, to be used in earthquake cycle or aseismic slip simulations. We note that a special case of this method has been derived for quasi-dynamic cycle models, for the spatially discretized problem with a boundary element discretization of the linear elastic slip-stress change relation, by Kano et al. \cite{Kano2013,Kano2015,Kano2020}. We also draw attention to new work exploring the use of physics-informed neural networks (PINNs) for inversion of rate-and-state friction parameters \cite{rucker2023physics}. \reviewerTwo{Additionally, FWI has been applied to seismic source inversion in, \eg, \cite{sjogreen_petersson_2014}, with seismic sources modeled as point moment tensor forcings in the elastic wave equation.}

To demonstrate the adjoint-based optimization framework we consider a 2D dynamic rupture problem with a rough fault in antiplane shear, discretized using finite difference methods and explicit Runge--Kutta time stepping methods. High-order finite difference methods satisfying a summation-by-parts (SBP) property have successfully been applied to linear elasticity with rate-and-state friction in \eg, \cite{kozdon_et_al_2013, erickson_dunham_2014, oreilly_et_al_2015, duru_dunham_2016, duru_et_al_2019, harvey_et_al_2023}, and extended to poro- and viscoelasticity in \cite{torberntsson_et_al_2018, alison_dunham_2018, heimisson_et_al_2019}. In this work, we utilize the recently developed boundary-optimized second derivative operators of \cite{stiernstrom_et_al_2023}, providing increased accuracy for problems where boundary or interface effects are of particular importance. The operators are combined with the non-stiff interface treatment based on characteristic variables presented in \cite{erickson_et_al_2022}. The fault interface conditions are imposed weakly, using simultaneous approximation terms (SAT). The resulting spatial SBP-SAT discretization allows for efficient time integration by means of explicit methods. \reviewerOne{We prove that the SBP-SAT discretization of the misfit functional and forward problem is dual consistent \cite{pierce_giles_2000, hartmann_2007, berg_nordstrom_2012, hicken_zingg_2014, ghasemi_thesis}, meaning that in deriving the semi-discrete adjoint equations one obtains a consistent discretization of the continuous adjoint problem. As a result,} the gradient of the semi-discrete problem is the exact gradient (to machine precision) of the semi-discrete misfit functional as well as a consistent approximation of the continuous gradient. \reviewerOne{Obtaining an exact gradient to the discrete misfit is beneficial from a numerical point of view, since discretization errors present in the gradient otherwise may negatively affect the convergence of the optimization algorithm \cite{giles_pierce_2001}, especially when simulations are carried out on marginally resolved grids.} \reviewerOne{That the characteristics-based SAT of \cite{erickson_et_al_2022}, when paired with an SBP discretization, yields a dual-consistent scheme is to our knowledge a new result.} We extend dual consistency to the fully discrete problem by employing the standard fourth-order Runge--Kutta time stepping and using the associated Runge--Kutta quadrature to discretize the time integral in the misfit functional \cite{sanz_serna_2016}. While our discretization is based on SBP finite differences, dual consistency should straightforwardly extend to other methods satisfying SBP or corresponding discrete integration-by-parts properties. This includes discontinuous- and continuous finite element methods as well as finite volume methods.

Our theoretical findings are corroborated by a series of numerical experiments on model problems using synthetic data. First, the discrete adjoint-based gradient is compared to a brute-force finite-difference approximation of the gradient. Second, we perform inversions for frictional parameters in an inverse crime setting, as well as using high-resolution synthetic data, demonstrating the capabilities of the method. 

The article is developed as follows: In Section \ref{sec:notation} we establish the notation used in the continuous and discrete analysis. In Section \ref{sec:elasticity} the equations of linear elasticity and rate-and-state friction are presented. In Section \ref{sec:inverse} the inverse problem is formulated as a PDE-constrained optimization problem. The adjoint-based gradient of the misfit functional is derived in Section \ref{sec:adjoint_eqs}. We demonstrate how filtered and windowed residuals in particle displacement and velocity enter the adjoint equations and discuss well-posedness of the \changed{forward and} adjoint problem. In Section \ref{sec:antiplane_shear} we present the dual-consistent SBP-SAT discretization for dynamic rupture simulations in antiplane shear and use it to invert synthetic data, demonstrating that the method indeed is capable of reconstructing frictional parameters. Finally, the article is concluded in Section \ref{sec:conclusion}.

\section{Notation}\label{sec:notation}
\subsection{Integrals, inner products and index notation}
The $L^2$-inner product for scalar functions $u$, $v$ on $\Omega$ is denoted
\begin{equation}\label{eq:inner_product_cont}
    \ip{u}{v}_\Omega := \int_\Omega uv \mathrm{d}\Omega.
\end{equation}
Similarly, define the bilinear form for integrating over a surface $\Gamma \in \partial\Omega$ as 
\begin{equation}\label{eq:bilinear_form_cont}
    \ip{u}{v}_\Gamma := \int_\Gamma uv \mathrm{d}\Gamma.
\end{equation}
The same notation is used when integrating over a space-time domain, \ie, for the domain $\Omega \times \T$ with $\T = [0,T]$, we write
\begin{equation}\label{eq:inner_product_space_time_cont}
  \ip{u}{v}_{\Omega \times \T} := \int_0^T \ip{u}{v}_\Omega \mathrm{d}t,
\end{equation}
and correspondingly for $\ip{u}{v}_{\Gamma \times \T}$. 

\reviewerOne{Vector-valued functions are denoted with a bar, \eg, $\bar{u} = [u_I]_{I=1}^d$ where $d = \dim(\Omega)$.} Throughout the article, we will make use of index notation. Here the Einstein summation convention is applied to sum over repeated indices $I, J, K, L$ used to denote spatial components. For instance, \eqref{eq:inner_product_space_time_cont} is extended to vector-valued functions $\bar{u}$ and $\bar{v}$ as
\reviewerTwo{
\begin{equation}\label{eq:inner_product_vector_valued_cont}
    \ip{u_I}{v_I}_{\Omega \times T} := \sum_{I=1}^d \ip{u_I}{v_I}_{\Omega \times \T} = \ip{\bar{u}}{\bar{v}}_{\Omega \times T}.
\end{equation}
}Furthermore, we will use the notation $\partial_I := \frac{\partial}{\partial x_I}$, and the summation convention then applies to derivatives. For instance, the variable coefficient Laplace operator is $\partial_I \mu \partial_I$. \reviewerOne{In a few places we make use of mixed index and vector notation. For instance, $f_I(\bar{u})$, means that the $I$th component of the vector-valued function $\bar{f}$ is a function of the components of $\bar{u}$.}

\subsection{Grid functions and discrete operators}
Boldface symbols will be used to denote discrete quantities, \eg, the discrete grid function approximating the function $u(\brx)$, $\brx \in \Omega$ is denoted $\bu$ and is defined on the grid $\bOmega$. \reviewerOne{Similarly, vector-valued grid functions are denoted with a bar}, \eg, the grid function restriction of $\bar{u}$ is denoted $\bar{\bu}$. Scalar functions acting as diagonal operators are denoted with double bars, \eg, the operator approximating $f(\brx)$ is $\dbar{\mathbf{f}} := \text{diag}(\mathbf{f})$.

$\bH$ with appropriate subscripts will be used to denote quadrature rules. For instance, a quadrature on $\bOmega$ is denoted $\bH_{\Omega}$. The notation for inner products is analogous to the continuous setting. For instance 
\begin{equation}\label{eq:inner_product_disc}
    \ip{\bu}{\bv}_{\bOmega} := \bu^T\bH_{\Omega}\bv.
\end{equation}
For numerical integration along a boundary grid segment $\bGamma$, we first introduce the \reviewerOne{boundary restriction operators $\be^T_\Gamma$}, where $\be_\Gamma^T\bu \approx u(\brx), \brx \in \Gamma$. A discrete bilinear form approximating \eqref{eq:bilinear_form_cont} is then defined as
\begin{equation}\label{eq:bilinear_form_disc_gamma}
    \ip{\bu}{\bv}_{\bm{\Gamma}} := (\be_\Gamma^T\bu)^T\bH_\Gamma (\be_\Gamma^T\bv),
\end{equation}
where $\bH_\Gamma$ is a quadrature rule on $\bGamma$. A bilinear form over the entire boundary grid $\bm{\partial \Omega}$ is defined by summing over all boundary grid segments, \ie,
\begin{equation}\label{eq:bilinear_form_disc}
  \ip{\bu}{\bv}_{\bm{\partial \Omega}} := \sum_{\bGamma \subset \bm{\partial \Omega}} \ip{\bu}{\bv}_{\bm{\Gamma}} .
\end{equation}
In a semi-discrete setting, where time is left continuous, we will also use semi-discrete inner products, \eg, $\ip{\bu}{\bv}_{\bOmega \times \T} := \int_0^T \bu^T\bH_{\bOmega}\bv \mathrm{d}t$. The summation convention for $I,J,K,L$ also applies to discrete relations. For instance, for vector-valued grid functions $\bar{\bv}$ and $\bar{\bu}$ the extension of \eqref{eq:inner_product_disc} is
\begin{equation}\label{eq:inner_product_vector_valued_discr}
    \ip{\bu_I}{\bv_I}_{\bOmega} := \sum_{I=1}^d \ip{\bu_I}{\bv_I}_{\bOmega}.
\end{equation}

\section{Linear elasticity with rate-and-state friction}\label{sec:elasticity}
Consider a linear elastic medium $\Omega = \Omega_- \cup \Omega_+$, separated by a frictional fault $\Gamma = \Omega_- \cap \Omega_+$. Let $\bru$ denote the displacement vector. From Hooke's law, the changes in the stress tensor $\sigma$ and traction vector $\brT$ for elastic deformations about a prestressed reference configuration are
\begin{equation}
  \sigma_{IJ} \changed{:=} C_{IJKL} \dK u_L,   
\end{equation} 
and
\begin{equation}
  T_J \changed{:=} n_I \sigma_{IJ},
\end{equation}
where $\bar{n}$ denotes the outward unit normal and $C_{IJKL}$ is the elastic stiffness tensor. (We note that these stress changes are sometimes denoted as $\Delta \sigma_{IJ}$.) \changed{Henceforth, superscripts $+$ and $-$ will be used to distinguish between fields on the $+$ and $-$ sides of the fault.} The \changed{change in} shear traction \changed{on the $-$ side of the fault} is the projection of $\brT^{-}$ onto the fault plane, given by
\begin{equation}
  \brtau \changed{:=} \brT^{\changed{-}} - (\brn^{\changed{-}} \cdot \brT^{\changed{-}}) \brn^{\changed{-}},
\end{equation}
\changed{
while the change in compressive normal stress is
\begin{equation}
  \sigma_n \changed{:=} - \brn^- \cdot \brT^-.
\end{equation}

We denote the the jump in particle velocity across the fault by
\begin{equation}
\jump{\dot{\bru}} := \dot{\bru}^+ - \dot{\bru}^- = V_n \brn^- + \brV,
\end{equation}
where the opening velocity $V_n$ and the slip velocity $\brV$ are defined as the normal component of $\jump{\dot{\bru}}$ and the projection of $\jump{\dot{\bru}}$ onto the fault plane, respectively, \ie,
\begin{align}
  V_n &:= \brn^- \cdot \jump{\dot{\bru}}, \label{eq:opening_velocity} \\
 \brV &:= \jump{\dot{\bru}} - V_n \brn^-. \label{eq:slip_velocity} 
\end{align}

Exploiting linearity of the elasticity equations, the total stress on the fault is the sum of prestress, assumed to be in equilibrium with any external loading, and changes in stress due to the displacement $\bru$.} \reviewerOne{Denoting the initial compressive normal stress as $\sigma_n^0$, the total compressive normal stress is}
\changed{\begin{equation}\label{eq:tot_fault_normal_stress}
 \sigma_n^{tot} := \sigma_n^0 + \sigma_n.
\end{equation}
The total shear traction vector is
\begin{equation}\label{eq:tot_fault_sher_stress}
  \brtau^{tot} := \brtau^0 + \changed{\brtau} - \eta \brV, 
\end{equation}
where $\bar{\tau}^0$ is the initial shear traction and $\eta \brV$ is the radiation damping term. The radiation damping term is only used in quasi-dynamic models, where the coefficient $\eta$ is the impedance of shear waves radiating away from the fault \cite{rice_1993}. In the fully dynamic setting, $\eta$ is set to zero. Although we only perform experiments with $\eta=0$ in this paper, we include the radiation damping term in the derivation for completeness.}

Force balance across the fault and the condition of no opening or interpenetration of the fault walls are stated as
\begin{align}
  \brT^+ &= - \brT^-, \label{eq:force_balance} \\
  \changed{V_n} &\changed{= 0} \label{eq:no_opening}. 
\end{align}
\reviewerOne{The force balance condition \eqref{eq:force_balance} is what permits the simplified notation in \eqref{eq:tot_fault_normal_stress} - \eqref{eq:tot_fault_sher_stress}, where $\brtau^{tot}$ and $\sigma_n^{tot}$, without superscripts $+$ or $-$, are used to represent values for shear and normal tractions on the two sides of the fault.}
\reviewerTwo{We also restrict attention to problems where $\sigma_n^{tot}$ remains compressive. A more general problem formulation would allow for fault opening and the transition to traction-free fault walls when constraining the fault against opening would lead to tensile $\sigma_n^{tot}$. This requires enforcing inequality constraints \cite{day_et_al_2005}.
It is more common in earthquake modeling to always enforce the no opening condition, but then to handle tensile normal stresses by setting $\sigma_n^{tot}$ to zero when evaluating frictional strength \cite{harris_et_al_2009,harris_et_al_2018,erickson_et_al_2023}.}

Equating shear traction with frictional shear strength yields
\begin{equation}\label{eq:fault_strength}
  \brtau^{tot} = \sigma_n^{tot} f \frac{\brV}{\abs{\brV}},
\end{equation}
where $f$ is the friction coefficient. \reviewerOne{In this work, we consider rate- and state-dependent friction coefficients $f(\abs{\brV},\Psi)$, where the dimensionless state variable $\Psi$ is a measure of the interface contact strength related to the past history of sliding \cite{dieterich_1979,rice_ruina_1983,ruina_1983}.} Solving \eqref{eq:tot_fault_sher_stress} and \eqref{eq:fault_strength} for \changed{$\brtau$ and using tangential components of \eqref{eq:force_balance} leads to the rate-and-state friction law, given by
\begin{equation} \label{eq:friction_law_fwd}
  \brtau = \brF(\brV,\Psi, \sigma_n ) :=  \left(\sigma_n^0 + \sigma_n \right) f(\abs{\brV},\Psi) \frac{\brV}{\abs{\brV}} - \brtau^0 + \eta \brV.  
\end{equation}
Note that $\brtau$, $\brV$ and $\brF$ are three-component vectors. However, since $\brtau$ and $\brV$ are coplanar, one of the three equations implied by $\brtau = \brF$ is $0=0$, and the third interface condition is given by the no opening condition \eqref{eq:no_opening}.

The state variable $\Psi$ is in turn governed by a state evolution equation
\begin{equation}\label{eq:state_evolution_fwd}
  \dot{\Psi} = G(\brV, \Psi, \sigma_n, \dot{\sigma}_n).
\end{equation}
} \reviewerTwo{The rate-and-state friction coefficient $f(|\bar{V}|,\Psi)$ and state evolution equation $G(\bar{V},\Psi, \sigma_n, \dot{\sigma}_n)$ are empirical laws obtained from laboratory experiments \cite{dieterich_1979,linker_dieterich_1992,rice_ruina_1983,marone_1998}. They contain parameters which can be difficult to constrain at depths that are inaccessible to drilling. Even if laboratory friction experiments are performed on core samples from drilling, the relation between experimental parameter values and parameter values relevant for large-scale slip remains unclear. Examples of such parameters are the direct effect parameter $a$ and the state evolution parameter $b$, which we define in a later section. Their difference, $a-b$, determines if friction increases or decreases with slip velocity, which controls the stability of sliding and the possibility of unstable rupture. These parameters may be spatially variable along the fault, but are time-independent.}

\changed{We are now ready to state the forward problem.} Let $\bar{Q}(\brx,t)$ denote external forcings. The governing equations of the forward problem are given by
\begin{equation}\label{eq:elastic_3d_fwd}
  \begin{array}{lll}
  \smallskip
    \rho  \ddot{u}_J = \dI C_{IJKL} \dK u_L + Q_J, & \brx\in \Omega, & t\in \T,\\
    \smallskip
    \bru = \bru_0, \quad \dot{\bru} = \brv_0, & \brx \in \Omega, & t = 0, \\
    \smallskip
    L \bru = \bar{g} , & \brx \in \partial\Omega, & t\in \T, \\
    \smallskip
    \changed{V_n=0} , & \brx \in \Gamma, & t\in \T, \\
    \smallskip
    \changed{\brT^+ = -\brT^-}, & \brx \in \Gamma, & t\in \T, \\
    \smallskip
    \displaystyle
    \changed{\brtau = \brF(\brV, \Psi , \sigma_n)} , & \brx \in \Gamma, & t\in \T, \\
    \smallskip
    \dot{\Psi} = G(\brV,\Psi \changed{,\sigma_n, \dot{\sigma}_n}), & \brx \in \Gamma, & t\in \T, \\
    \smallskip
    \Psi = \Psi_0, & \brx \in \Gamma, & t = 0,
  \end{array}
\end{equation}
where the first equation is the momentum balance equation \changed{with Hooke's law used to replace stress with spatial derivatives of displacement}, $\bru_0$ and $\brv_0$ are initial data, and $L \bru = \bar{g}(\brx,t)$ denotes boundary conditions on exterior boundaries, for a boundary operator $L$.  The data $\bru_0$, $\brv_0$, and $\bar{g}$ are assumed to be independent of \reviewerTwo{frictional parameters or initial stress} that we later invert for. We will consider $L$ such that for homogeneous boundary data $\bar{g} = \bar{0}$, the boundary conditions are of the form
\begin{equation} \label{eq:bc_general_fwd}
  \changed{T}_J = -\alpha_{JL} \dot{u}_L,
\end{equation}
where $\alpha_{JL} = \alpha_{LJ}$ is positive semi-definite. Note that this includes \reviewerOne{traction-free} conditions ($\alpha_{JL}=0$), \changed{rigid-wall} conditions ($\alpha_{JL} \rightarrow \infty$), and characteristic \changed{non-reflecting} conditions \cite{petersson_sjogreen_2009}. \changed{Again, we emphasize that the friction law $\brtau = \brF$ constitutes two equations, such that there is no overspecification of the fault interface conditions.}

\section{The inverse problem}\label{sec:inverse}
Now consider $N_{rec}$ receivers positioned at $\brx = \brx_r^{(k)}$, $k = 1,\dots,N_{rec}$, each with a time series of measurements $\bar{m}_{data}^{(k)}(t)$ of either particle displacements ($\bar{m} = \bar{u}$) or velocities (\reviewerTwo{$\bar{m} = \dot{\bar{u}}$}). We seek parameter values that minimize the residual $\brr$, defined as the difference between model predictions at the receiver locations (obtained by solving \eqref{eq:elastic_3d_fwd}) and the measured data, \ie, 
\begin{equation}\label{eq:residual}
  \brr^{(k)}(t) = \bar{m}(\brx_r^{(k)},t) - \bar{m}_{data}^{(k)}(t) .
\end{equation}
In practice, the residual is often filtered or windowed. We therefore consider applying linear operators $W^{(k)}$ to obtain the adjusted residuals
\begin{equation}\label{eq:adjusted_residual}
  \brR^{(k)}(t) = W^{(k)}\left[ \brr^{(k)} \right](t) .
\end{equation} 
The notation $W^{(k)}\left[ \brr^{(k)} \right]$ means that the linear operator $W^{(k)}$ acts on $\brr^{(k)}$. The result is the time-dependent adjusted residual. We define the misfit functional $\F$ as the sum of the $L^2$-norms of the adjusted residuals:
\begin{equation}\label{eq:misfit}
  \F = \frac{1}{2} \sum_{k=1}^{N_{rec}} \int \limits_{0}^{T} \abs{ \brR^{(k)}(t) }^2 \, \mathrm{d}t. 
\end{equation}
Weighting of the different terms can be included in the $W^{(k)}$ operator.

For a model parameter $p$ in \eqref{eq:elastic_3d_fwd}, the inverse problem is the PDE-constrained optimization problem given by
\begin{equation}\label{eq:inverse_prob_3d}
  \min_{p} \F \text{ subject to \eqref{eq:elastic_3d_fwd}}.
\end{equation}
\reviewerTwo{In this work, we particularly consider the case where $p(\bar{x})$, $\bar{x} \in \Gamma$ is a spatially variable parameter in the rate-and-state friction coefficient $f$ in \eqref{eq:friction_law_fwd} and/or state evolution equation $G$ in \eqref{eq:state_evolution_fwd} (\eg, the direct effect parameter $a$). Additionally, $p$ may also represent initial stresses $\brtau^0$ and $\sigma_n^0$ in $\eqref{eq:friction_law_fwd}$ or initial state $\Psi_0$.} To solve \eqref{eq:inverse_prob_3d} using gradient-based optimization algorithms, the misfit gradient $\ddd{\F}{p}$ is required\reviewerTwo{, where $\ddd{}{p}$ denotes the functional derivative with respect to $p$.}  However, taking the functional derivative of \eqref{eq:misfit} directly requires computing $\ddd{\bar{m}}{p}$ (see Section \ref{sec:misfits}). In a discrete setting, where $\bp$ is represented by $N$ degrees of freedom, even the simplest first-order difference approximation of $\ddd{\F}{p}$ would require \eqref{eq:elastic_3d_fwd} to be solved numerically $N+1$ times. This is of course not feasible for large-scale 3D computations. We therefore seek an alternative expression through the adjoint-state framework.

\section{Adjoint equations and misfit gradient}\label{sec:adjoint_eqs}
We begin this section by presenting the adjoint equations and the adjoint-based gradient of the misfit \eqref{eq:misfit} in the form of a theorem. To this end, introduce the following adjoint variables:
\begin{equation}\label{eq:adjoint_variables}
  \begin{array}{ll}
  \smallskip
  \bru^{\dagger} & \text{(adjoint displacement)}, \\
  \smallskip
  \Psi^{\dagger} & \text{(adjoint state variable)}, \\
  \smallskip
  \changed{V^{\dagger}_n :=  - \brn^- \cdot \jump{\dot{\bru}^{\dagger}}} & \changed{\text{(adjoint opening velocity)},} \\
  \smallskip
  \brV^{\dagger} :=  \changed{-\jump{\dot{\bru}^{\dagger}} - V_n^{\dagger} \brn^-} & \text{(adjoint slip velocity)}, \\
  \smallskip
  T^{\dagger}_J :=  n_I C_{IJKL} \dK u^{\dagger}_L & \text{(adjoint traction)}, \\
  
  \smallskip
  \brtau^\dagger := \brT^{\dagger\changed{-}} - (\brn^{\changed{-}} \cdot \brT^{\dagger\changed{-}}) \brn^{\changed{-}} & \text{(adjoint shear traction)}, \\
    \smallskip
  \changed{\sigma_n^\dagger := - \brn^- \cdot \brT^{\dagger-}} & \changed{\text{(adjoint normal stress)}}. \\
\end{array}
\end{equation}
The negative signs in \changed{$V_n^\dagger$ and $\brV^{\dagger}$ are} for notational convenience. As shortly explained, the adjoint problem will later be phrased in reversed time in which case \changed{$V_n^\dagger$ and $\brV^{\dagger}$ are} defined analogously to \changed{\eqref{eq:opening_velocity}-\eqref{eq:slip_velocity}}. Furthermore, let $\bar{Q}^\dagger(\brx,t)$ be the adjoint source term. It consists of the residuals \eqref{eq:adjusted_residual} acting as point forces at the receiver locations. For a complete description, see Section \ref{sec:misfits}. For now, it is sufficient that $\bar{Q}^\dagger$ satisfies the relation
\begin{equation}\label{eq:adjoint_source_rel}
  \ip{Q_J^\dagger}{\ddd{\dot{u}_J}{p}}_{\Omega\times\T} = \ddd{\F}{p}.
\end{equation}

We now define the adjoint equations to the inverse problem \eqref{eq:inverse_prob_3d} as
\begin{equation}\label{eq:elastic_3d_adj}
  \begin{array}{lll}
  \smallskip
    \rho  \ddot{u}^{\dagger }_J = \dI C_{IJKL} \dK u^{\dagger }_L + Q^{\dagger }_J, & \brx\in \Omega, & t\in \T,\\
    \smallskip
    \bru^{\dagger } = \bar{0}, \quad \dot{\bru}^{\dagger } = \bar{0}, & \brx \in \Omega, & t = T, \\
    \smallskip
    L^\dagger \bru^{\dagger} = 0 , & \brx \in \partial\Omega, & t\in \T, \\
    \smallskip
    \changed{V_n^\dagger = H^\dagger(\bar{V}^\dagger,\Psi^\dagger)}, & \brx \in \Gamma, & t\in \T, \\
    \smallskip
    \changed{\brT^{\dagger+} = -\brT^{\dagger-}}, & \brx \in \Gamma, & t\in \T, \\
    \smallskip
    \displaystyle
    \changed{\brtau^{\dagger} = \brF^{\dagger}(\brV^{\dagger}, \Psi^{\dagger})} , & \brx \in \Gamma, & t\in \T, \\
    \smallskip
    -\dot{\Psi}^{\dagger} = G^{\dagger}(\brV^{\dagger},\Psi^{\dagger}), & \brx \in \Gamma, & t\in \T, \\
    \smallskip
    \Psi^{\dagger} = 0, & \brx \in \Gamma, & t = T.
  \end{array}
\end{equation}
\changed{The functions $F^{\dagger}$, $G^{\dagger}$, and $H^{\dagger}$ are linear in their arguments. They govern the fault dynamics of the adjoint problem and are discussed in more detail below.} The adjoint problem is equipped with terminal conditions (imposed at time $t=T$) and is naturally solved in reversed time. It is common to introduce the change of variables $t^{\dagger} = T - t$. We utilize this when discretizing \eqref{eq:elastic_3d_adj} (see Section \ref{sec:antiplane_shear_time}). Here, for notational convenience, we keep the original time $t$ in the derivation of the adjoint-based gradient. Note that $\pdd{}{t} = - \pdd{}{t^\dagger}$, and therefore
\begin{equation}\label{eq:adjoint_time_der_rel}
     -\dot{\Psi}^{\dagger} = \pdd{\Psi^{\dagger}}{t^\dagger}, \quad 
     \changed{-\jump{\dot\bru^\dagger} = \pdd{}{t^\dagger}{\jump{\bru^\dagger}}}.
\end{equation}
Notably, the boundary operator is self-adjoint in reversed time, \ie, $L^\dagger = L$, and the adjoint boundary conditions can be formulated as
\begin{equation} \label{eq:bc_general_adj}
  \changed{T}^\dagger_J = \alpha_{JL} \dot{u}^\dagger_L = -\alpha_{JL} \pdd{u^\dagger_L}{t^\dagger},
\end{equation}
(cf.\, \eqref{eq:bc_general_fwd}). As mentioned in Section \ref{sec:elasticity}, \eqref{eq:bc_general_adj} includes \changed{rigid-wall}, \reviewerOne{traction-free}, and characteristic \changed{non-reflecting} conditions. The relations \eqref{eq:adjoint_time_der_rel} - \eqref{eq:bc_general_adj} show that the adjoint problem \eqref{eq:elastic_3d_adj}, when considered in reversed time, is of the same form as the forward problem \eqref{eq:elastic_3d_fwd}.

The adjoint friction law and state evolution equation in \eqref{eq:elastic_3d_adj} are given by
\begin{equation}\label{eq:friction_law_adj}
 \changed{\tau_J^\dagger = }F^{\dagger}_J(\brV^\dagger, \Psi^{\dagger}) \changed{:=} \pdd{F_I}{V_J} V^{\dagger}_I + \pdd{G}{V_J} \Psi^{\dagger}, 
\end{equation}
and
\begin{equation}\label{eq:state_evolution_adj}
  \changed{-\dot{\Psi}^{\dagger} = }G^{\dagger}(\brV^{\dagger},\Psi^{\dagger}) \changed{:=} \pdd{F_J}{\Psi}  V^{\dagger}_J +\pdd{G}{\Psi}\Psi^{\dagger}.
\end{equation}
Note that \eqref{eq:friction_law_adj} - \eqref{eq:state_evolution_adj} resemble the equations of linearized rate-and-state friction. However, the coefficients are time-dependent and are actually functions of the forward variables, \eg $ \pdd{F_I}{V_J} = \pdd{F_I}{V_J}(\bar{V},\Psi, \changed{\sigma_n})$. \changed{Instead of the no opening or interpenetration condition \eqref{eq:no_opening} in the forward problem, the adjoint problem satisfies an equation for nonzero adjoint opening velocity if the forward problem involves normal stress changes, given by
\begin{equation}\label{eq:adjoint_opening}
\begin{aligned}
 V_n^\dagger &= H^\dagger(\bar{V}^\dagger,\Psi^\dagger) \\
 &\changed{:=}\left[\pdd{G}{\dot{\sigma}_n}\pdd{F_J}{\Psi}+\pdd{F_J}{\sigma_n}\right]V_J^\dagger  + \left[\pdd{G}{\sigma_n} + \pdd{G}{\dot{\sigma}_n}\pdd{G}{\Psi} - \frac{\mathrm{d}}{\mathrm{d}t}\left(\pdd{G}{\dot{\sigma}_n}\right)\right] \Psi^{\dagger}. 
 \end{aligned}
\end{equation}
}

With the adjoint problem defined, we are ready to state the first major result of this paper.
\begin{theorem}\label{thm:cont_grad}
  Let $\changed{\brV^{\dagger}}$, $\Psi^{\dagger}$, satisfy \eqref{eq:elastic_3d_adj}. Further, let $p$ be a parameter in $\bar{F}$ or $G$, and let $\Psi_0$ be the initial state. Then, the gradient of the misfit in the inverse problem \eqref{eq:inverse_prob_3d} is given by
  \begin{equation}\label{eq:gradient_adj}
    \begin{aligned}
      \ddd{\F}{p} &= - \ip{V^\dagger_J}{\pdd{F_J}{p}}_{\T} - \ip{\Psi^\dagger}{\pdd{G}{p} }_{\T}, \\
      \ddd{\F}{\Psi_0} &= \Psi_0^\dagger,
    \end{aligned}
  \end{equation}
  where $\Psi^\dagger_{0} = \Psi^\dagger(t = 0)$.
\end{theorem}
\begin{proof}
  See section \ref{sec:gradient_derivation}.
\end{proof}

\subsection{Derivation of the adjoint-based gradient}\label{sec:gradient_derivation}
To prove Theorem \ref{thm:cont_grad}, we formulate the Lagrangian cost functional
\begin{equation}\label{eq:lagrangian}
\begin{aligned}
  \L = \F &+ \ipo{\dot{u}^{\dagger}_J}{\rho \ddot{u}_J - \dI C_{IJKL} \dK u_L - Q_J} \\
          &+ \ipf{\Psi^{\dagger}}{\dot{\Psi} - G},
\end{aligned}
\end{equation}
\reviewerOne{where $\bar{u}^\dagger$ and $\Psi^\dagger$ are solutions to \eqref{eq:elastic_3d_adj}. Note that for any $\bru$ and $\Psi$ satisfying \eqref{eq:elastic_3d_fwd}, $\L = \F$ and further $\ddd{\L}{p} = \ddd{\F}{p}$ by the Lagrange multiplier theorem \cite{troltzsch}. 
Also note that we are using adjoint velocity $\dot{\bar{u}}^\dagger$ as a Lagrange multiplier when forming the cost functional. This choice is explained in a remark at the end of this section.} Moreover, since initial- and boundary data in \eqref{eq:elastic_3d_fwd} are independent of $p$ we may set $\bru_0 = 0$, $\brv_0 = 0$ and $\bar{g} = 0$ for the remaining part of the derivation.

Before considering the first variation $\delta\L$, we rewrite the cost functional by performing partial integrations to shift derivatives from $\bru$ to $\bru^\dagger$. Integrating by parts in time yields
\begin{equation} \label{eq:cost_kinetic_ibp}
  \ipo{\dot{u}^\dagger_J}{\rho \ddot{u}_J} = - \ipo{\rho \ddot{u}^\dagger_J}{\dot{u}_J},
\end{equation}
where terms at $t=0$ and $t=T$ vanish due to the initial and terminal conditions $\dot{\bru} = 0$ and $ \dot{\bru}^\dagger = 0$ in \eqref{eq:elastic_3d_fwd} and \eqref{eq:elastic_3d_adj}. Next, consider the spatial elastic operator and integrate by parts in space to obtain
\begin{equation}
\begin{aligned}
  - \ipo{\dot{u}^\dagger_J}{\dI C_{IJKL} \dK u_L} = &-\ipdo{\dot{u}^\dagger_J}{\changed{T}_J} + \ipdo{\dot{\changed{T}}^\dagger_J}{u_J} \\
  &-\ipf{\dot{u}^{\dagger +}_J}{\changed{T}^+_J} + \ipf{\dot{\changed{T}}^{\dagger+}_J}{u^+_J} \\
  &-\ipf{\dot{u}^{\dagger -}_J}{\changed{T}^-_J} + \ipf{\dot{\changed{T}}^{\dagger-}_J}{u^-_J} \\
  &- \ipo{\dI C_{IJKL} \dK \dot{u}^\dagger_L}{u_J},
\end{aligned}
\end{equation}
where we used the major symmetry of the stiffness tensor: $C_{IJKL} = C_{KLIJ}$. Integrating by parts in time leads to
\begin{equation}
\begin{aligned}
  \ipdo{\dot{\changed{T}}^\dagger_J}{u_J} &= - \ipdo{\changed{T}^\dagger_J}{\dot{u}_J} \\
  \ipf{\dot{\changed{T}}^{\dagger \pm}_J}{u^\pm_J} &= - \ipf{\changed{T}^{\dagger \pm}_J}{\dot{u}^\pm_J} \\
\end{aligned}
\end{equation}
and
\begin{equation}
\begin{aligned}
  - \ipo{\dI C_{IJKL} \dK \dot{u}^\dagger_L}{u_J} = \ipo{\dI C_{IJKL} \dK u^\dagger_L}{\dot{u}_J} ,
\end{aligned}
\end{equation}
where, again, terms at the initial and final times vanish due to the initial and terminal conditions $\bru = 0$, $\bru^\dagger = 0$.
We have now derived
\begin{equation}
\begin{aligned}
  \ipo{\dot{u}^\dagger_J}{\rho \ddot{u}_J - \dI C_{IJKL} \dK u_L} = &- \ipo{\rho \ddot{u}^\dagger_J - \dI C_{IJKL} \dK u^\dagger_L}{\dot{u}_J} \\
  &\underbrace{- \ipdo{\dot{u}^\dagger_J}{\changed{T}_J} - \ipdo{\changed{T}^\dagger_J}{\dot{u}_J}}_{BT_{ext}} \\
  &\underbrace{- \ipf{\dot{u}^{\dagger+}_J}{\changed{T}^+_J} - \ipf{\changed{T}^{\dagger+}_J}{\dot{u}^+_J}}_{BT_{fault}^+} \\
  &\underbrace{- \ipf{\dot{u}^{\dagger-}_J}{\changed{T}^-_J} - \ipf{\changed{T}^{\dagger-}_J}{\dot{u}^-_J}}_{BT_{fault}^-}
  .
\end{aligned}
\end{equation}
To simplify this expression, we first note that the exterior boundary terms $BT_{ext}$ vanish. To see this, we use that the forward and adjoint boundary conditions satisfy \eqref{eq:bc_general_fwd} and \eqref{eq:bc_general_adj}, and $\alpha_{JL} = \alpha_{LJ}$, such that
\begin{equation}
  \begin{aligned}
  BT_{ext} &= \ipdo{\dot{u}^\dagger_J}{\alpha_{JL} \dot{u}_L} - \ipdo{\changed{T}^\dagger_J}{\dot{u}_J}\\
  &= \ipdo{\alpha_{JL} \dot{u}^\dagger_L - \changed{T}^{\dagger}_J}{\dot{u}_J} = 0.
  \end{aligned}
\end{equation}
This naturally holds for self-adjoint boundary operators $L = L^\dagger$. Second, for the fault terms $BT_{fault}^{\pm}$ , \changed{we use the forward and adjoint force balances, $\brT^+ = -\brT^-$ and $\brT^{\dagger+} = -\brT^{\dagger-}$, to obtain
\begin{equation}
BT_{fault}^+ + BT_{fault}^- = \ipf{\jump{\dot{u}_J^\dagger}}{T_J^-} + \ipf{T_J^{\dagger-}}{\jump{\dot{u}_J}}. 
\end{equation}
Next, we decompose tractions and jumps in particle velocity into fault-parallel and fault-normal components:
\begin{equation}
\begin{array}{rc}
  \medskip
  \jump{\dot{u}_J} = V_J + V_n n_J^-, &
    T_J^- = \tau_J - \sigma_n n_J^-, \\
    -\jump{\dot{u}_J^\dagger} = V_J^\dagger + V_n^\dagger n_J^-, &
    T_J^{\dagger-} = \tau_J^{\dagger} - \sigma_n^\dagger n_J^-.
\end{array}
\end{equation}
Using the decomposition, together with the forward and adjoint friction laws $\tau_J = F_J$ and $\tau_J^{\dagger} = F_J^\dagger$ and the no opening condition $V_n = 0$, allows us to write
\begin{equation}
BT_{fault}^+ + BT_{fault}^- = -\ipf{V_J^\dagger}{F_J} + \ipf{F_J^{\dagger}}{V_J} + \ipf{V_n^\dagger}{\sigma_n}.
\end{equation}
}

The last term in the Lagrangian cost functional \changed{\eqref{eq:lagrangian} containing the state evolution equation} can be rewritten, using integration by parts in time,
\begin{equation}
  \ipf{\Psi^{\dagger}}{\dot{\Psi} - G} = -\ipf{\dot{\Psi}^{\dagger}}{\Psi} - \ipf{\Psi^{\dagger}}{G} + \ip{\Psi^{\dagger}_0}{\Psi_0}_{\Gamma},  
\end{equation}
where $\Psi^{\dagger}_0, \Psi_0$ are the state variables at $t=0$. The corresponding term at $t=T$ vanishes due to the terminal condition $\Psi(t=T) = 0$ in \eqref{eq:elastic_3d_adj}.
We have derived
\begin{equation}
\begin{aligned}
  \L = \F &- \ipo{\rho \ddot{u}^\dagger_J - \dI C_{IJKL} \dK u^\dagger_L}{\dot{u}_J} -\ipo{\dot{u}^\dagger_J}{Q_J} \\
  &-\ipf{V^{\dagger}_J}{F_J} + \changed{\ipf{F^{\dagger}_J}{V_J}} \changed{+\ipf{V_n^\dagger}{\sigma_n}} \\
  &-\ipf{\dot{\Psi}^{\dagger}}{\Psi} - \ipf{\Psi^{\dagger}}{G} + \ip{\Psi^{\dagger}_0}{\Psi_0}_{\Gamma}.
\end{aligned}
\end{equation}

We are now ready to derive the first variation of $\L$. First, $\delta Q_J = 0$ since external forces do not depend on frictional parameters or initial stresses. Second, by the chain rule,
\begin{equation}
  \delta F_J = \pdd{F_J}{V_I} \delta V_I + \pdd{F_J}{\Psi} \delta \Psi + \changed{\pdd{F_J}{\sigma_n} \delta \sigma_n +} \pdd{F_J}{p} \delta p,
\end{equation}
\begin{equation}
  \delta G = \pdd{G}{V_I} \delta V_I + \pdd{G}{\Psi} \delta \Psi + \changed{\pdd{G}{\sigma_n} \delta \sigma_n + \pdd{G}{\dot{\sigma}_n} \delta \dot{\sigma}_n +} \pdd{G}{p} \delta p.
\end{equation}
The above considerations together with the relation \eqref{eq:adjoint_source_rel} lead to
\begin{equation}\label{eq:lagrangian_variation_unsimplified}
  \begin{aligned}
  \delta \L = &- \ipo{\rho \ddot{u}^\dagger_J - \dI C_{IJKL} \dK u^\dagger_L - Q_J^\dagger}{\delta \dot{u}_J} \\
  &-\ipf{V^{\dagger}_J}{\pdd{F_J}{V_I} \delta V_I + \pdd{F_J}{\Psi} \delta \Psi + \changed{\pdd{F_J}{\sigma_n} \delta \sigma_n +} \pdd{F_J}{p}\delta p } \\
  &+ \ipf{F^{\dagger}_J }{\delta V_J} \changed{+\ipf{V_n^\dagger}{\delta \sigma_n}}  \\
  & -\ipf{\dot{\Psi}^{\dagger}}{\delta \Psi} - \ipf{\Psi^{\dagger}}{\pdd{G}{V_I} \delta V_I + \pdd{G}{\Psi} \delta \Psi + \changed{\pdd{G}{\sigma_n} \delta \sigma_n + \pdd{G}{\dot{\sigma}_n} \delta \dot{\sigma}_n +} \pdd{G}{p} \delta p} \\
  &+ \ip{\Psi^{\dagger}_0}{\delta \Psi_0}_{\Gamma}.
  \end{aligned}
\end{equation}
\changed{
For the term involving $\delta \dot{\sigma}_n$, integrating by parts in time yields
\begin{equation}
- \ipf{\Psi^{\dagger}}{\pdd{G}{\dot{\sigma}_n} \delta \dot{\sigma}_n} = \ipf{\pdd{G}{\dot{\sigma}_n}\dot{\Psi}^{\dagger} +  \frac{\mathrm{d}}{\mathrm{d}t}\left(\pdd{G}{\dot{\sigma}_n}\right)\Psi^{\dagger} }{ \delta \sigma_n},
\end{equation}
where we used that $\Psi^\dagger(t=T)=0$ and $\delta \sigma_n(t=0) = 0$ due to the terminal and initial conditions. Using the adjoint state evolution equation \eqref{eq:state_evolution_adj} to substitute $\dot{\Psi}^\dagger$ yields
\begin{equation}
\begin{aligned}
&-\ipf{\Psi^{\dagger}}{\pdd{G}{\dot{\sigma}_n} \delta \dot{\sigma}_n} = \\
& -\ipf{\pdd{G}{\dot{\sigma}_n}\pdd{F_J}{\Psi}V_J^\dagger  + \left[\pdd{G}{\dot{\sigma}_n}\pdd{G}{\Psi} - \frac{\mathrm{d}}{\mathrm{d}t}\left(\pdd{G}{\dot{\sigma}_n}\right)\right]\Psi^{\dagger}}{ \delta \sigma_n} .
\end{aligned}
\end{equation}
}

The volume term \changed{in \eqref{eq:lagrangian_variation_unsimplified}} vanishes due to the adjoint momentum balance in \eqref{eq:elastic_3d_adj}. Gathering the remaining terms leads to
\begin{equation}
  \begin{aligned}
  \delta \L &= \ipf{F^{\dagger}_J - \pdd{F_I}{V_J} V^{\dagger}_I - \pdd{G}{V_J} \Psi^{\dagger} }{ \delta V_J} \\
  &- \ipf{\dot{\Psi}^{\dagger} + \pdd{F_J}{\Psi}  V^{\dagger}_J +\pdd{G}{\Psi}\Psi^{\dagger} }{\delta \Psi} \\
  & \changed{ + \ipf{V_n^\dagger - \left[\pdd{G}{\dot{\sigma}_n}\pdd{F_J}{\Psi}+\pdd{F_J}{\sigma_n}\right]V_J^\dagger  - \left[\pdd{G}{\sigma_n} + \pdd{G}{\dot{\sigma}_n}\pdd{G}{\Psi} - \frac{\mathrm{d}}{\mathrm{d}t}\left(\pdd{G}{\dot{\sigma}_n}\right)\right] \Psi^{\dagger}}{\delta \sigma_n}} \\
  &- \ipf{\Psi^{\dagger}}{\pdd{G}{p} \delta p} -\ipf{V^{\dagger}_J}{\pdd{F_J}{p} \delta p }\\
  &+ \ip{\Psi^{\dagger}_0}{\delta \Psi_0}_{\Gamma} .
  \end{aligned}
\end{equation}
The first two terms vanish due to the adjoint friction law \eqref{eq:friction_law_adj} and state evolution equation \eqref{eq:state_evolution_adj}, \changed{while the third term vanishes due to the adjoint opening equation \eqref{eq:adjoint_opening}}. We have arrived at
\begin{equation}\label{eq:gradient_var}
  \delta \L = - \ipf{\Psi^{\dagger}}{\pdd{G}{p} \delta p} -\ipf{V^{\dagger}_J}{\pdd{F_J}{p} \delta p } + \ip{\Psi^{\dagger}_0}{\delta \Psi_0}_{\Gamma}.
\end{equation}
Since $\ddd{\F}{p} = \ddd{\L}{p}$ and $\ddd{}{p}\ipf{v}{w\delta p} = \ip{v}{w}_\T$, \eqref{eq:gradient_adj} follows. This proves Theorem \ref{thm:cont_grad}.

\begin{remark}
  Note that in the Lagrangian cost functional \eqref{eq:lagrangian} we have taken the inner product with adjoint particle velocity $\dot{\bru}^\dagger$, rather than particle displacement $\bru^\dagger$. This choice significantly simplifies the adjoint friction and state evolution equations. If one instead were to use adjoint displacement, the equations would, in addition to the current terms, depend also on forward particle accelerations, adjoint slip and adjoint state rate.
\end{remark}
\changed{
\begin{remark}
 Interestingly, the adjoint rate-and-state equations and the adjoint-based gradient do not depend on the adjoint normal stress $\sigma_n^\dagger$. This is due to the no opening condition \eqref{eq:no_opening}. If the fault in the forward problem instead satisfied an opening condition $V_n = H$, with $H$ a function of a forward variable and/or a frictional parameter, then $\sigma_n^\dagger$ would appear in the adjoint rate-and-state equations and/or the gradient.
\end{remark}
}

\subsection{Misfits and adjoint sources}\label{sec:misfits}
This section details how misfits of particle displacement and velocity enter the adjoint problem \eqref{eq:elastic_3d_adj}, \ie, how the different types of misfits translate to the adjoint source term $\bar{Q}^\dagger$.

Consider the adjusted residuals $\bar{R}^{(k)}$ in \eqref{eq:adjusted_residual}. We will refer to the linear operators $W^{(k)}$ as filters, although they could equally well be windowing operators. It is understood that the $W^{(k)}$ may be different for each component of $\bar{r}^{(k)}$, but to simplify the presentation we avoid explicit notation for this. Furthermore, we first consider a single receiver \changed{at $\brx = \brx_r$} and will temporarily drop the \changed{$(k)$} superscript. Since $W$ is linear, the first variation of $R_J$ can be expressed as
\begin{align}
  \delta R_J &= W[\delta r_J] = \changed{ \int \limits_{\Omega} W[\delta m_J] \hat{\delta}_{\brx_r} \, \mathrm{d} \brx },
\end{align}
where $\changed{\hat{\delta}_{\brx_r}(\brx)} := \hat{\delta}(\brx - \changed{\brx_r)}$ is the shifted Dirac delta function.
Let $W^{\dagger}$ denote the Hilbert adjoint of $W$ with respect to the $L^2$-inner product, such that
\begin{equation}
  \int \limits_{0}^{T} q(t) W[s](t) \, \mathrm{d} t = \int \limits_{0}^{T} W^\dagger[q](t) s(t) \, \mathrm{d} t \quad \forall q, s \in L^2(\T).
\end{equation}
Note that windowing operators of the form $W[r](t) = w(t)r(t)$, for some function $w$, are self-adjoint ($W^\dagger = W$), while a general filter might not be self-adjoint.

Now consider the first variation $\delta F$ of the misfit functional \eqref{eq:misfit} with a single receiver. Carrying out the chain rule for the first variation yields
\begin{equation}\label{eq:adjoint_source_rel_derivation}
  \begin{aligned}
    \delta \F &= \changed{\int \limits_{0}^{T} R_J\delta R_J \, \mathrm{d}t} = \int \limits_{0}^{T} \int \limits_{\Omega} R_J W[\delta m_J] \changed{\hat{\delta}_{\brx_r}} \, \mathrm{d} \brx \, \mathrm{d}t \\
    &= \int \limits_{0}^{T} \int \limits_{\Omega} W^{\dagger} [ R_J ]\changed{\hat{\delta}_{\brx_r}} \delta m_J \, \mathrm{d} \brx \, \mathrm{d}t = \ipo{W^{\dagger} [ R_J ]\changed{\hat{\delta}_{\brx_r}} }{\delta m_J} .
  \end{aligned}
\end{equation}
For a velocity misfit where $\bar{m} = \dot{\bru}$, we note that $Q_J^\dagger = W^{\dagger} [ R_J ]\changed{\hat{\delta}_{\brx_r}} $ satisfies \eqref{eq:adjoint_source_rel}, which was used to derive the adjoint-based gradient. In the case of a displacement misfit, \ie, $\bar{m} = \bru$, introduce the time-integrated quantity
\begin{equation}\label{eq:adjoint_source_displacement_integral}
  \hat{R}_J(t) = \int \limits_0^t W^{\dagger}[R_J](t') \, \mathrm{d}t' + \hat{R}_{J,0},
\end{equation}
where the constant $\hat{R}_{J,0}$ is selected such that $\hat{R}_J(T) = 0$. By \eqref{eq:adjoint_source_displacement_integral}, $\hat{R}_J$ satisfies the ODE 
\begin{equation}\label{eq:displacement_misfit_ode}
    \dot{\hat{R}}_J = W^{\dagger}[R_J].    
\end{equation}
Using integration by parts in time, it therefore follows that
\begin{equation}
  \delta \F = \ipo{W^{\dagger} [ R_J ]\changed{\hat{\delta}_{\brx_r}} }{\delta u_J} = -\ipo{\hat{R}_J \changed{\hat{\delta}_{\brx_r}} }{\delta \dot{u}_J}, 
\end{equation}
where the term at $t=0$ vanishes because of the initial condition $u_J = 0$ (or in general the $p$-independence of the initial data) and the term at $t=T$ vanishes because $\hat{R}_J(T) = 0$. Considering multiple receivers again, the adjoint source function at receiver $k$ is thus
\begin{equation}
  S_J^{(k)}(t) = 
  \begin{cases}
    -\hat{R}_J^{(k)}(t)& m_J = u_J \quad \text{(displacement)}, \\
    W^{(k)\dagger} [ R^{(k)}_J](t)& m_J = \dot{u}_J \quad \text{(velocity)}.
  \end{cases}
\end{equation}
Note that both types of signals involve the application of an adjoint filter operator. The complete adjoint source term is obtained by summing over all receivers,
\begin{equation}
  Q_J^\dagger(\brx, t) = \sum_{k=1}^{N_{rec}} S_J^{(k)}(t) \changed{\hat{\delta}(\brx - \brx_r^{(k)}) }.
\end{equation}
In this work, the misfit functionals are considered to measure either displacement or velocity exclusively. It is straightforward to extend the derivation to a linear combination of displacement and velocity misfits. We briefly comment on extending the misfit to other types of measurements. Measurements to consider include strain rate from fiber optics cables, pressure from hydrophones, or snapshot measurements of the displacement of Earth's surface through InSAR. If the resulting misfit can be cast in the form of \eqref{eq:adjoint_source_rel}, the expression for the adjoint-based gradient in Theorem \ref{thm:cont_grad} follows directly. The difference will only be in the adjoint source term. In the case of an InSAR measurement, we anticipate that the adjoint source term consists of a surface forcing in space, multiplied by a Heaviside function in time activating at the snapshot time.

\subsection{Well-posedness} \label{sec:wellposed}
\reviewerOne{For Theorem \ref{thm:cont_grad} to be valid, both the forward and adjoint problems need to be well-posed. A problem is well-posed if a solution a) exists, b) is unique, and c) depends continuously on boundary and initial data (also referred to as stability) \cite{gustafsson_kreiss_oliger}. Well-posedness results for linear elastodynamics with rate-and-state friction are incomplete, and we provide a brief review of known results below. For the adjoint problem, well-posedness is tied to that of the linearized forward problem, \ie, the adjoint problem is well-posed if and only if the linearized forward problem is \cite{buithanh2023adjoint}.

For linear initial-boundary value problems, conditions for well-posedness may be established using the energy method, in which an energy estimate for the problem is derived. By prescribing appropriate boundary and interface conditions, the energy estimate shows that the solution is bounded in terms of data, ensuring stability \cite{gustafsson_kreiss_oliger}. Uniqueness can then be established by applying the the energy method to the difference of two solutions. Existence further requires that a minimal number of initial, boundary, and interface conditions are specified \cite{nordstrom_2013}. For the forward problem \eqref{eq:elastic_3d_fwd}, where non-linearity is limited to the friction law, the energy method leads to conditions on the friction law, required for stability and uniqueness, as discussed below.

Energy estimates for problems with general friction laws were established decades ago \cite{kostrov_1974,rudnicki_freund_1981}. Dissipation (rather than creation) of mechanical energy during slip is guaranteed when the frictional shear traction is coplanar with and opposite to slip velocity and the normal stress remains compressive. In \cite{kozdon_et_al_2012}, the energy method was applied to 2D antiplane shear problems with purely velocity-dependent nonlinear friction and no prestress. For scalar friction laws of the form $\tau = F(V)$, the stability condition of frictional energy dissipation reads $F(V)V \ge 0$. Additionally, the authors showed that a sufficient condition for uniqueness is $\frac{\mathrm{d}F}{\mathrm{d}V} \ge 0$. In \cite{kozdon_et_al_2010}, the results were extended to full 3D elastodynamics assuming that fault normal stress remains compressive. Here, the stability condition reads $F_JV_J \ge 0$ and the sufficient uniqueness condition is that the Jacobian of the friction law is symmetric-positive definite, \ie, $\pdd{F_I}{V_J}=\pdd{F_J}{V_I}$, and $V_I\pdd{F_J}{V_I}V_J \ge 0$. Moreover, existence was shown using the method of characteristics, thereby proving well-posedness for this particular class of problems with purely velocity-dependent friction and no prestress. 

In \cite{duru_et_al_2019}, the analysis of 2D antiplane shear problems was extended to rate-and-state friction with prestress. In addition to the previously mentioned constraints on $F$, it was further established that well-posedness requires that the state evolution law $G(V,\Psi)$ is such that $\pdd{G}{\Psi}$ is bounded and such evolution laws were termed admissible\footnote{It suffices that $\pdd{G}{\Psi}$ is bounded almost everywhere, \ie, $G(V,\Psi)$ is Lipschitz continuous in $\Psi$.}. Well-posedness results for general 3D elastodynamics with rate-and-state friction are to the best of our knowledge currently lacking.
 
There have been many studies of linearized rate-and-state friction, as summarized by \cite{rice_lapusta_2001}. With a few exceptions \cite{ray2017,viesca2023}, these studies have been limited to linearization about a state of steady sliding on a planar fault with constant frictional parameters (making this a constant coefficient problem). Sliding occurs between between two homogeneous elastic half-spaces though some studies have considered other (e.g., layered) geometries \cite{ranjith2009,ranjith2014,aldam2016}. Of particular note are so-called bimaterial problems, where sliding occurs between dissimilar elastic solids. Spatially nonuniform sliding alters the fault normal stress and therefore frictional strength. This feedback can destabilize steady sliding. Compromises to the full rate-and-state formulation, such as removing the direct effect so that the friction coefficient satisfies $\pdd{f}{|V|}=0$, are known to render the problem ill-posed \cite{ranjith2001,rice_lapusta_2001}.

To summarize, the adjoint problem \eqref{eq:elastic_3d_adj} is well-posed if and only if the forward problem \eqref{eq:elastic_3d_fwd} with linearized rate-and-state friction is well-posed. The latter is determined by the stability of the linearized rate-and-state friction laws, for which well-posedness results are currently incomplete. Studying the stability of linearized rate-and-state friction is thus of interest also for adjoint-based optimization. Finally, we note that the 2D antiplane shear} numerical studies presented herein are performed using the typical rate-and-state friction coefficient \eqref{eq:friction_coeff} together with the slip law \eqref{eq:G_antiplane_shear}. No blow-ups of the adjoint solution have been observed, and the results therefore indicate that the adjoint equations are well-posed, at least in this setting.

\section{Dynamic rupture in antiplane shear}\label{sec:antiplane_shear}
We now proceed to demonstrate the method applied to dynamic rupture simulations in \reviewerTwo{2D} antiplane shear. Consider the domain $\Omega$ with a frictional fault along $\Gamma$, as illustrated in Figure \ref{fig:fractal_fault_domain}. In the semi-discrete setting, the discretizations will be carried out blockwise in $\Omega_\pm$. For this reason, quantities on the respective domain will be denoted with the subscripts $+$ and $-$, or compactly when the equations apply to both domains individually by $\pm$. \changed{Note that this use of $+$ and $-$ differs slightly from the continuous setting in Sections \ref{sec:elasticity} and \ref{sec:adjoint_eqs} where it was used only when referring to quantities on the fault; here it is used for both the entire volumes including boundaries and the fault. The notation should be clear from context}. Letting $(x_1, x_2)=(x,y) \in \Omega$ be the in-plane coordinates, \reviewerTwo{the only nonzero component of particle displacement, $u = u_z$, is in the antiplane direction $x_3=z$, with all fields being functions only of $x$ and $y$. Slip, slip velocity, and shear traction are scalars, and the fault normal stress is unaltered by slip. In this setting, the governing equations \eqref{eq:elastic_3d_fwd} reduce to}
\begin{equation}\label{eq:anti_plane_shear_fwd}
  \begin{array}{lll}
    \rho_\pm  \ddot{u}_\pm = \dI \mu_\pm \dI u_\pm, & \brx\in \Omega_\pm, & t\in \T,\\
    u_\pm= u_{0\pm}, \quad \dot{u}_\pm= v_{0\pm}, & \brx \in \Omega_\pm, & t = 0, \\
    \tau_\pm + Z\dot{u}_\pm = 0 , & \brx \in \partial\Omega_\pm \setminus \Gamma, & t\in \T, \\
    \tau_\pm = \mp F(V,\Psi), & \brx \in \Gamma, & t\in \T, \\
    \dot{\Psi} = G(V,\Psi), & \brx \in \Gamma, & t\in \T, \\
    \Psi = \Psi_0, & \brx \in \Gamma, & t = 0.
  \end{array}
\end{equation}
\changed{Here, we use $\tau_\pm = \mu_\pm \partial_{n_\pm} u_\pm$ to denote shear traction on the boundaries of $\Omega_\pm$ as well as the two sides of the fault, while in Section \ref{sec:elasticity} $\brtau$ was only defined for the $-$ side of the fault. The equation $\tau_\pm = \mp F(V,\Psi)$ therefore specifies both the friction law and force balance on the fault.} On the exterior boundaries $\partial\Omega_\pm \setminus \Gamma$ characteristic \changed{non-reflecting} conditions \cite{engquist_majda_77} are imposed. On the fault, we consider a typical rate-and-state friction coefficient \cite{rice_lapusta_2001} 
\begin{equation}\label{eq:friction_coeff}
  f(|V|,\Psi) = a \sinh^{-1}\left(\frac{|V|}{2V_0}e^{\frac{\Psi}{a}}\right),
\end{equation}
where $a$ is the dimensionless direct effect parameter and $V_0$ an arbitrarily chosen reference velocity. \reviewerTwo{The direct effect parameter is positive ($a>0$), which ensures that the instantaneous response of the fault to perturbations is of velocity-strengthening character ($\partial F / \partial V > 0$). This is well supported by experiments and also expected on a theoretical basis \cite{rice_lapusta_2001}. Furthermore, some studies referenced in Section \ref{sec:wellposed} suggest it to be a requirement for well-posedness.} In addition, an external loading $\tau_L(\brx)$, $\brx \in \Gamma$ is added to the fault \changed{to initiate the rupture at a pre-specified location}, such that the friction law reads
\begin{equation}\label{eq:F_antiplane_shear}
    F(V,\Psi) = \sigma_n^{\changed{0}} f(|V|,\Psi)\frac{V}{|V|} - \changed{\tau^0} - \tau_L.
\end{equation}
The external loading could of course be included directly in $\changed{\tau^0}$, but is kept separate here since we later will perform inversions for $\changed{\tau^0}$ with a known $\tau_L$. State evolution is governed by the slip law \cite{marone_1998,lapusta_et_al_2000}
\begin{equation}\label{eq:G_antiplane_shear}
    G(V, \Psi) = -\frac{|V|}{D_c}\left(f(|V|,\Psi) - f_{ss}(|V|)\right),
\end{equation}
where $f_{ss}(|V|) =  f_0 + (a-b)\ln \left(|V| / V_0 \right)$ is the steady state friction coefficient \cite{rice_lapusta_2001}. Here $b$ is the state evolution parameter, $D_c$ the state evolution distance and $f_0$ is the reference coefficient for steady sliding at velocity $V_0$. \reviewerTwo{The fault is velocity-strengthening where $a-b > 0$ and velocity-weakening where $a-b<0$; the latter is required for unstable slip and sustained rupture propagation \cite{marone_1998}.}

\reviewerOne{In the following sections, we introduce a dual consistent \cite{hicken_zingg_2014} space-time discretization based on SBP difference operators  combined with weakly enforced boundary and interface conditions through characteristics-based SAT and explicit Runge--Kutta time integration. Based on this discretization we present a discrete counterpart to Theorem \ref{thm:cont_grad} for dynamic rupture in antiplane shear. While the discretization of the forward problem was developed previously in \cite{duru_et_al_2019,erickson_et_al_2022,harvey_et_al_2023,stiernstrom_et_al_2023}, its space-time dual consistency is to the best of our knowledge a novel result.}
\begin{figure}[h!]
  \centering
  \includegraphics[width=0.5\linewidth]{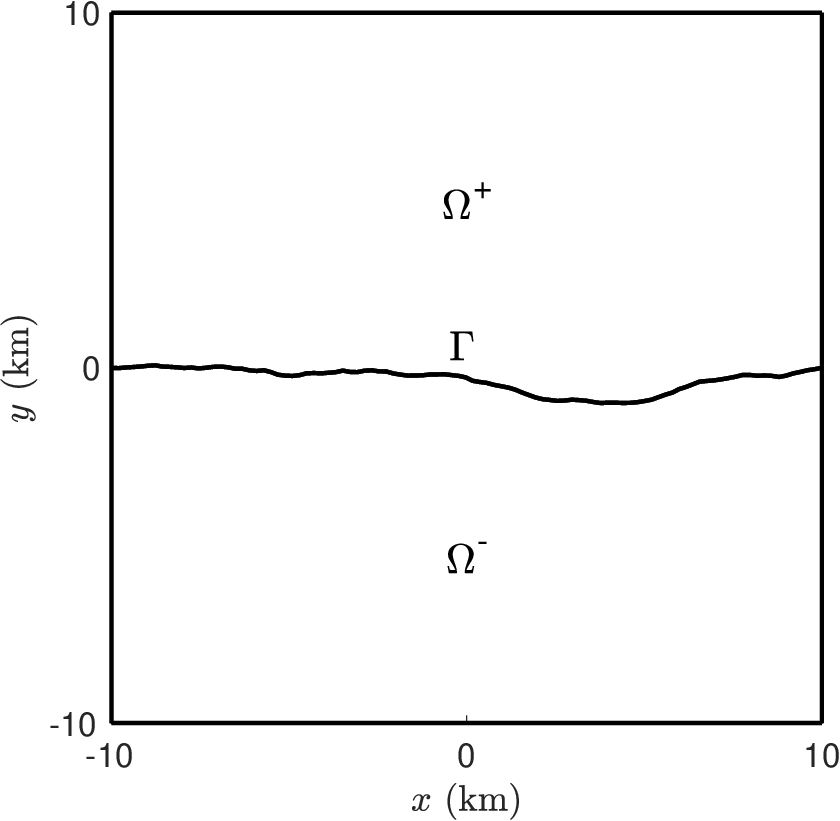}
  \caption{Domain with \changed{the} fault \changed{$\Gamma$} separating $\Omega^+$ and $\Omega^-$}\label{fig:fractal_fault_domain}
\end{figure}

\subsection{SBP finite difference discretization}\label{sec:antiplane_shear_sbp}
Discretizations of \eqref{eq:anti_plane_shear_fwd} using SBP finite differences, with weak enforcement of fault and boundary conditions through the simultaneous-approximation-term (SAT) method have been presented in \eg \cite{erickson_dunham_2014,duru_et_al_2019,erickson_et_al_2022,harvey_et_al_2023}. Here, we utilize the boundary-optimized SBP difference operators presented in \cite{mattsson_et_al_2018,stiernstrom_et_al_2023} to discretize the scalar wave operator on the physical grid $\bOmega$. The physical grid is obtained through transfinite interpolation of the coordinate mapping between $\Omega$ and a two-block Cartesian reference domain. See \eg \cite{almquist_dunham_2020,almquist_dunham_2021, stiernstrom_et_al_2023} for further details on the finite difference discretization of the scalar wave operator and the discretization of $\bOmega$.

Let $\bD_{II}(\bmu_\pm) \approx \dI \mu_\pm \dI$ be \changed{the} SBP finite difference approximation of the variable coefficient Laplace operator on the physical domains $\bOmega_\pm$. Then, the discretization of the scalar wave equation in \eqref{eq:anti_plane_shear_fwd} is given by
\begin{equation}\label{eq:anti_plane_shear_discr}
  \begin{array}{ll}
    \dbar{\brho}_\pm  \ddot{\bu}_\pm = \bD_{II}(\bmu_\pm) \bu_\pm + \text{SAT}_\pm, & t\in \T,\\
    \bu_\pm=\bu_{0\pm}, \quad \dot{\bu}_\pm = \bv_{0\pm}, & t = 0, \\
  \end{array}
\end{equation}
where $\text{SAT}_\pm$ weakly enforces boundary and interface conditions. Imposing the non-linear rate-and-state friction law through a standard traction SAT (\eg as in \cite{almquist_dunham_2020}) may lead to significant stiffness in dynamic rupture problems. To resolve the issue, we use the non-stiff SAT presented in \cite{erickson_et_al_2022,harvey_et_al_2023}, enforcing boundary and interface conditions through characteristic variables. \reviewerOne{Since the characteristics-based non-stiff SAT is a recent development for second-order wave equations it is described in detail below, for the sake of completeness.}

Consider first a single block $\Omega$ (either $\Omega_+$ or $\Omega_-$). On a boundary segment $\bm{\zeta}$ (either an exterior boundary or the fault interface), the non-stiff SAT is given by
\begin{equation}\label{eq:anti_plane_shear_char_SAT}
  \text{SAT}(\bu, \bu^*, \btau^*) = \bH^{-1}_{\Omega}\left(\be \bH_{\zeta}(\btau^* - \btau) - \bT^T \bH_{\zeta}(\bu^* - \bu)\right),
\end{equation}
where $\bu$, $\bu^*$, and $\btau^*$ all are evaluated on $\bm{\zeta}$. Here, $\bH_{\Omega}$ and $\bH_{\Gamma}$ are the SBP quadratures on $\bOmega$ and $\bm{\zeta}$, while \reviewerOne{$\be^T$} is the boundary restriction operator. Moreover, $\bT$ is the boundary traction operator such that $\btau := \bT \bu \approx \mu n_I\partial_I u (\brx)$, $\brx \in \zeta$, while $\bu^*$ and $\btau^*$ are target values (or fluxes) for displacement and traction. The target values are \changed{required} to preserve the outgoing grid-based characteristic variables,
\begin{equation}\label{eq:preserve_outgoing}
  \bw^* := \bZ\dot{\bu}^* - \btau^* = \bZ\dot{\bu} - \tilde{\btau} \changed{=:} \bw,\\
\end{equation}
where $\bZ = \sqrt{\brho\bmu}$ is the shear impedance and $\tilde{\btau}$ is the modified traction
\begin{equation}\label{eq:modified_tau}
   \tilde{\btau} = \btau + \gamma(\bu^* - \bu).
\end{equation}
The parameter $\gamma$ is the semi-bounded Dirichlet penalty parameter of \cite{almquist_dunham_2020}. Note that the units of $\gamma$ are such that $\gamma(\bu^* - \bu)$ has the units of traction. Moreover, $\bu^*$ and $\btau^*$ should satisfy the boundary or fault conditions. 

On an exterior boundary, the boundary conditions are conveniently expressed by introducing a reflection coefficient $\bR$, such that we may write the ingoing characteristic as $\bZ\dot{\bu}^* + \btau^* = \bR\bw^*$. This together with \eqref{eq:preserve_outgoing} results in
\begin{align}
  \btau^* &= \frac{\bR-\bI}{2}(\bZ\dot{\bu} - \tilde{\btau}), \label{eq:tau_star_bc}\\
  \dot{\bu}^* &= \frac{\bR+\bI}{2}(\dot{\bu} - \bZ^{-1}\tilde{\btau}). \label{eq:ode_ustar_bc}
\end{align}
Here $\bI$ is an identity matrix with dimensions matching the number of boundary points. Then, \changed{rigid-wall} (Dirichlet),  \reviewerOne{traction-free} (Neumann) or characteristic \changed{non-reflecting} boundary conditions are imposed by setting
\begin{equation}
  \begin{array}{ll}
    \bR = -\bI &\text{\changed{(rigid-wall)}},\\
    \bR = \bI &\text{\reviewerOne{(traction-free)}},\\
    \bR = \mathbf{0} &\text{(\changed{non-reflecting})},\\
  \end{array}
\end{equation}
along the corresponding boundary segment. On the outer boundaries, characteristic \changed{non-reflecting} conditions are specified, and we therefore set $\bR = 0$.

On the fault, the target tractions instead satisfy the fault condition
\begin{equation}\label{eq:tau_star_fault}
   \btau_\pm^* = \mp \bF(\bV^*,\bPsi),
\end{equation}
where $\bF$ is considered to apply point-wise. From \eqref{eq:preserve_outgoing} and \eqref{eq:tau_star_fault} the target slip velocity $\bV^*$ can be shown to satisfy
\begin{equation}\label{eq:V_star_nonlin}
  \bm{\kappa}\bV^* + \bF(\bV^*,\bPsi) = -\btau_\ell,
\end{equation}
where $\bm{\kappa} = \frac{\bZ_+\bZ_-}{\bZ_+ + \bZ_-}$ and $\btau_\ell = \frac{\bZ_+\bw_- - \bZ_-\bw_+}{\bZ_+ + \bZ_-}$. For non-linear $F$, solving \eqref{eq:V_star_nonlin} requires employing a root-finding algorithm. Note that $\bPsi$ is held fixed while solving \eqref{eq:V_star_nonlin}. Let $\tilde{\bV}^*(\btau_\ell)$ denote the solution to
\begin{equation}
  \bF(\tilde{\bV}^*,\bPsi) = -\btau_\ell,
\end{equation}
which can be solved analytically with our friction coefficient \eqref{eq:friction_coeff}. As noted in \cite{erickson_et_al_2022} one can show that $\bV^* \in [\changed{\tilde{\bV}^*}(\btau_\ell), 0]$ for \changed{$\btau_\ell \geq \btau^0 + \btau_L$} and $\bV^* \in [0, \changed{\tilde{\bV}^*}(\btau_\ell)]$ for \changed{$\btau_\ell < \btau^0 + \btau_L$}. Since $F_V > 0$ for admissible friction laws, there exists a unique solution $\bV^*$. 

From \eqref{eq:preserve_outgoing} it directly follows that $\dot{\bu}^*_\pm$ on the fault is given by
\begin{equation} \label{eq:ode_ustar_fault}
  \dot{\bu}^*_\pm = \dot{\bu}_\pm - \bZ_\pm^{-1}(\tilde{\btau}_\pm-\btau_\pm^*).
\end{equation}
Note that the characteristic boundary and fault conditions require the target displacements $\bu^*_\pm$ to be regarded as additional unknown grid functions, evolved together with $\bu_\pm$ and $\bPsi$ through \eqref{eq:ode_ustar_bc} and \eqref{eq:ode_ustar_fault}. For notational purposes we introduce
\begin{equation}\label{eq:ustar_tot}
  \dot{\bu}^*_\pm = \bL^*(\dot{\bu}_\pm,\tilde{\btau}_\pm) =
  \begin{cases}
    \frac{\bR+\bI}{2}(\dot{\bu}_\pm - \bZ_\pm^{-1}\tilde{\btau}_\pm)& \text{(exterior boundaries)}, \\
    \dot{\bu}_\pm - \bZ_\pm^{-1}(\tilde{\btau}_\pm-\btau_\pm^*)& \text{(fault)},
  \end{cases}
\end{equation}
with $\btau_\pm^*$ given by \eqref{eq:tau_star_fault}. Similarly, the linear operator for the right-hand-side of \eqref{eq:anti_plane_shear_discr} is denoted
\begin{equation}\label{eq:sbp_sat_rhs}
  \bL(\bu_\pm,\bu^*_\pm,\btau^*_\pm) := \bD_{II}(\bmu_\pm) \bu_\pm + \sum_{\bm{\zeta}}\text{SAT}(\bu_\pm,\bu^*_{\pm},\btau^*_{\pm}),
\end{equation}

The semi-discretization of the state evolution equation \eqref{eq:G_antiplane_shear} is
\begin{equation}\label{eq:state_evolution_discr}
    \begin{array}{ll}
    \dot{\bPsi} = \bG(\bV^*,\bPsi),& t \in \T,\\
    \bPsi = \bPsi_0,& t = 0,\\
    \end{array}
\end{equation}
where $\bG$ is evaluated pointwise. Note that $\bV^*$ is used in evolving $\bPsi$. This will be of importance when deriving the gradient to the discrete optimization problem in Section \ref{sec:antiplane_shear_opt}.

Finally, the SBP-SAT semi-discretization of \eqref{eq:anti_plane_shear_fwd} is given by
\begin{equation}\label{eq:anti_plane_shear_discr_fwd}
  \begin{array}{ll}
    \dbar{\brho}_\pm  \ddot{\bu}_\pm = \bL(\bu_\pm,\bu^*_\pm,\btau^*_\pm), & t\in \T,\\
    \dot{\bu}^*_\pm = \bL^*(\dot{\bu}_\pm,\tilde{\btau}_{\pm}) & t\in \T,\\
    \bu_\pm=\bu_{0\pm}, \quad \dot{\bu}_\pm = \bv_{0\pm}, & t = 0, \\
    \bu^*_{\pm} = \be^T_{\partial\Omega_\pm}\bu_{0\pm}, & t = 0, \\
    \dot{\bPsi} = \bG(\bV^*,\bPsi),& t \in \T,\\
    \bPsi = \bPsi_0,& t = 0.
  \end{array}
\end{equation}
In \cite{duru_et_al_2019,erickson_et_al_2022} it is shown (by using the SBP properties in Appendix \ref{app:sbp}) that the scheme is energy stable and that the semi-discrete energy rate is a consistent approximation of the continuous energy rate.

With the forward problem discretized in space, the semi-discretization of the optimization problem \eqref{eq:inverse_prob_3d} reads
\begin{equation}\label{eq:anti_plane_shear_semi_discr_opt}
  \min_\bp \bm{\F} \quad \text{subject to \eqref{eq:anti_plane_shear_discr_fwd}},
\end{equation}
where the semi-discrete misfit is given by
\begin{equation}\label{eq:anti_plane_shear_semi_discr_misfit}
  \bm{\F} =  \frac{1}{2}\sum_{k=1}^{N_{rec}}\int_0^T |\mathbf{r}^{(k)}(t)|^2\mathrm{d}t,
\end{equation}
and the semi-discrete residuals are
\begin{equation}\label{eq:anti_plane_shear_semi_discr_residual}
  \mathbf{r}^{(k)}(t) = \left(\ip{\changed{\hat{\bm{\delta}}_{\brx_r^{(k)}}}}{\mathbf{m}(t)}_\bOmega - m_{data}(t)\right),
\end{equation}
for a measured field $\mathbf{m}$. We refrain from including filtering operators $W$ in the semi-discrete presentation since no filtering is used in the numerical experiments performed in Section \ref{sec:num_studies}. The discrete Dirac delta $\changed{\hat{\bm{\delta}}_{\brx_r^{(k)}}} \approx \hat{\delta}(\brx - \brx^{(k)}_r)$ presented in \cite{petersson_et_al_2016} is used to restrict the measured field to the receiver location \changed{$\brx_r^{(k)}$}. The number of moment conditions used when constructing \changed{$\hat{\bm{\delta}}$} is the same as the order of the SBP difference operators. No smoothness conditions are used.

\subsection{Adjoint scheme and semi-discrete gradient}\label{sec:antiplane_shear_opt}
Analogously to the forward problem, the adjoint equations \eqref{eq:elastic_3d_adj} reduce to
\begin{equation}\label{eq:anti_plane_shear_adj}
  \begin{array}{lll}
    \rho_\pm  \ddot{u}^\dagger_\pm = \dI\mu_\pm \dI u^\dagger_\pm + Q^\dagger_\pm, & \brx\in \Omega_\pm, & t\in \T,\\
    u^\dagger_\pm =0, \quad \dot{u}^\dagger_\pm= 0, & \brx \in \Omega_\pm, & t = T, \\
    \tau^\dagger_\pm - Z\dot{u}^\dagger_\pm= 0 , & \brx \in \partial\Omega_\pm \setminus \Gamma, & t\in \T, \\
    \tau^\dagger_\pm = \mp F^\dagger(V^\dagger,\Psi^\dagger), & \brx \in \Gamma, & t\in \T, \\
    -\dot{\Psi}^\dagger = G^\dagger(V^\dagger,\Psi^\dagger), & \brx \in \Gamma, & t\in \T, \\
    \Psi^\dagger = 0, & \brx \in \Gamma, & t = T.
  \end{array}
\end{equation}
In the semi-discrete setting, we start by defining
\begin{equation}\label{eq:preserve_outgoing_adj}
  \bw^{\dagger*} := -\bZ\dot{\bu}^{\dagger*} - \btau^{\dagger*} = -\bZ\dot{\bu}^\dagger - \tilde{\btau}^\dagger \changed{=:} \bw^\dagger,
\end{equation}
which are the outgoing characteristic adjoint variables on a boundary segment after time reversal, where $\tilde{\btau}^\dagger$ is defined analogously to \eqref{eq:modified_tau}. On an exterior boundary, $\btau^{\dagger*}$ and $\dot{\bu}^{\dagger*}$ then satisfies
\begin{align}
  \btau^{\dagger*} &= \frac{\bR-\bI}{2}(-\bZ\dot{\bu}^\dagger - \tilde{\btau}^\dagger), \label{eq:tau_star_bc_adj}\\
  -\dot{\bu}^{\dagger*} &= \frac{\bR+\bI}{2}(-\dot{\bu}^\dagger - \bZ^{-1}\tilde{\btau}^\dagger). \label{eq:ode_ustar_bc_adj}
\end{align}
Similarly, on the fault
\begin{align}
  \btau_\pm^{\dagger*} &= \mp \bF^\dagger(\bV^{\dagger*}, \bPsi^\dagger) = \mp (\dbar{\bF}_V\bV^{\dagger*} + \dbar{\bG}_V\bPsi^\dagger), \label{eq:tau_star_fault_adj}\\
  -\dot{\bu}^{\dagger*}_\pm &=  -\dot{\bu}^\dagger_\pm - \bZ_\pm^{-1}(\tilde{\btau}^\dagger_\pm-\btau_\pm^{\dagger*}), \label{eq:ode_ustar_fault_adj}
\end{align}
where $\bF_V(\bV^*,\bPsi) \approx \frac{\partial F}{\partial V}$ and $\bG_V(\bV^*,\bPsi) \approx \frac{\partial G}{\partial V}$ are evaluated point-wise. The adjoint target slip velocity now satisfies
\begin{equation}\label{eq:V_star_nonlin_adj}
  \bm{\kappa}\bV^{\dagger*} + \bF^\dagger(\bV^{\dagger*},\bPsi^\dagger) = -\btau_\ell^\dagger = -\frac{\bZ_+\bw_-^\dagger - \bZ_-\bw_+^\dagger}{\bZ_+ + \bZ_-}.
\end{equation}
Note that $F^\dagger$ is linear in $V^\dagger$, such that \eqref{eq:V_star_nonlin_adj} reduces to a time-dependent linear equation in $\bV^{\dagger*}$. 
The semi-discrete adjoint state evolution equation is simply
\begin{equation}\label{eq:state_evolution_discr_adj}
  \begin{array}{ll}
  -\dot{\bPsi}^\dagger = \bG^\dagger(\bV^{\dagger*},\bPsi^\dagger) = (\dbar{\bG}_\Psi\bPsi^\dagger + \dbar{\bF}_\Psi\bV^{\dagger*}),& t \in \T,\\
  \bPsi^\dagger = 0,& t = T.\\
  \end{array}
\end{equation}
Again, $\bG_\Psi(\bV^*,\bPsi) \approx \frac{\partial G}{\partial \Psi}$ and $\bF_\Psi(\bV^*,\bPsi) \approx \frac{\partial F}{\partial \Psi}$ are evaluated point-wise.

Adopting the notation for the operators $\bL$ and $\bL^*$ in \eqref{eq:sbp_sat_rhs} and \eqref{eq:ustar_tot}, the adjoint equations are then discretized according to \eqref{eq:anti_plane_shear_discr_fwd}, resulting in
\begin{equation}\label{eq:anti_plane_shear_discr_adj}
  \begin{array}{ll}
    \dbar{\brho}_\pm  \ddot{\bu}_\pm^\dagger = \bL(\bu^\dagger_\pm,\bu^{\dagger*}_\pm,\btau^{\dagger*}_\pm) + \bQ^\dagger_\pm, & t\in \T,\\
    -\dot{\bu}^{\dagger*}_{\pm} = \bL^*(-\dot{\bu}^\dagger_\pm,\tilde{\btau}^\dagger_{\pm}) & t\in \T,\\
    \bu^\dagger_\pm = 0, \quad \dot{\bu}^\dagger_\pm = 0, & t = T, \\
    \bu^{\dagger*}_{\pm} = 0, & t = T, \\
    -\dot{\bPsi}^\dagger = \bG^\dagger(\bV^{\dagger*},\bPsi^\dagger),& t \in \T,\\
    \bPsi^\dagger = 0,& t = T.
  \end{array}
\end{equation}
The adjoint residual source term is given by $\bQ^\dagger_\pm = \sum_{k =1 }^{N_{rec}} \changed{\bS^{(k)} \hat{\bm{\delta}}_{\brx_r^{(k)}}}$ where
\begin{equation}
  \bS^{(k)}(t) = \begin{cases} 
    -\hat{\mathbf{r}}^{(k)}(t),& \mathbf{m} = \bu, \\
    \mathbf{r}^{(k)}(t),& \mathbf{m} = \dot{\bu},
  \end{cases}  
\end{equation}
with $\hat{\mathbf{r}}^{(k)}(t) = \int_0^t \mathbf{r}^{\changed{(k)}}(t')\mathrm{d}t' + \hat{\mathbf{r}}^{(k)}_0$, with $\hat{\mathbf{r}}_0^{(k)}$ chosen such that $\hat{\mathbf{r}}^{(k)}(T) = 0$. See Appendix \ref{app:semi_discr_receiver} for a derivation of the residual source term. \changed{As mentioned in Section \ref{sec:antiplane_shear_sbp} energy stability of the forward scheme \eqref{eq:anti_plane_shear_discr_fwd} was established in \cite{duru_et_al_2019}, \cite{erickson_et_al_2022}. Since the adjoint scheme \eqref{eq:anti_plane_shear_discr_adj} only differs in the source term and the rate-and-state equations, the stability of the adjoint scheme follows directly from the stability of the continuous adjoint rate-and-state equations.
}

We may now state the gradient to \eqref{eq:anti_plane_shear_semi_discr_opt}. It is given by the following Lemma which is the semi-discrete counterpart to Theorem \ref{thm:cont_grad} for antiplane shear.
\begin{lemma}\label{lemma:anti_plane_shear_semi_disc_grad}
  Let $\bV^{\dagger*}$, $\bPsi^{\dagger}$, satisfy \eqref{eq:anti_plane_shear_discr_adj}. Further, let $\bp$ be the fault grid function of a parameter in $F$ or $G$, and let $\bPsi_0$ be the initial state. If $\bV^*$ is computed exactly from \eqref{eq:V_star_nonlin}, then the gradient of the misfit in the semi-discrete inverse problem \eqref{eq:anti_plane_shear_semi_discr_opt} is given by
  \begin{equation}\label{eq:anti_plane_shear_semi_discr_grad}
    \begin{aligned}
      \frac{\partial \bm{\F}}{\partial \bp} &= -\int_0^T\bH_{\changed{\Gamma}}\left(\dbar{\bF}_p\bV^{\dagger*} +  \dbar{\bG}_p\bPsi^{\dagger}\right)\mathrm{d}t, \\
      \frac{\partial \bm{\F}}{\partial \bPsi_0} &= \bH_{\changed{\Gamma}}\bPsi^\dagger_0,
    \end{aligned}
  \end{equation}
  where $\bF_p(\bV^*,\bPsi) \approx \frac{\partial F}{\partial p}$ and $\bG_p(\bV^*,\bPsi) \approx \frac{\partial G}{\partial p}$ are evaulated point-wise and $\bPsi^\dagger_{0} = \bPsi^\dagger(t = 0)$.
\end{lemma}
\begin{proof}
  See Appendix \ref{app:proof_grad}.
\end{proof}
Note that \eqref{eq:anti_plane_shear_semi_discr_grad} is a consistent approximation of \eqref{eq:gradient_adj}. \reviewerOne{In proving Lemma \ref{lemma:anti_plane_shear_semi_disc_grad} we further show that the forward scheme \eqref{eq:anti_plane_shear_discr_fwd} is dual consistent \cite{hicken_zingg_2011,berg_nordstrom_2012,hicken_zingg_2014,ghasemi_thesis}, \ie, that the adjoint scheme \eqref{eq:anti_plane_shear_discr_adj} is the adjoint of \eqref{eq:anti_plane_shear_discr_fwd}. This implies that the semi-discrete gradient \eqref{eq:anti_plane_shear_semi_discr_grad} is the exact (to machine precision) gradient to \eqref{eq:anti_plane_shear_semi_discr_misfit}, and that it is a consistent approximation of the continuous gradient.} The condition on the accuracy of $\bV^*$ is discussed in the remark at the end of Appendix \ref{app:proof_grad}. In summary, errors in $\bV^*$ translate to an error in the gradient.

\subsection{Time integration and the discrete gradient}\label{sec:antiplane_shear_time}
To integrate the forward and adjoint schemes \eqref{eq:anti_plane_shear_discr_fwd}, \eqref{eq:anti_plane_shear_discr_adj} in time, the standard fourth-order accurate Runge--Kutta method (RK4) will be used. To this end, introduce $\bv_\pm = \dot{\bu}_\pm$ such that \eqref{eq:anti_plane_shear_discr_fwd} may be rewritten to first-order in time as
\begin{equation}\label{eq:antiplane_shear_ode}
  \begin{array}{ll}
  \frac{\mathrm{d}{\bU}(t)}{\mathrm{d}t} = \bA(t,\bU),& t \in \T, \\
  \bU = \bU_0,& t = 0,
  \end{array}
\end{equation}
where

\begin{equation}\label{eq:antiplane_shear_first_order_system}
  \begin{array}{ll}
    \bA(t,\bU) = \begin{bmatrix}
      \bA_-(t,\bU_-) \\
      \bA_+(t,\bU_+) \\ 
      \bG(\bV^*,\bPsi)
    \end{bmatrix},
    &\bU = \begin{bmatrix}
      \bU_- \\
      \bU_+ \\
      \bPsi
    \end{bmatrix},
  \end{array}
\end{equation}

and
\begin{equation}\label{eq:antiplane_shear_first_order_system_pm}
  \begin{array}{ll}
    \bA_\pm(t,\bU_\pm) = \begin{bmatrix}
      \bv_\pm \\
      \dbar{\brho}^{-1}_\pm(\bL(\bu_\pm,\bu^*_\pm,\btau^*_\pm) + \bQ_\pm(t)) \\ 
      \bL^*(\bv_\pm,\tilde{\btau}_{\pm}) \\
    \end{bmatrix},
    &\bU_\pm = \begin{bmatrix}
      \bu_\pm \\
      \bv_\pm \\ 
      \bu^*_{\pm}
    \end{bmatrix}.
  \end{array}
\end{equation}
The ODE system \eqref{eq:antiplane_shear_ode} is then integrated in time using RK4, where in each step \eqref{eq:V_star_nonlin} is solved for $\bV^*$ using bisection. 

To integrate the adjoint scheme \eqref{eq:anti_plane_shear_discr_fwd} the reversed time $t^\dagger = T-t$ is introduced. Then, the adjoint ODE system is
\begin{equation}\label{eq:antiplane_shear_ode_adj}
  \begin{array}{ll}
    \frac{\mathrm{d}{\bU}^\dagger(t^\dagger)}{\mathrm{d}t^\dagger} = \bA^\dagger(t^\dagger,\bU^\dagger),& t^\dagger \in \T, \\
  \bU^\dagger = \bU_0^\dagger,& t^\dagger = 0.
  \end{array}
\end{equation}
Here 
\begin{equation}\label{eq:antiplane_shear_first_order_system_adj}
  \begin{array}{ll}
    \bA^\dagger(t^\dagger,\bU^\dagger) = \begin{bmatrix}
      \bA^\dagger_-(t^\dagger,\bU^\dagger_-) \\
      \bA^\dagger_+(t^\dagger,\bU^\dagger_\changed{+}) \\ 
      \bG^\dagger(\bV^{\dagger*},\bPsi^\dagger)
    \end{bmatrix},
    &\bU = \begin{bmatrix}
      \bU^\dagger_- \\
      \bU^\dagger_+ \\
      \bPsi^\dagger
    \end{bmatrix},
  \end{array}
\end{equation}
\begin{equation}\label{eq:antiplane_shear_first_order_system_pm_adj}
  \begin{array}{ll}
    \bA^\dagger_\pm(t^\dagger,\bU^\dagger_\pm) = \begin{bmatrix}
      \bv^\dagger_\pm \\
      \dbar{\brho}^{-1}_\pm(\bL(\bu^\dagger_\pm,\bu^{\dagger*}_\pm,\btau^{\dagger*}_\pm) + \bQ^\dagger_\pm(\changed{T-t^\dagger})) \\ 
      \bL^*(\bv^\dagger_\pm,\tilde{\btau}^\dagger_{\pm}) \\
    \end{bmatrix},
    &\bU^\dagger_\pm = \begin{bmatrix}
      \bu^\dagger_\pm \\
      \bv^\dagger_\pm \\ 
      \bu^{\dagger*}_{\pm} \\
    \end{bmatrix},
  \end{array}
\end{equation}
where for each \changed{$t^\dagger$} \eqref{eq:V_star_nonlin_adj} is solved for $\bV^{\dagger*}$. Note that the forward variables $\bV^*$ and $\bPsi$ entering as variable coefficients in $\bF^\dagger$ and $\bG^\dagger$, also are evaluated \changed{at time $T-t^\dagger$}.

To evaluate the integrals over time in \eqref{eq:anti_plane_shear_semi_discr_misfit} and \eqref{eq:anti_plane_shear_semi_discr_grad}, the Runge--Kutta quadrature is utilized. That is, for time interval $n$ with step size $\Delta t_n$, the Runge--Kutta quadrature weight at substage $s \in [1,4]$ is given by $\bH_{\T,4n+s} = \Delta t_n b_s$, where $b = [1/6,	1/3,	1/3,	1/6]$ \cite{butcher_num_ode}. Then for $N$ time steps, the $4N\times 4N$ diagonal matrix $\bH_\T$ is a quadrature on the temporal grid $\bm{\T}$. The fully discrete optimization problem reads
\begin{equation}\label{eq:anti_plane_shear_fully_disc_opt}
  \min_\bp \bm{\F} \quad \text{subject to \eqref{eq:antiplane_shear_ode} solved using RK4},
\end{equation}
where
\begin{equation}\label{eq:anti_plane_shear_fully_disc_misfit}
  \bm{\F} = \frac{1}{2}\sum_{k=1}^{N_{rec}}\sum_{n=0}^{N-1} \sum_{s = 1}^4 \bH_{\T,4n+s}|\widehat{\mathbf{r}}^{(k)}_{4n+s}|^2.
\end{equation}
with the discrete residual given by
\begin{equation}\label{eq:anti_plane_shear_fully_disc_residual}
  \widehat{\mathbf{r}}^{(k)}_{4n+s} = \ip{\changed{\hat{\bm{\delta}}_{\brx_r^{(k)}}}}{\widehat{\mathbf{m}}^{(k)}_{4n+s}}_\bOmega - m^{(k)}_{data}(t_{n}+c_s \Delta t_n).
\end{equation}
Here $\widehat{\mathbf{m}}^{(k)}$ is the measured field at receiver $k$ for all Runge--Kutta stages between $t = 0$ and $t = T$, $c = [0, 1/2, 1/2, 1]$ are the RK4 stage nodes, and $t_0 = 0$. Note that for each index $i$, $\widehat{\mathbf{m}}^{(k)}_i$ is a grid function. 

We now arrive at one of the major results of this work, namely the discrete counterpart to Theorem \ref{thm:cont_grad} for dynamic rupture inversions in antiplane shear.
\begin{theorem}\label{thm:disc_grad}
  Let $\widehat{\bV}^{\dagger*}$, $\widehat{\bPsi}^{\dagger}$, be the \changed{space-time} solutions (including Runge--Kutta substages) to \eqref{eq:antiplane_shear_ode_adj} solved using RK4. Further, let $\bp$ be the fault grid function of a parameter in $F$ or $G$, and let $\bPsi_0$ be the initial state. If in each time step $\bV^*$ is computed exactly from \eqref{eq:V_star_nonlin}, then the gradient of the misfit in the discrete inverse problem \eqref{eq:anti_plane_shear_fully_disc_opt} is given by
  \begin{equation}\label{eq:anti_plane_shear_fully_discr_grad}
    \begin{aligned}
      \frac{\partial \bm{\F}}{\partial \bp} &= -\sum_{n=0}^{N-1} \sum_{s = 1}^4 \bH_{\T,4n+s} \bH_\Gamma\left(\dbar{\bF}_{p,4n+s}\widehat{\bV}^{\dagger*}_{4n+s} + \dbar{\bG}_{p,4n+s}\widehat{\bPsi}^{\dagger}_{4n+s}\right), \\
      \frac{\partial \bm{\F}}{\partial \bPsi_0} &= \bH_\Gamma \widehat{\bPsi}^\dagger_0,
    \end{aligned}
  \end{equation}
  where $\bF_p(\widehat{\bV}^*,\widehat{\bPsi}) \approx \frac{\partial F}{\partial p}$, $\bG_p(\widehat{\bV}^*,\widehat{\bPsi}) \approx \frac{\partial G}{\partial p}$ are evaluated point-wise using the Runge--Kutta stage approximations.
\end{theorem}
\begin{proof}
  RK4 is self-adjoint with respect to $\bH_{\T}$ under time reversal, see \cite{sanz_serna_2016,matsuda_miyatake_2021}. Therefore, by Lemma \ref{lemma:anti_plane_shear_semi_disc_grad}, the result follows.
\end{proof}
\reviewerOne{The combination of RK4 with the SBP-SAT discretization thus results in a space-time dual-consistent discretization. Specifically, this means that \eqref{eq:anti_plane_shear_fully_discr_grad} is the exact gradient to \eqref{eq:anti_plane_shear_fully_disc_misfit} and a consistent approximation of the continuous gradient.} Again, errors in the non-linear solve of $\bV^*$ will result in an error in $\frac{\partial \bm{\F}}{\partial \bp}$, see the remark at the end of Appendix \ref{app:proof_grad}.

To summarize, when solving the optimization problem \eqref{eq:anti_plane_shear_fully_disc_opt}, the gradient of \eqref{eq:anti_plane_shear_fully_disc_misfit} is obtained through the following procedure:
\begin{enumerate}
  \item Solve \eqref{eq:antiplane_shear_ode} using RK4, storing the Runge--Kutta stage approximations $\widehat{\bV}^*$, $\widehat{\bPsi}$ and $\widehat{\mathbf{r}}$.
  \item Solve \eqref{eq:antiplane_shear_ode_adj} using RK4, with $\widehat{\bV}^*$, $\widehat{\bPsi}$ and the residuals $\widehat{\mathbf{r}}^{(k)}$ in $\bQ^\dagger$ reversed in time. Use the time steps from the forward solve in reverse. Store the Runge--Kutta stage approximations $\widehat{\bV}^{\dagger*}$, $\widehat{\bPsi}^{\dagger}$.
  \item Reverse $\widehat{\bV}^{\dagger*}$, $\widehat{\bPsi}^{\dagger}$ in time and compute the gradient according to \eqref{eq:anti_plane_shear_fully_discr_grad}.
\end{enumerate}

\subsection{Lower-dimensional parameter representation}
The misfit functional in the inverse problem \eqref{eq:anti_plane_shear_fully_disc_opt} is typically non-convex with many local minima. As a way to regularize \eqref{eq:anti_plane_shear_fully_disc_opt} by reducing the discrete solution space, a lower-dimensional parameter representation of $\bp$ may therefore be beneficial. Examples of such lower-dimensional representations are polynomials or splines \cite{samareh_2001,bader_et_al_2023}. Here we make use of a coarse grid representation of the parameters and move between the lower-dimensional parameter grid and the higher-dimensional computational grid using interpolation operators.

Denote an inversion parameter on the coarse fault grid by $\bp^C$ and let $\bI^{C2F}$ be an interpolation operator from the coarse grid to the fine computational grid. Using superscript $F$ for fault grid functions on the computational grid, $\bp^F$ is given by
\begin{equation} \label{eq:interp_p}
    \bp^F = \bI^{C2F} \bp^C.
\end{equation}
The gradient expressions in \eqref{eq:anti_plane_shear_fully_discr_grad} are derived with respect to $\bp^F$, and are of the form
\begin{equation}\label{eq:anti_plane_shear_grad_fine}
    \frac{\partial \bm{\F}}{\partial \bp^F} = \bH^F_{\Gamma} \bm{\Phi}^F,
\end{equation}
where $\bm{\Phi}^F$ is a combination of forward and adjoint fields. By the chain rule, the gradient with respect to $\bp^C$ is
\begin{equation}
    \frac{\partial \bm{\F}}{\partial \bp^C} = \frac{\partial \bm{\F}}{\partial \bp^F} \frac{\partial \bp^F}{\partial \bp^C}. 
\end{equation}
Differentiating \eqref{eq:interp_p} stated in \reviewerTwo{index} notation (where $i,j,k$ denote components in the matrices and vectors \reviewerTwo{and the summation convention applies}) leads to
\begin{equation}
    \frac{\partial \bp^F_i}{\partial \bp^C_k} = \bI_{ij}^{C2F} \delta_{jk} = \bI_{ik}^{C2F} = (\bI^{C2F})^T_{ki}.
\end{equation}
Hence,
\begin{equation}\label{eq:anti_plane_shear_grad_coarse}
    \frac{\partial \bm{\F}}{\partial \bp^C} = (\bI^{C2F})^T \bH^F_{\Gamma} \bm{\Phi}^F.
\end{equation}
If $\bI^{C2F}$ is any interpolation operator, it is not guaranteed that $(\bI^{C2F})^T$ is a consistent approximation of interpolation of fields from the fine to the coarse grid. Thus, it is not guaranteed that $\frac{\partial \bm{\F}}{\partial \bp^C}$ is consistent with the continuous gradient. However, by requiring that $\bI^{C2F}$ is constructed such that
\begin{equation}\label{eq:adjoint_interpolation}
    \bI^{F2C} := (\bH_\Gamma^C)^{-1} (\bI^{C2F})^T \bH^F_\Gamma
\end{equation}
is a consistent interpolation operator from the fine grid to the coarse, then \eqref{eq:anti_plane_shear_grad_coarse} reads
\begin{equation}\label{eq:anti_plane_shear_grad_coarse_sbp}
    \frac{\partial \bm{\F}}{\partial \bp^C} = \bH^C_{\Gamma} \bI^{F2C} \bm{\Phi}^F.
\end{equation}
This corresponds to first interpolating $\bm{\Phi}^F$ to the coarse grid and then applying the coarse fault quadrature. In \cite{almquist_et_al_2019} interpolation operators satisfying \eqref{eq:adjoint_interpolation} are shown to exist, provided $\bH_\Gamma^{C,F}$ are accurate quadratures. Here, we use intermediate glue grids to construct interpolation operators for the boundary-optimized SBP difference operators, see \cite{kozdon_wilcox_2016, waves2022}.

\begin{remark}
  Replacing $(\bI^{C2F})^T$ in \eqref{eq:anti_plane_shear_grad_coarse} with any accurate interpolation operator $\bI^{F2C}$ would result in consistency with the continuous gradient and an approximation of \eqref{eq:anti_plane_shear_grad_fine} to the order of accuracy of $\bI^{F2C}$. The resulting gradient computation would then correspond to a continuous-adjoint approach \cite{hicken_zingg_2014}. The additional requirement \eqref{eq:adjoint_interpolation} guarantees that the continuous-adjoint approach is equivalent to the discrete-adjoint approach, meaning that we obtain the exact gradient of \eqref{eq:anti_plane_shear_fully_disc_misfit} while remaining consistent with the continuous gradient. \changed{In other words, the interpolation preserves the dual consistency of the forward scheme \eqref{eq:anti_plane_shear_discr_fwd}.} The property \eqref{eq:adjoint_interpolation} states that $\bI^{F2C} = (\bI^{C2F})^\dagger$ in the $\bH_\Gamma$ norm. It is termed SBP or inner product preserving and is commonly used to couple SBP schemes across non-conforming interfaces \cite{mattsson_carpeter_2010,nissen_et_al_2015,lundquist_et_al_2018,kozdon_wilcox_2016,almquist_et_al_2019,lundquist_et_al_2020}. Interestingly, the property is useful also in the present setting.
\end{remark}

\subsection{Numerical studies}\label{sec:num_studies}
In this section, we corroborate our theoretical results and present inversions for friction parameters and initial stresses in dynamic rupture simulations using synthetic data. The experiments are performed on the domain in Figure \ref{fig:fractal_fault_domain}, where the two elastic blocks are separated by a rough (band-limited self-similar fractal) fault \cite{dunham_et_al_2011,fang_dunham_2013}. The parameter set used for synthetic data is listed in Table \ref{table:parameters}. 
\begin{table}[h!]
  \centering
  \caption{\reviewerTwo{Parameter values used in reference problem setup, divided into material parameters, frictional parameters, and initial conditions.}}\label{table:parameters}
  \begin{tabular}{l|l|l}
    Parameter & Symbol & Value \\
    \hline
    Density   & $\rho_\pm$   & 2.6700 g/$\text{cm}^3$ \\
    Shear modulus   & $\mu_\pm$   & 32.0381 GPa \\
    \hline
    Initial stress   & $\sigma^{\changed{0}}_n$   & 120 MPa \\
                     & $\sigma^{\changed{0}}_{yz}$        & 72 MPa \\
    Reference friction coefficient for steady sliding   & $f_0$   & 0.6 \\
    Direct effect parameter   & $a$   & 0.009 $\brx \in \Gamma_{VW}$,\\
                              &       & 0.013 otherwise \\
    State evolution effect parameter   & $b$   & 0.011 \\
    State evolution distance   & $D_c$ & 0.2 m $\brx \in \Gamma_{VW}$,\\
                               &       & 1.0 m otherwise \\
    Reference slip velocity   & $V_0$   & $10^{-6}$ m/s \\
    \hline
    Initial displacement   & $u_{0\pm}$   & 0\\
    Initial velocity   & $v_{0\pm}$   & $\changed{\pm} 5 \changed{\times} 10^{-13}$ m/s\\
    Initial state   & $\Psi_0$   & 0.7243\\ 
  \end{tabular}
\end{table}
Here $\Gamma_{VW} = \{\brx \in \Gamma, x \in [-5,6] \text{ km} \}$ denotes the central velocity\changed{-}weakening (VW) region of the fault, \ie, the part of the fault where \reviewerTwo{$a(\bar{x}) - b(\bar{x}) < 0$}. Almost all slip occurs in this region. In contrast, the region where \reviewerTwo{$a(\bar{x}) - b(\bar{x}) > 0$} is velocity\changed{-}strengthening (VS), which serves to arrest the rupture. The initial stress tensor is such that the normal stress $\sigma^{\changed{0}}_n$ is constant while \changed{the} shear traction varies along the fault according to \reviewerTwo{$\tau^0(\bar{x}) = \sigma_{yz}^0 \hat{n}^-_y(\bar{x})$}, where \reviewerTwo{$\hat{n}^-_y(\bar{x})$} is the (spatially variable) $y$-component of the fault normal \changed{on the $-$ side of the fault,} and $\sigma_{yz}^0$ is a constant remote shear stress \changed{(and $\sigma_{xz}^0 = 0$)}. A one-dimensional Gaussian distribution centered about $x_c = 3$ km, with amplitude $A = 25$ MPa and standard deviation $d = 2$ km is added as loading to the fault, \ie, setting $\tau_L(\brx) = A e^{-\frac{(x-x_c)^{2}}{2d^2}}$ in \eqref{eq:F_antiplane_shear}. The problem setup is such that the external loading will cause an earthquake to nucleate at \changed{$\bar{x}_c = (x_c,y_c)$ where $y_c = -0.896$ km is the $y$-coordinate on the fault at $x_c$}. The stress and friction conditions are chosen so that rupture nucleates gradually over the first few seconds of the simulation, before transitioning to a propagating rupture. Left- and right-going rupture fronts propagate along the fault at a speed slightly lower than the shear wave speed until they reach the VS region where rupture arrests.

Using $m$ grid points in the $x$-direction, the domain is discretized such that the grid spacing is approximately equal in both directions, resulting in $m \times (m+1)$ grid points in total across the multiblock grid $\bOmega = \bOmega_-\cup\bOmega_+$. The number of grid points used for lower-dimensional parameter representations is denoted $m_p$. 8th-order accurate boundary-optimized SBP operators are used for the spatial discretization. In \cite{stiernstrom_et_al_2023} it was shown that these operators provide an efficient alternative for wave propagation problems where boundary or interface effects are significant. The non-linear equation \eqref{eq:V_star_nonlin} for target slip velocity $\bV^*$ is solved using bisection, specifying an absolute tolerance $10^{-13}$ m/s. 

\changed{Snapshots of the velocity field from a forward simulation using $m = 1001$ and $\Delta t = 0.0005$ s up to $T = 6$ s are presented in Figure \ref{fig:fractal_fault_v}, illustrating the nucleation, propagation, and arrest of the rupture.} The velocity field is nonzero within a circular wavefront known in seismology as the starting phase. It has peaks at the two rupture fronts and is discontinuous across the parts of the fault that are actively slipping. Additional circular wavefronts emanate from the fault as the rupture accelerates or decelerates in response to the spatially variable initial stress and complex fault geometry. Additional larger amplitude waves are produced \changed{by} rupture arrests that bring the particle velocity back toward zero; these are known as stopping phases. Finally, wave reflections from the right boundary are visible, as the outflow boundary condition is only effective at absorbing normally incident waves. \changed{In Figure \ref{fig:fractal_fault_slip_and_vel} plots of slip velocity $\bV^*$ and fault slip $\jump{\bu}$ are presented. This setup will be used in Section \ref{sec:high_resolution_data} to generate high-resolution synthetic data used for inversion.}

\begin{figure}[h!]
  \centering
  \begin{subfigure}[b]{0.49\textwidth}
      \centering
      \includegraphics[width=\textwidth]{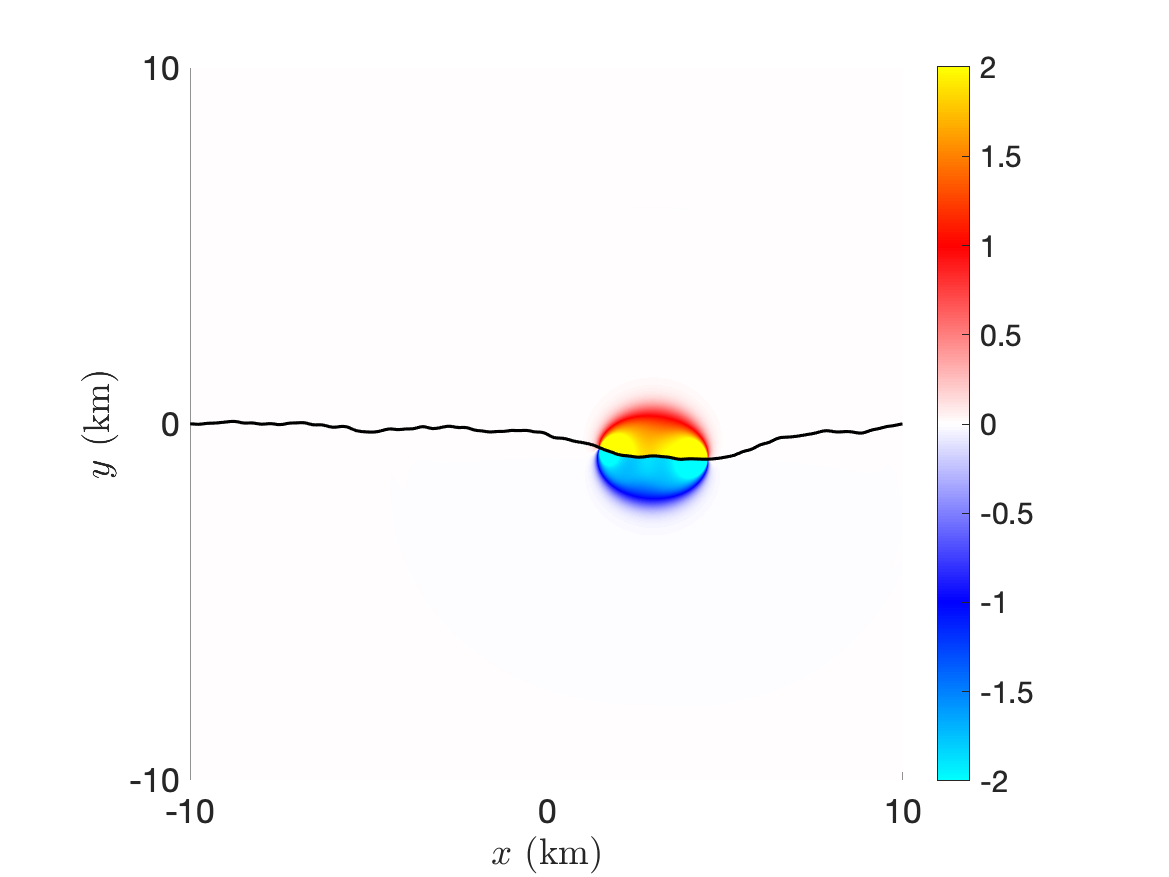}
      \caption{$t=2.00$ s}
  \end{subfigure}
  \hfill
  \begin{subfigure}[b]{0.49\textwidth}
      \centering
      \includegraphics[width=\textwidth]{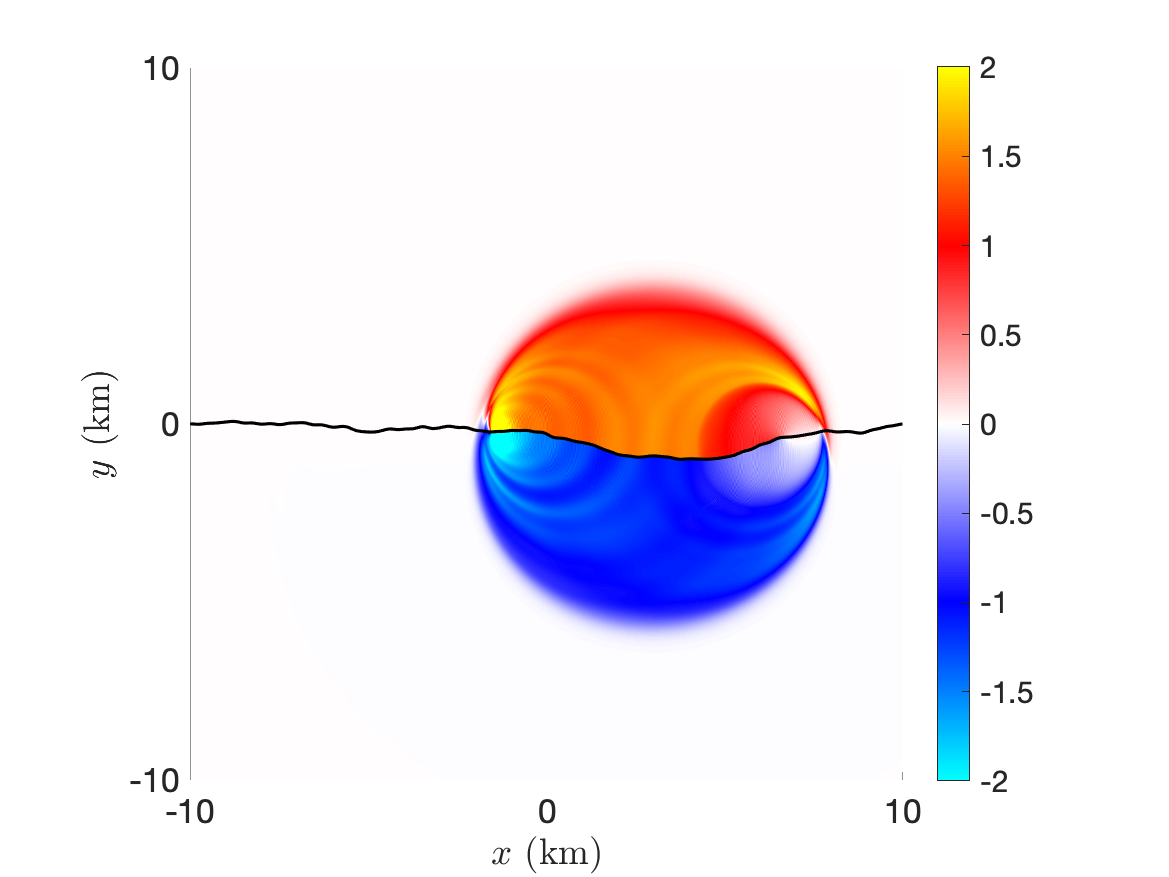}
      \caption{$t=3.00$ s}
  \end{subfigure}
  \\
  \begin{subfigure}[b]{0.49\textwidth}
    \centering
    \includegraphics[width=\textwidth]{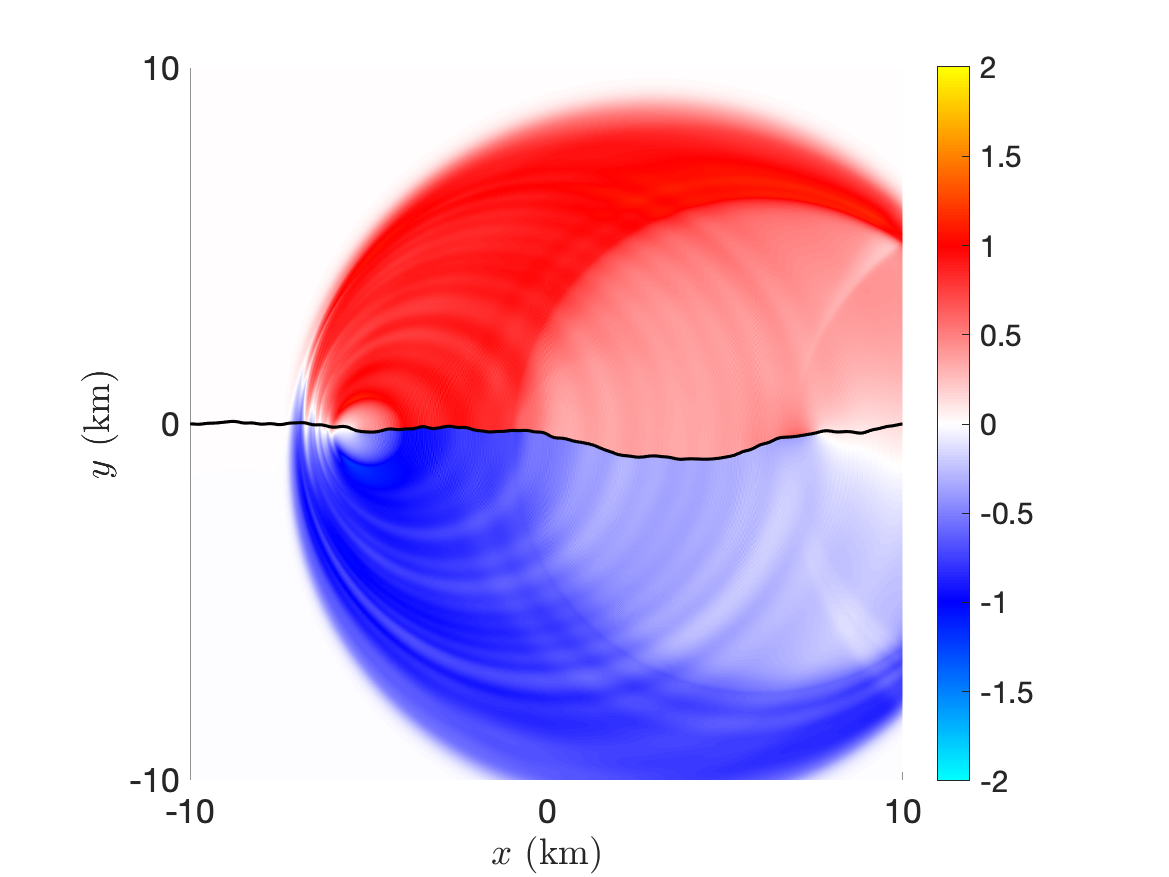}
    \caption{$t=4.50$ s}
  \end{subfigure}
  \hfill
  \begin{subfigure}[b]{0.49\textwidth}
      \centering
      \includegraphics[width=\textwidth]{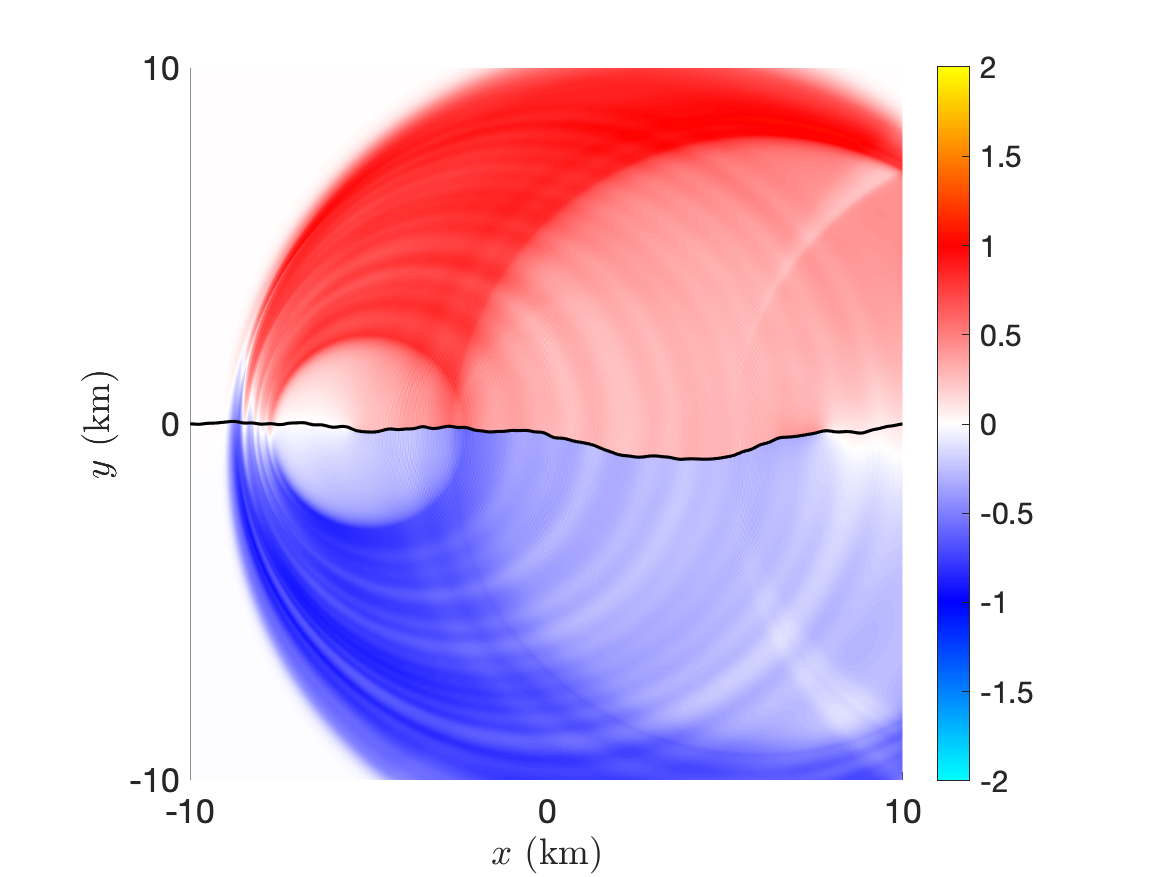}
      \caption{$t=5.00$ s}
  \end{subfigure}
     \caption{Velocity (m/s) at different points in time, showing rupture propagation and arrest.}
     \label{fig:fractal_fault_v}
\end{figure}

\begin{figure}[h!]
  \centering
\begin{subfigure}[b]{0.5325\textwidth}
  \centering
  \includegraphics[width=\textwidth]{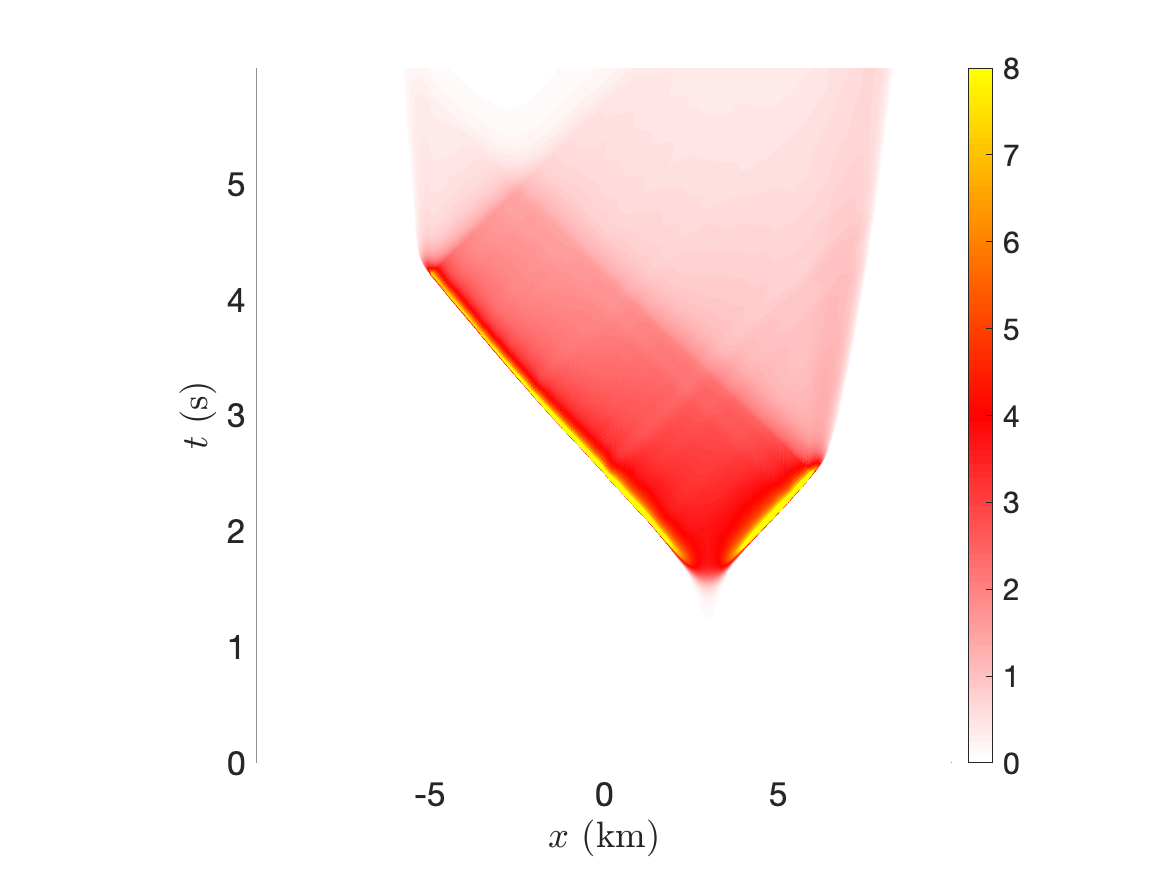}
  \caption{}
  \label{fig:fractal_fault_slip_velocity}
\end{subfigure}
\hfill
\begin{subfigure}[b]{0.4575\textwidth}
    \centering
    \includegraphics[width=\textwidth]{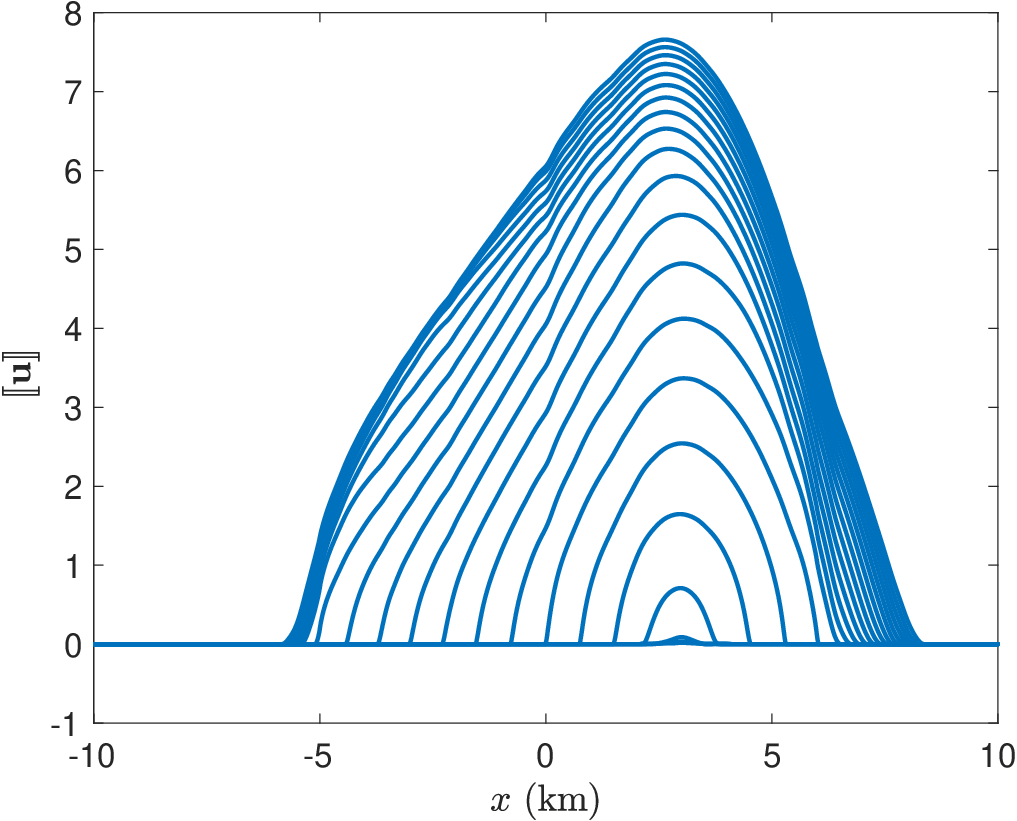}
    \caption{}
    \label{fig:fractal_fault_slip}
\end{subfigure}
  \caption{\changed{(\subref{fig:fractal_fault_slip_velocity}) Space-time plot of slip velocity $\bV^*$ (m/s). (\subref{fig:fractal_fault_slip}) Slip $\jump{\bu}$ (m) plotted every $t = 0.25$ s.}}
    \label{fig:fractal_fault_slip_and_vel}
\end{figure}
\changed{In the inversions, receivers are placed in a cut-out rectangle such that the region closest to the VW part of the fault is excluded. An example of $N_{rec} = 88$ receivers placed in the cut-out rectangle with outer boundary $-9 \le x \le 9$, $-9 \le y \le 9$, and inner boundary $-7 \le x \le 7$, $-3 \le y \le 3$, spaced 2 km apart, is shown in Figure \ref{fig:receivers}.} \reviewerTwo{In Figure \ref{fig:seismograms} an example of recorded synthetic velocity data from the high-resolution simulation described above is shown, with receivers ordered according to the radial distance $r$ to the hypocenter $\bar{x}_c$. From Figure \ref{fig:seismograms} it is clear that the wavefield is densely sampled and that the receivers record the radial moveout of the high-amplitude waves generated as the rupture front expands.} 
\begin{figure}[h!]
  \centering
  \begin{subfigure}[b]{0.4\textwidth}
      \centering
      \includegraphics[width=\textwidth]{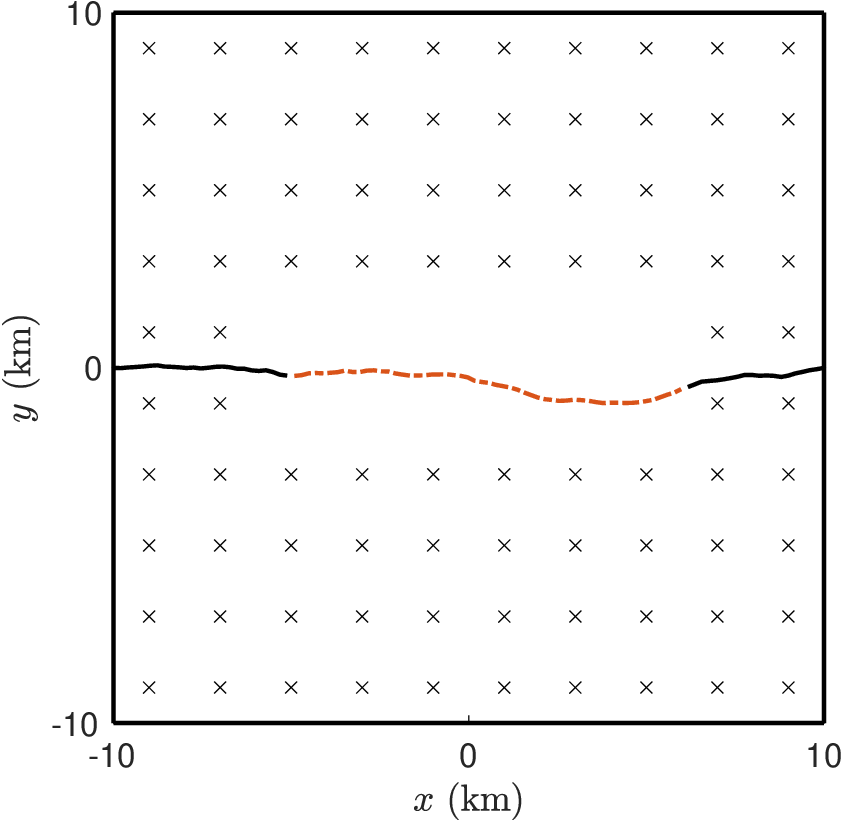}
      \caption{}
      \label{fig:receivers}
  \end{subfigure}
  \hfill
  \begin{subfigure}[b]{0.56\textwidth}
      \centering
      \includegraphics[width=\textwidth]{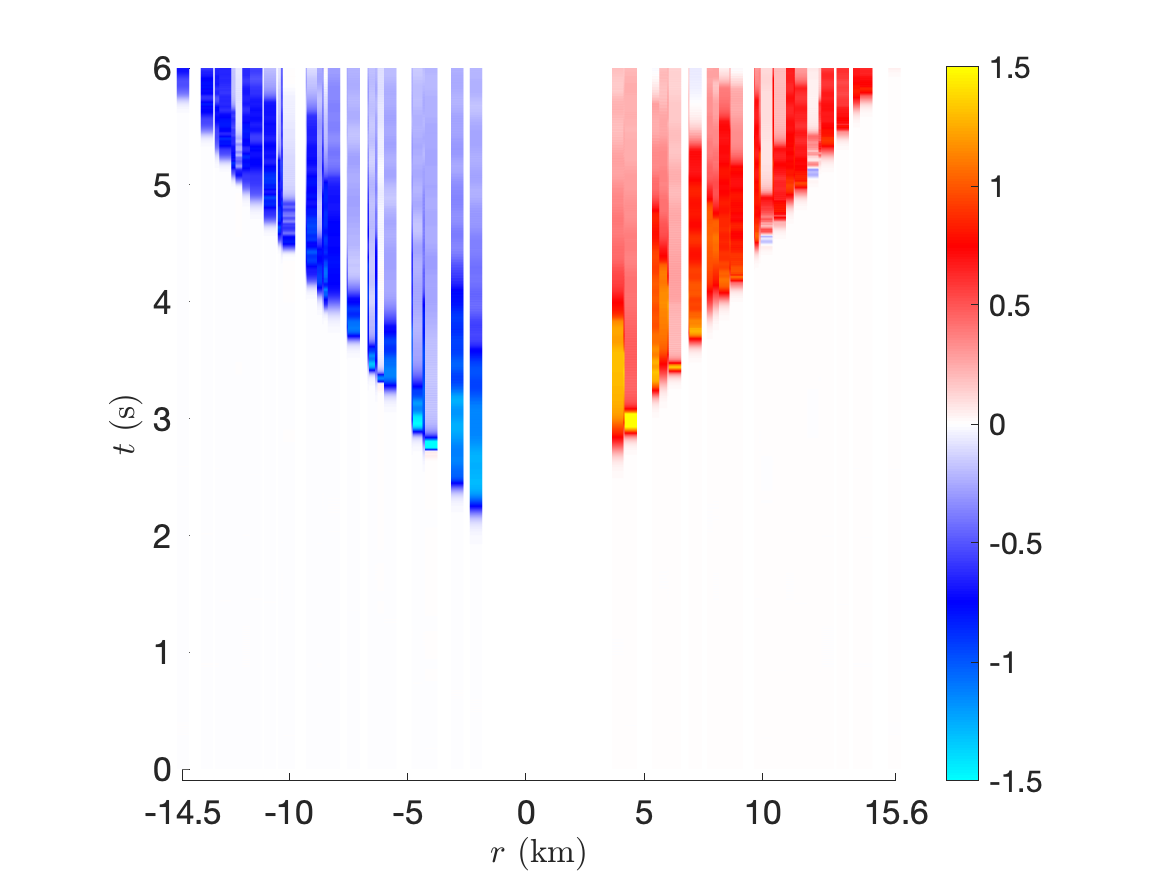}
      \caption{}
      \label{fig:seismograms}
  \end{subfigure}
  \caption{\changed{
    (\subref{fig:receivers}) Position of $N_{rec} = 88$ receivers spaced 2 km apart ($\times$). Velocity-weakening part of fault (-\,-).}
    \reviewerTwo{(\subref{fig:seismograms}) Synthetic velocity data (m/s) at receivers, ordered according to the radial distance $r$ between the receiver and the hypocenter $\bar{x}_c$. $r < 0$ refers to receivers in $\Omega^-$ while $r > 0$ to receivers in $\Omega^+$.}}
   \label{fig:fractal_fault_data}
\end{figure}

\subsubsection{Verification of discrete gradient}\label{sec:gradient_verification}
To verify Theorem \ref{thm:disc_grad} we compare with a first-order finite difference approximation of the gradient. Consider the parameter-normalized norm
\reviewerTwo{\begin{equation}
  \|\bv \|_{\bp} = \|\dbar{\bp}^{-1}\bv\|_\infty,
\end{equation}
where $\bv$ is a fault grid function on the parameter grid (\ie, of size $m_p \times 1$). We then define the relative error in the gradient as}
\begin{equation}
  e(\Delta p) = \frac{\|\frac{\partial \bm{\F}}{\partial \bp} - \bD_+\bm{\F}(\Delta p)\|_{\bp}}{\|\frac{\partial \bm{\F}}{\partial \bp}\|_{\bp}},
\end{equation}
where $\frac{\partial \bm{\F}}{\partial \bp}$ is the adjoint-based gradient and $\bD_+\bm{\F}(\Delta p )$ is the first-order finite difference approximation, with the $i$th component given by
\begin{equation}\label{eq:}
  \left(\bD_+\bm{\F}(\Delta p )\right)_i = \frac{\bm{\F}(\bp + \Delta p\be_i) - \bm{\F}(\bp)}{\Delta p}.
\end{equation}
Here, $\be_{i}$ is a vector with a 1 at the $i$th entry and zeros elsewhere, \reviewerTwo{such that $\bD_+\bm{\F}(\Delta p )$ is computed by individually perturbing each element in $\bp$ by $\Delta p$}. Synthetic data is generated up to $T=6$ s using the parameter values in Table \ref{table:parameters}, and discretization parameters $m = 101$,  $m_p = 11$ and $\Delta t = 0.005$ s. Note that the grid is too coarse for accurate dynamic rupture simulations but the setup is still useful in verifying that the correct discrete gradient is obtained since any approximation errors in the gradient should be clearly visible. To interpolate between the computational and the parameter grid we construct interpolation operators based on the 8th-order accurate boundary-optimized norm, and the second-order accurate traditional SBP norm on an equispaced grid. This means that the interpolation is second-order accurate. The reason for using a lower-order accurate interpolation is that it reduces oscillations when interpolating non-smooth parameters. \changed{$N_{rec} = 88$ receivers are used, positioned as illustrated in Figure \ref{fig:receivers}}.


\reviewerTwo{Initializing $\mathbf{a}$ from $1.1a(\bar{x})$ as given by Table \ref{table:parameters}}, the error $e(\Delta a)$ is then computed for $\Delta a$ decreasing from $10^{-5}$ to $10^{-12}$. The results are presented in Figure \ref{fig:grad_conv_fractal_fault} where it can be observed that $e(\Delta a)$ decreases with a first-order rate with decreasing $\Delta a$ until $e(\Delta a) \approx 1.8\changed{\times}10^{-6}$, at $\Delta a = 10^{-8.5}$ with a displacement misfit and $e(\Delta a) \approx 1.4\changed{\times10^{-5}}$, at $\Delta a = 10^{-9}$ with a velocity misfit. At these points, cancellation errors in the first-order approximation $\bD_+\bm{\F}(\Delta a)$ take effect after which $e(\Delta a)$ increases linearly (although oscillatory) as $\Delta a$ is further decreased. Note that if \changed{$\frac{\partial \bm{\F}}{\partial \mathbf{a}}$} included an approximation error \changed{$\varepsilon$} then the error curve would become horizontal at $e(\Delta a) = \varepsilon$\changed{, provided $\varepsilon$ is larger than cancellation errors in $\bD_+\bm{\F}(\Delta a)$}. Since no such piecewise constant parts of the error curves are observed, we conclude that \eqref{eq:anti_plane_shear_fully_discr_grad} indeed is the gradient of $\bm{\F}$, and that errors in computing $\bV^*$ are negligible.

\begin{figure}[h!]
    \centering
    \includegraphics[width=0.5\textwidth]{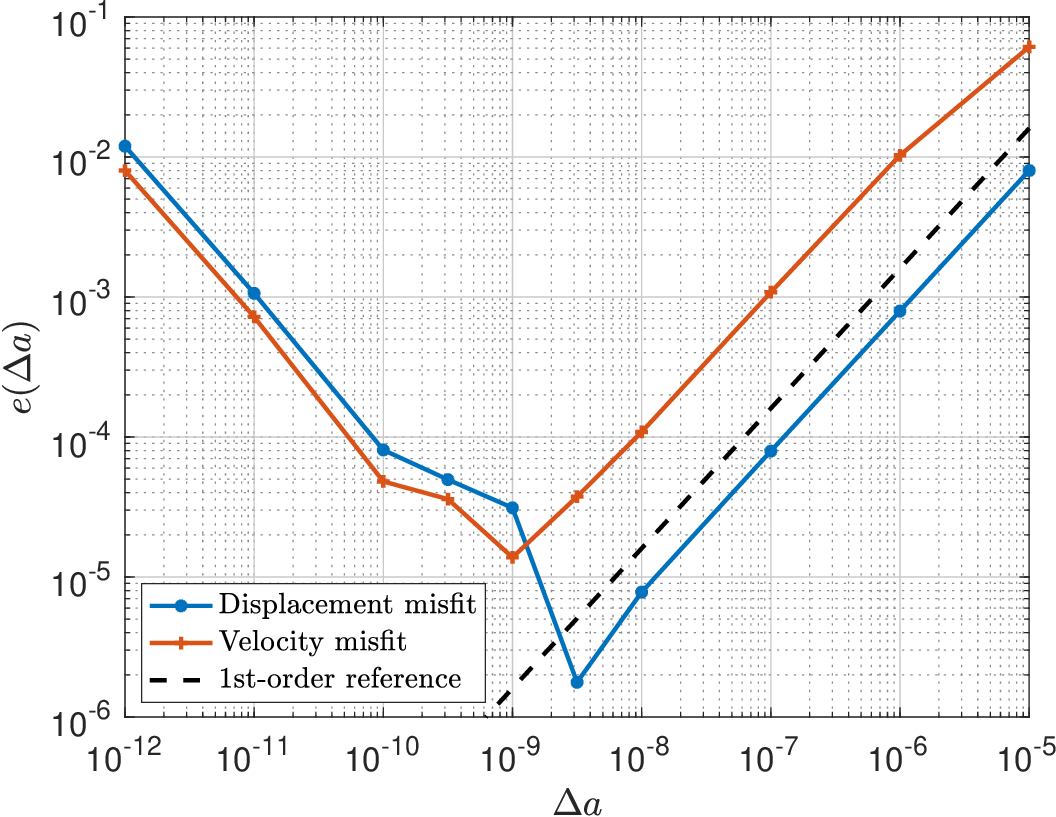}
    \caption{Relative error $e(\Delta a)$ between the adjoint-based gradient and the \changed{first-order finite difference} gradient, as a function of $\Delta a$.}
    \label{fig:grad_conv_fractal_fault}
\end{figure}

\subsubsection{Inverse crimes}\label{sec:inverse_crimes}
We now proceed to perform inversions in an inverse crime setting, i.e., we aim to reconstruct parameters based on synthetic \changed{velocity} data generated from a forward simulation using the same spatio-temporal resolution. There are no model errors or differences in the fault geometry, boundary conditions, or problem parameters (except for the parameters selected as model parameters in the inversions). The wavefield is densely sampled. Thus these inversions are not representative of conditions that might be encountered in real-world applications; instead they serve to test our methodology and to demonstrate the possibility of adjoint-enabled gradient-based optimization for this class of problems. We also perform inversions for a single parameter (i.e., the direct effect parameter \reviewerTwo{$a(\bar{x})$} or the initial shear stress \reviewerTwo{$\tau^0(\bar{x}$)}) rather than doing inversions for multiple parameters simultaneously. This avoids having to develop regularization schemes to handle parameter trade-offs that will inevitably arise in real-world applications.

The optimization problem \eqref{eq:anti_plane_shear_fully_disc_opt} is solved using the quasi-Newton limited-memory Broyden-Fletcher-Goldfarb-Shanno (L-BFGS) algorithm provided by the Matlab function \texttt{fmincon}. The forward and adjoint solves are performed using $m = 251$ and $\Delta t = 0.003$ s, with $m_p$ varying depending on the parameter considered for inversion. As before, synthetic data are generated using the parameters in Table \ref{table:parameters} as the true values.

Inverting for the direct effect parameter \reviewerTwo{$a(\bar{x})$} we set $m_p = 26$, \reviewerTwo{and initialize $\mathbf{a}$ identically to 0.0135}, which corresponds to a 50\% error in the VW region. This choice of initial \reviewerTwo{$\mathbf{a}$} makes the entire fault velocity-strengthening, so in the initial iteration, rupture immediately arrests. \changed{Again, the $N_{rec} = 88$ receiver setup illustrated in Figure \ref{fig:receivers} is used}. Figure \ref{fig:inverse_crime_a} presents $\mathbf{a}$ plotted for different iterations; the value of \reviewerTwo{$\mathbf{b}$} is also shown to help delimit VW and VS parameter values. \changed{After the first 5 iterations, we observe that $\mathbf{a}$ has been decreased close to its true value at the hypocenter $x_c = 3$ km. This allows for the earthquake to nucleate, radiating waves and thereby reducing the misfit for first-arrival waves. However, since $\mathbf{a}-\mathbf{b} > 0$ outside of the hypocenter, the rupture will not propagate. In subsequent iterations, $\mathbf{a}$ is adjusted to account for the spatial extent of the rupture.} After 200 iterations, $\mathbf{a}$ has decreased close to its true value, such that the fault becomes VW within the central region where slip is nonzero (Figure \ref{fig:fractal_fault_slip_and_vel}). Note that the difference in $\mathbf{a}$ in the first and final iterations is negligible in the VS parts of the fault. This part of the fault started VS, so that \changed{the} rupture arrests there, and no adjustments are required for consistency with the seismograms at the receivers because without slip this part of the fault does not radiate. Thus there is a very low sensitivity (small value of \changed{$\frac{\partial\bm{\F}}{\partial\mathbf{a}})$} in the VS region. In contrast, matching the timing and amplitude of wave arrivals from the interior slipped part of the fault is only possible when \reviewerTwo{$\mathbf{a}$} is close to its true value.

\begin{figure}[h!]
  \centering
  \begin{subfigure}[b]{0.49\textwidth}
      \centering
      \includegraphics[width=\textwidth]{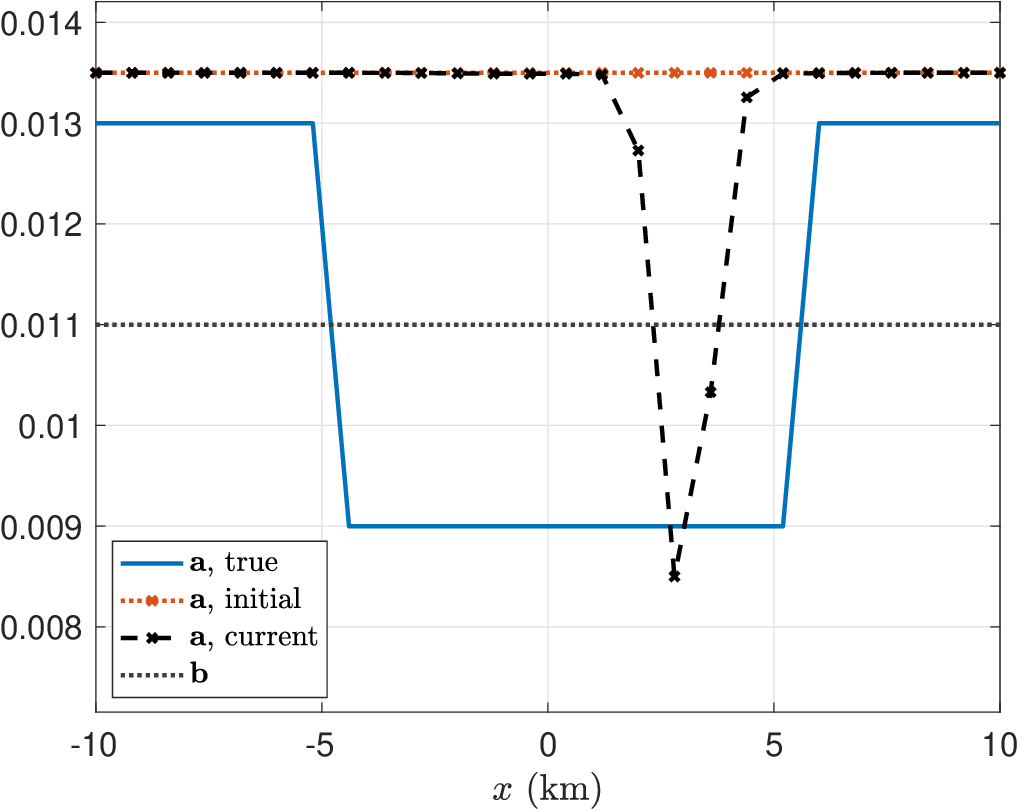}
      \caption{\changed{5} iterations}
  \end{subfigure}
  \hfill
  \begin{subfigure}[b]{0.49\textwidth}
      \centering
      \includegraphics[width=\textwidth]{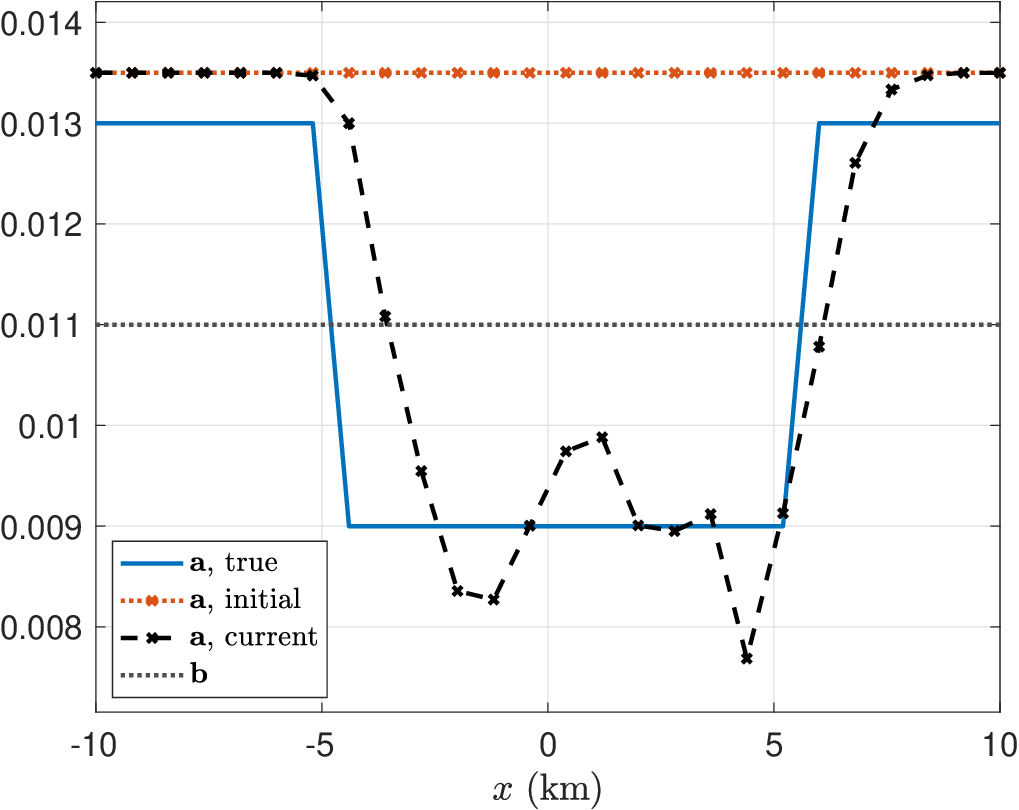}
      \caption{\changed{50} iterations}
  \end{subfigure}
  \\
  \begin{subfigure}[b]{0.49\textwidth}
    \centering
    \includegraphics[width=\textwidth]{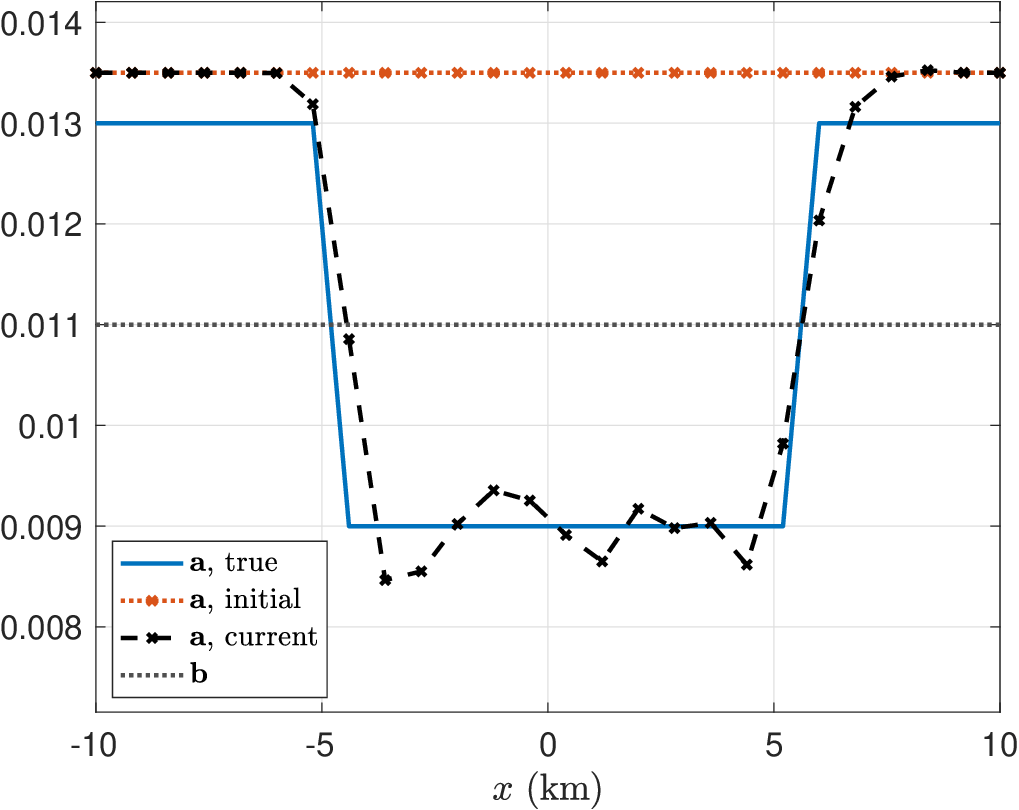}
    \caption{\changed{100} iterations}
  \end{subfigure}
  \hfill
  \begin{subfigure}[b]{0.49\textwidth}
      \centering
      \includegraphics[width=\textwidth]{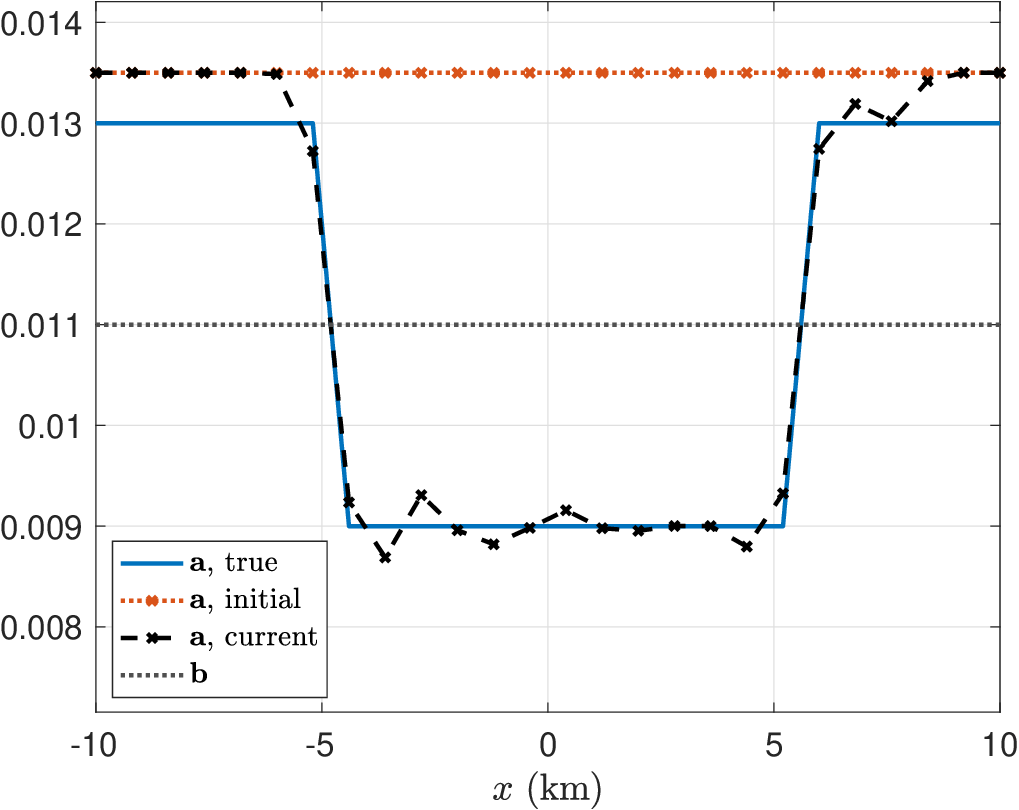}
      \caption{200 iterations}
  \end{subfigure}
  \caption{Direct effect parameter $\mathbf{a}$ at different iterations for $m_p = 26$.}
  \label{fig:inverse_crime_a}
\end{figure}

To illustrate the effect of the parameter resolution $m_p$ on the optimization procedure, we perform the same inversion using $m_p = 51$ and $m_p = 101$. The computed values of $\mathbf{a}$ after 200 iterations, and the misfit histories are presented in Figures \ref{fig:inverse_crime_a_cmp_mp}. Increasing the resolution on the parameter grid, setting $m_p = 51$, we see larger oscillations about the hypocenter $x_c$. This shows that reducing the discrete parameter space may be beneficial for faster convergence of the misfit, and for reducing the number of local minima. The oscillations around the hypocenter observed for higher $m_p$ could likely be remedied by introducing regularization to the misfit, for instance by penalizing large gradients in $\mathbf{a}$. Finding a suitable regularization is however out of the scope of this paper.

\begin{figure}[h!]
  \centering
  \begin{subfigure}[b]{0.49\textwidth}
      \centering
      \includegraphics[width=\textwidth]{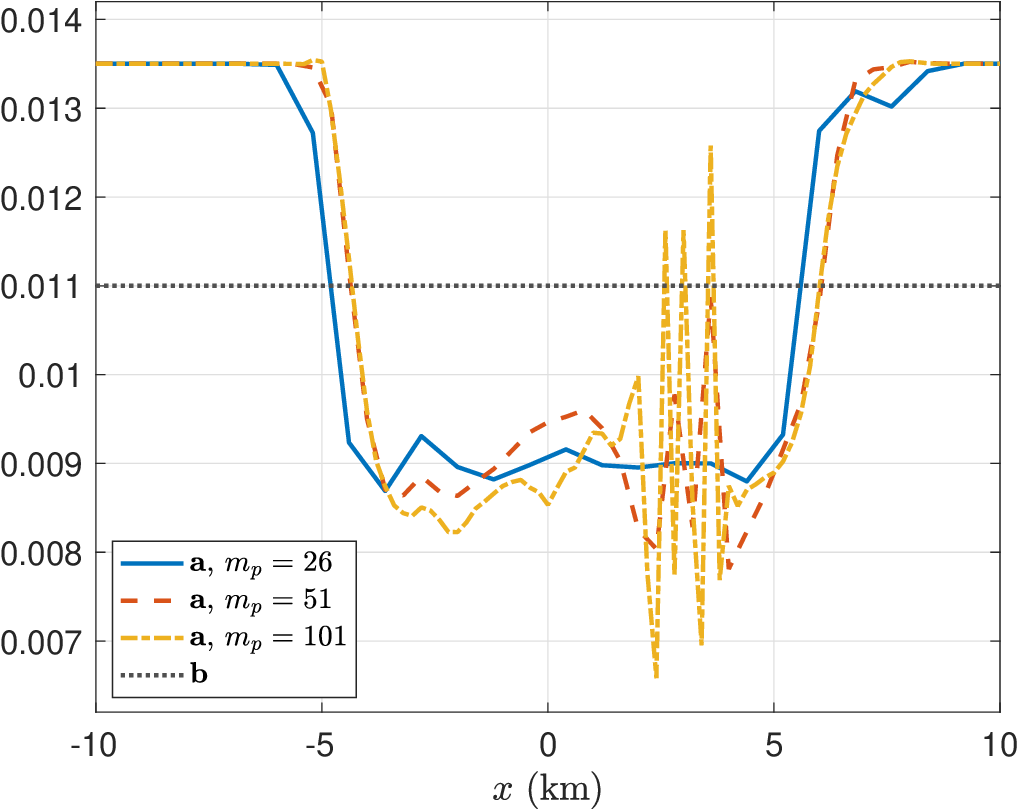}
      \caption{200 iterations}
  \end{subfigure}
  \hfill
  \begin{subfigure}[b]{0.49\textwidth}
      \centering
      \includegraphics[width=\textwidth]{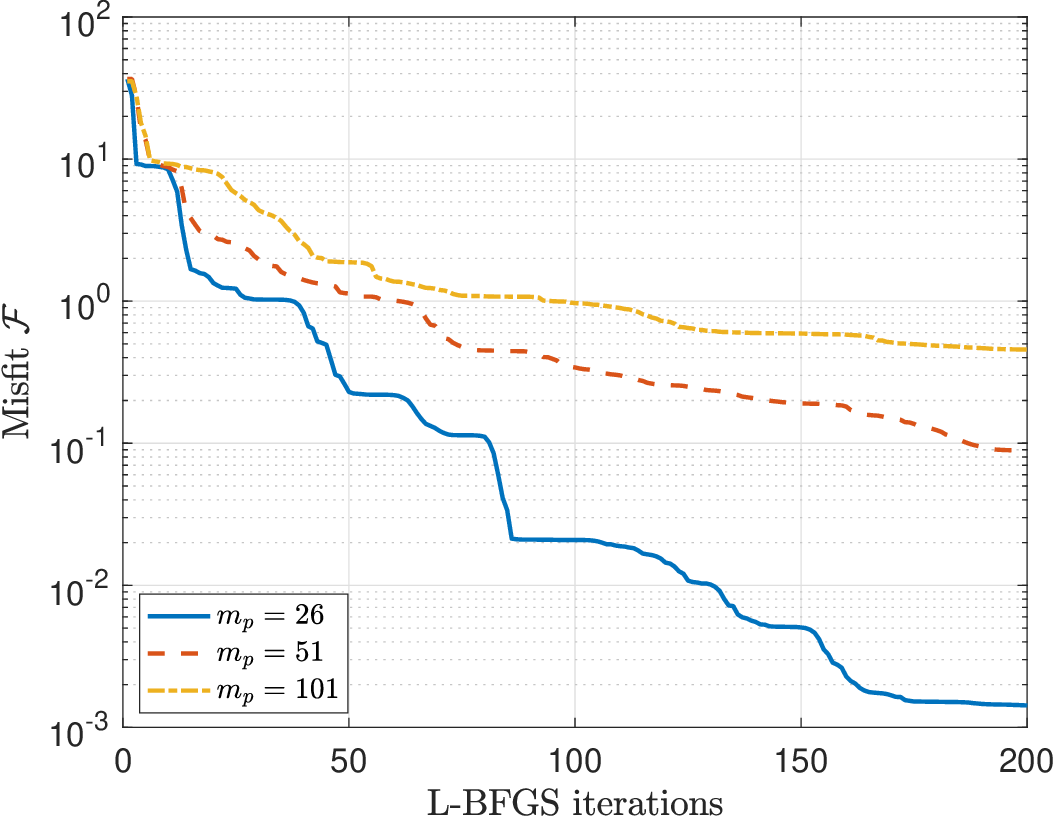}
      \caption{Misfit history}
  \end{subfigure}
  \caption{Direct effect parameter $\mathbf{a}$ and misfit history for different $m_p$.}
\label{fig:inverse_crime_a_cmp_mp}
\end{figure}

Next, we invert for the shear traction \reviewerTwo{$\tau^0(\bar{x}) = \sigma_{yz}^0 \hat{n}^-_y(\bar{x})$}, setting $m_p = 51$, and starting from \changed{$\btau^0$} given by \changed{$\sigma_{yz}^0 = 68$} MPa corresponding to approximately 5\% initial error. The smaller error in the initial guess is due to inversions of \reviewerTwo{$\tau^0(\bar{x})$} proving to be significantly more difficult for this problem setup, with a higher likelihood of incorrectly converging to a local minimum. This choice of initial guess reduces the initial stress below the true value, to the point where the rupture with the initial parameter values quickly arrests. The inversion must therefore increase \reviewerTwo{$\btau^0$} to produce a propagating rupture. Two inverse-crime inversions are performed, one with the receiver distribution in Figure \ref{fig:receivers} and one where the number of receivers is increased by placing them 1 km apart and decreasing the inner bounding box to  $-6 \le x \le 7$, $-2 \le y \le 2$, for distances in km, totaling $N_{rec} = 325$. In Figure \ref{fig:inverse_crime_tau0}\changed{, $\btau^0$} obtained with $N_{rec} = 88$ is plotted for different iterations. \changed{Again, the first 10 iterations adjust the parameter at the hypocenter, while subsequent iterations adjust for the spatial extent of the rupture. However, in this case, the} optimization stagnates at a local minimum after about \changed{100} iterations\changed{, as illustrated by the misfit history shown in Figure \ref{fig:inverse_crime_tau0_misfit}}. A comparison of the results at 200 iterations with $N_{rec} = 88$ and $N_{rec} = 325$ is shown in Figure \ref{fig:inverse_crime_tau0_cmp_rec}. The result using $N_{rec} = 325$  matches the true value well in the VW region. The increased receiver density as well as having receivers closer to the fault provides more constraints on the shorter wavelength and shorter timescale parts of the rupture process, which manifests as improved resolution of \reviewerTwo{$\btau^0$}.

\begin{figure}[h!]
  \centering
  \begin{subfigure}[b]{0.49\textwidth}
      \centering
      \includegraphics[width=\textwidth]{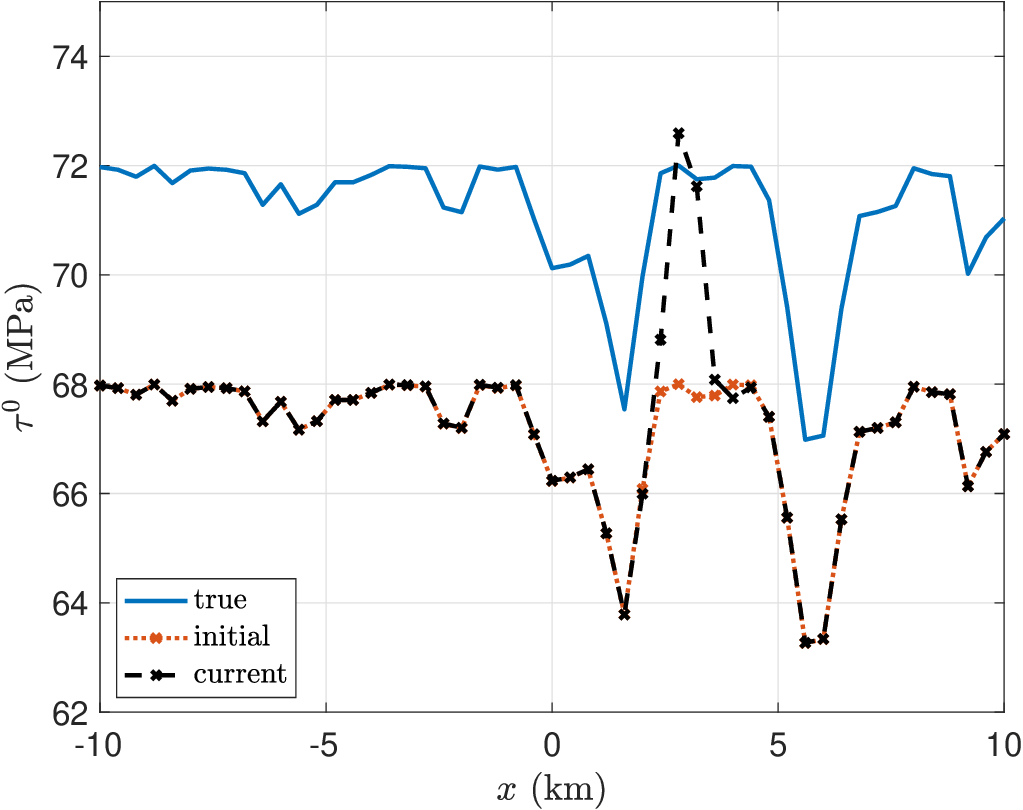}
      \caption{\changed{10 iterations}}
  \end{subfigure}
  \hfill
  \begin{subfigure}[b]{0.49\textwidth}
      \centering
      \includegraphics[width=\textwidth]{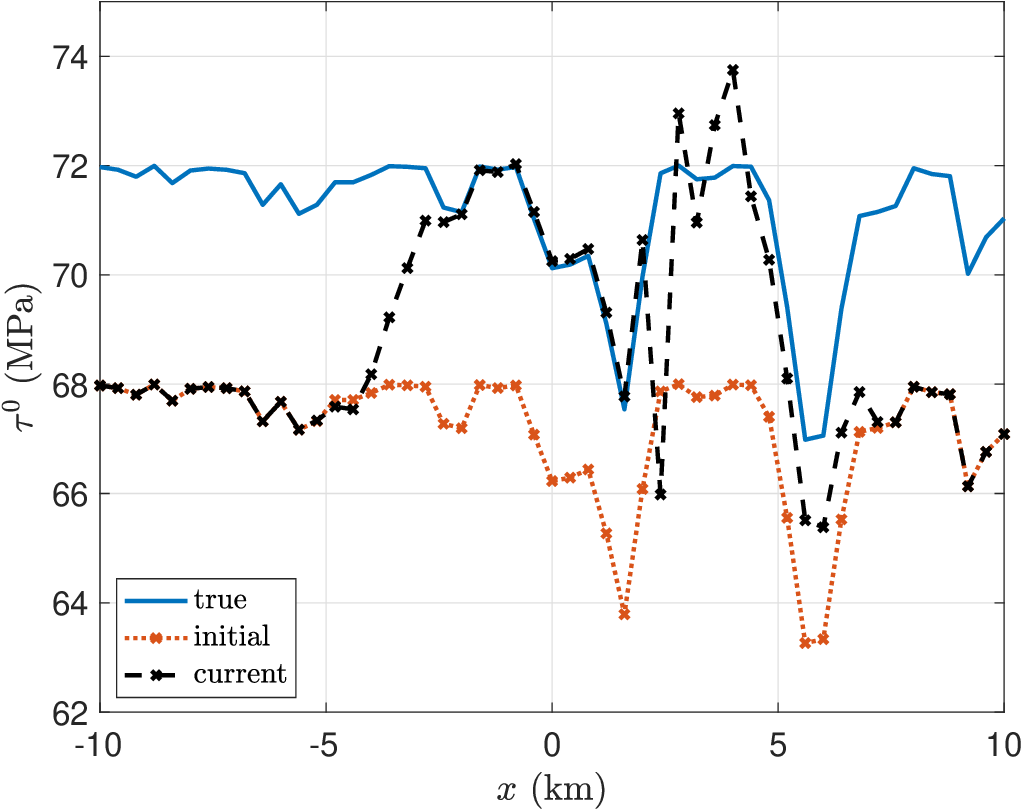}
      \caption{\changed{40 iterations}}
  \end{subfigure}
  \\
  \begin{subfigure}[b]{0.49\textwidth}
    \centering
    \includegraphics[width=\textwidth]{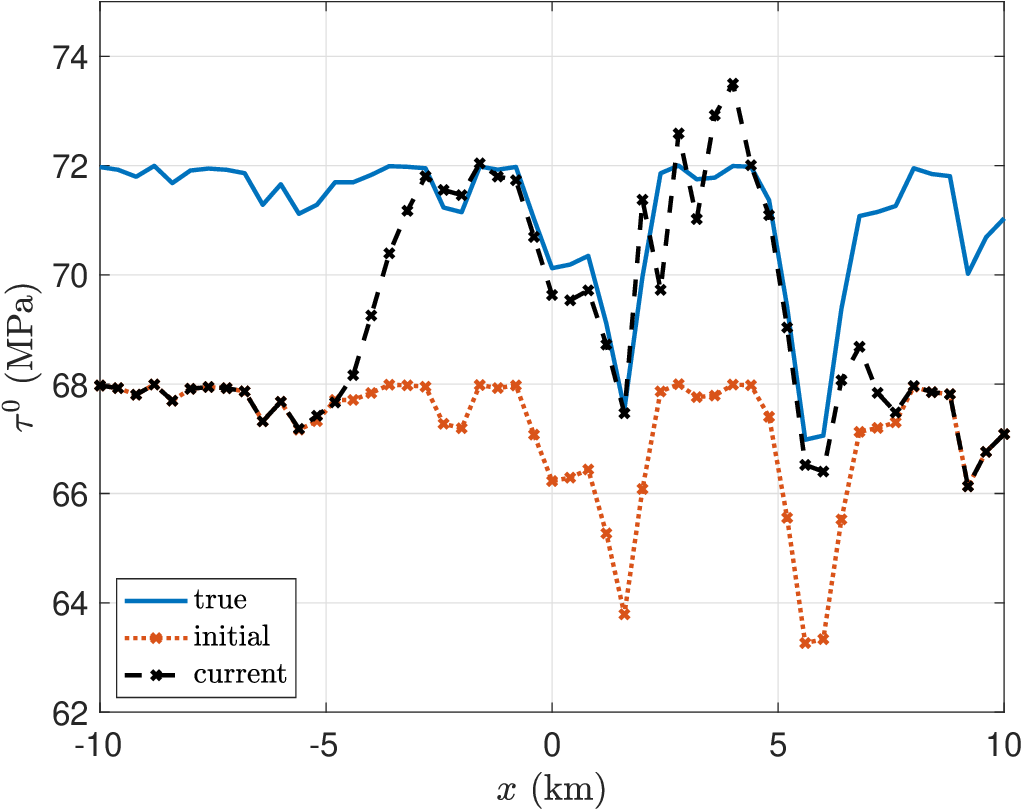}
    \caption{\changed{80 iterations}}
  \end{subfigure}
  \hfill
  \begin{subfigure}[b]{0.49\textwidth}
      \centering
      \includegraphics[width=\textwidth]{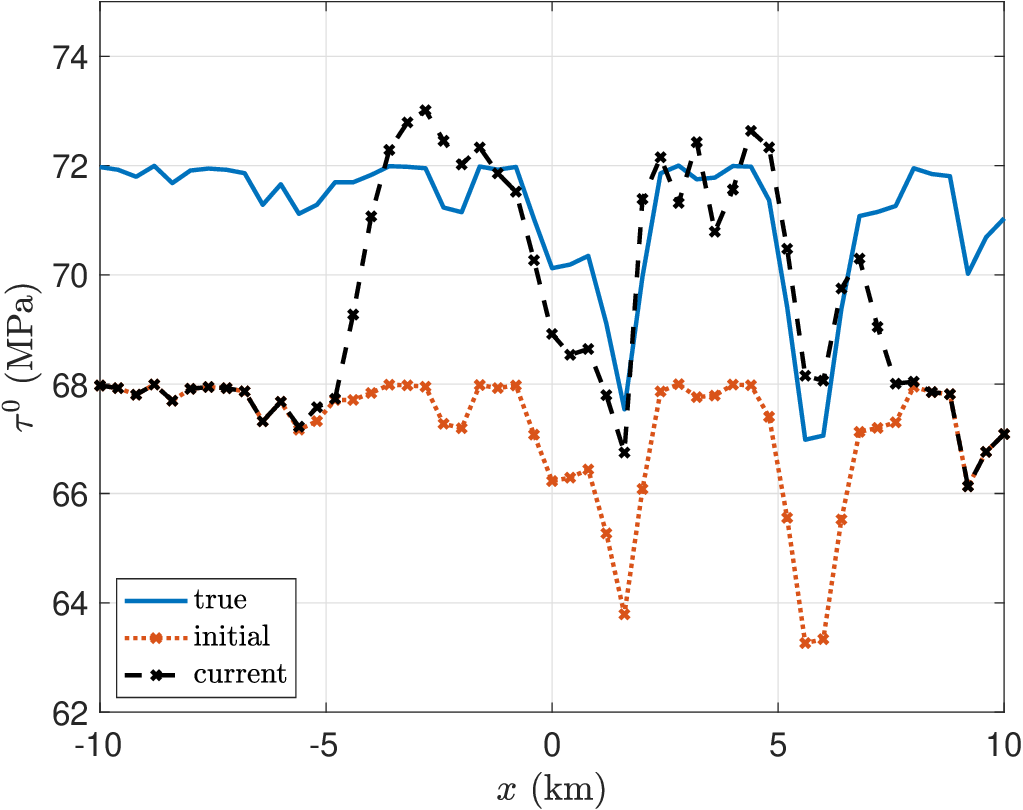}
      \caption{200 iterations}
  \end{subfigure}
  \caption{Shear traction \changed{$\btau^0$} at different iterations using $N_{rec} = 88$.}
  \label{fig:inverse_crime_tau0}
\end{figure}


\begin{figure}[h!]
  \begin{subfigure}[b]{0.49\textwidth}
    \centering
    \includegraphics[width=\textwidth]{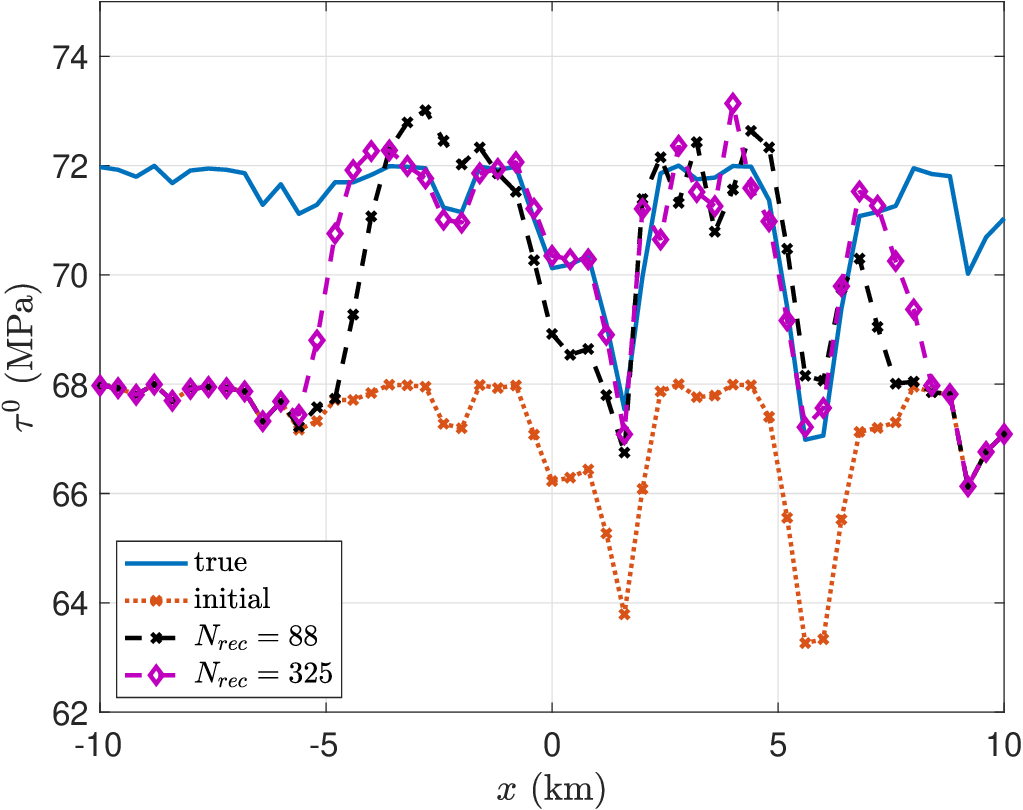}
    \caption{\changed{$\btau^0$} at 200 iterations.}
  \end{subfigure}
  \hfill
  \begin{subfigure}[b]{0.49\textwidth}
      \centering
      \includegraphics[width=\textwidth]{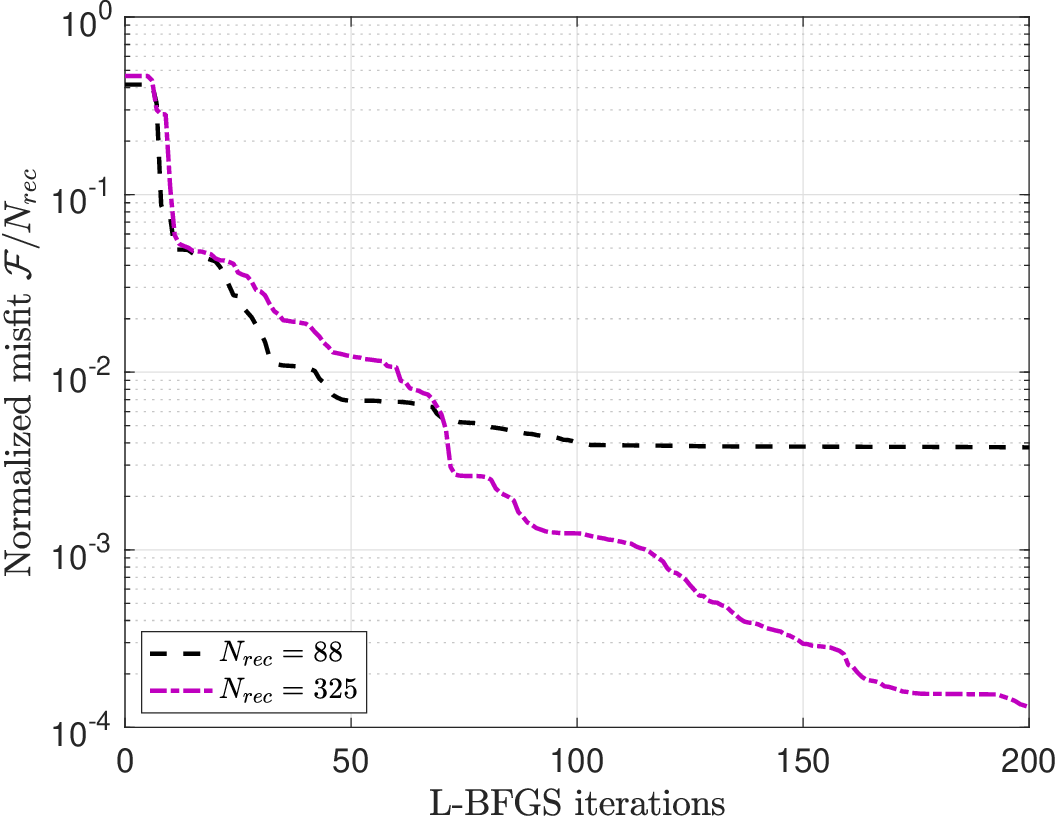}
      \caption{Normalized misfit $\bm{\F}/N_{rec}$ history.}
      \label{fig:inverse_crime_tau0_misfit}
  \end{subfigure}
  \caption{Comparison of shear traction \changed{$\btau^0$} at 200 iterations for $N_{rec} = 88$ and $N_{rec} = 325$.}
  \label{fig:inverse_crime_tau0_cmp_rec}
\end{figure}

\begin{remark}
  Numerical experiments (not presented here) indicate that the \changed{inversion} is more robust when \changed{initializing the inverted parameter such that no nucleation occurs. As observed, this causes the inversion to adjust the parameter to create radiation in order to reduce the misfit, which is done by first adjusting the parameter at the hypocenter and subsequently for the full spatial extent of the rupture. Conversely, starting from parameters causing truly stable parts of the fault to rupture, \eg, }if initially \reviewerTwo{$\mathbf{a} - \mathbf{b} < 0$} in parts of the fault that truly are VS, the \changed{inversion} is very likely to converge to an incorrect local minimum.
\end{remark}

\subsubsection{Inversion using high-resolution synthetic data}\label{sec:high_resolution_data}
Slightly increasing the difficulty of the optimization problem, we perform inversions using the same setups as in Section \ref{sec:inverse_crimes}, but use high-resolution \changed{velocity} data generated from the forward simulation presented at the beginning of this section (illustrated in \changed{Figures \ref{fig:fractal_fault_v} - \ref{fig:fractal_fault_data}}). In this setting, we effectively have no modeling errors (other than the perturbed parameters considered for inversion), but the data contain features that are not resolvable on the computational grid. When inverting for \reviewerTwo{$a(\bar{x})$} we use $m_p = 26$ and $N_{rec} = 88$ receivers \changed{(\ie, data as shown in Figure \ref{fig:seismograms})}, while for \reviewerTwo{$\tau^0(\bar{x})$}, we set $m_p = 51$ and $N_{rec} = 325$. The results after 200 iterations are presented in Figures \ref{fig:inverse_highres_data_a} - \ref{fig:inverse_highres_data_tau0}. In this setting, $\mathbf{a}$ oscillates around the true value in the VW region, while \changed{$\btau^0$} seems to interpolate the true value fairly well. \reviewerTwo{From the misfit in Figure \ref{fig:inverse_highres_data_a_misfit} it is clear that the optimization for $\mathbf{a}$ has stagnated at a local minimum.}

\changed{Figure \ref{fig:fractal_fault_v_inversion} shows the velocity field at $t = 5$ corresponding to: the true parameter set on (\subref{fig:fractal_fault_v_lowres}) the computational grid and (\subref{fig:fractal_fault_v_highres}) the high-resolution grid, (\subref{fig:fractal_fault_v_inverted_a}, \subref{fig:fractal_fault_v_inverted_tau0}) the inverted parameters, and (\subref{fig:fractal_fault_v_initial_a}, \subref{fig:fractal_fault_v_initial_tau0}) the initial parameters.} Comparing \changed{Figures \ref{fig:fractal_fault_v_lowres} and \ref{fig:fractal_fault_v_highres}} it is clear that spurious oscillations due to lower spatial resolution are present in the velocity fields on the computational grid. The oscillations observed in $\mathbf{a}$ in Figure \ref{fig:inverse_highres_data_a_par} are therefore likely a result of numerical errors in the forward and adjoint fields being mapped into the parameter. Increasing the number of receivers to $N_{rec} = 325$ had no significant effect, and the results are therefore omitted. To further improve the results, a combination of regularization and filtering of the misfit residual is likely needed.

\begin{figure}[h!]
  \begin{subfigure}[b]{0.49\textwidth}
    \centering
    \includegraphics[width=\textwidth]{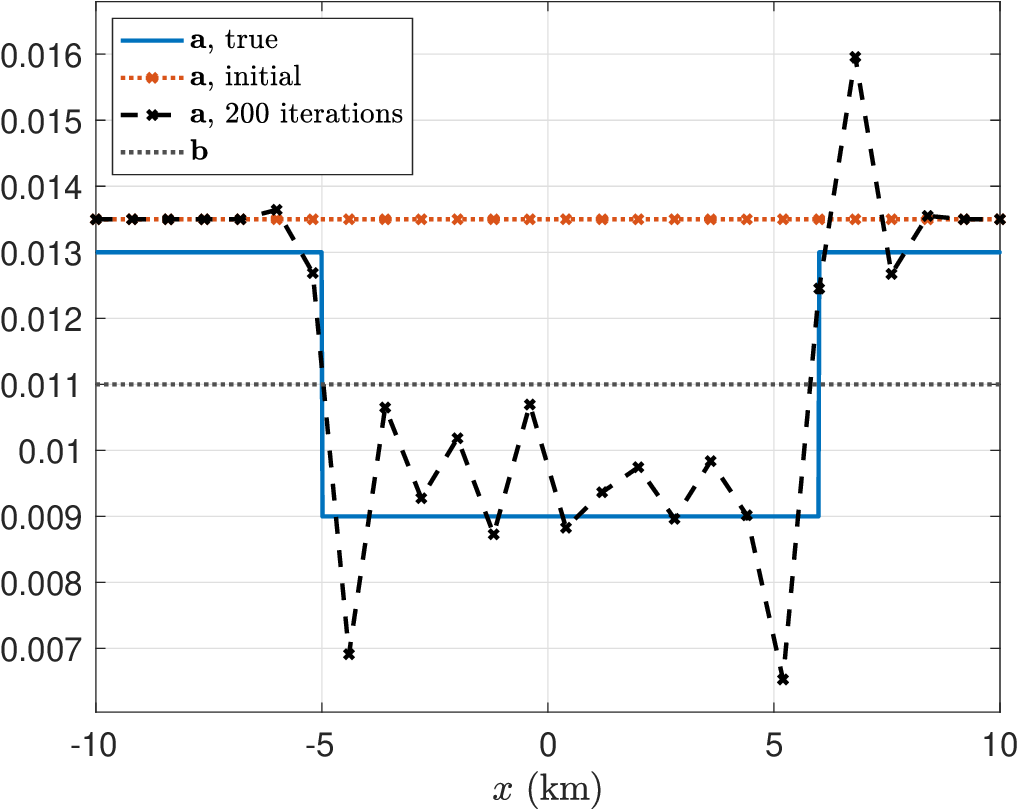}
    \caption{$\mathbf{a}$ after 200 iterations}
    \label{fig:inverse_highres_data_a_par}
  \end{subfigure}
  \hfill
  \begin{subfigure}[b]{0.49\textwidth}
      \centering
      \includegraphics[width=\textwidth]{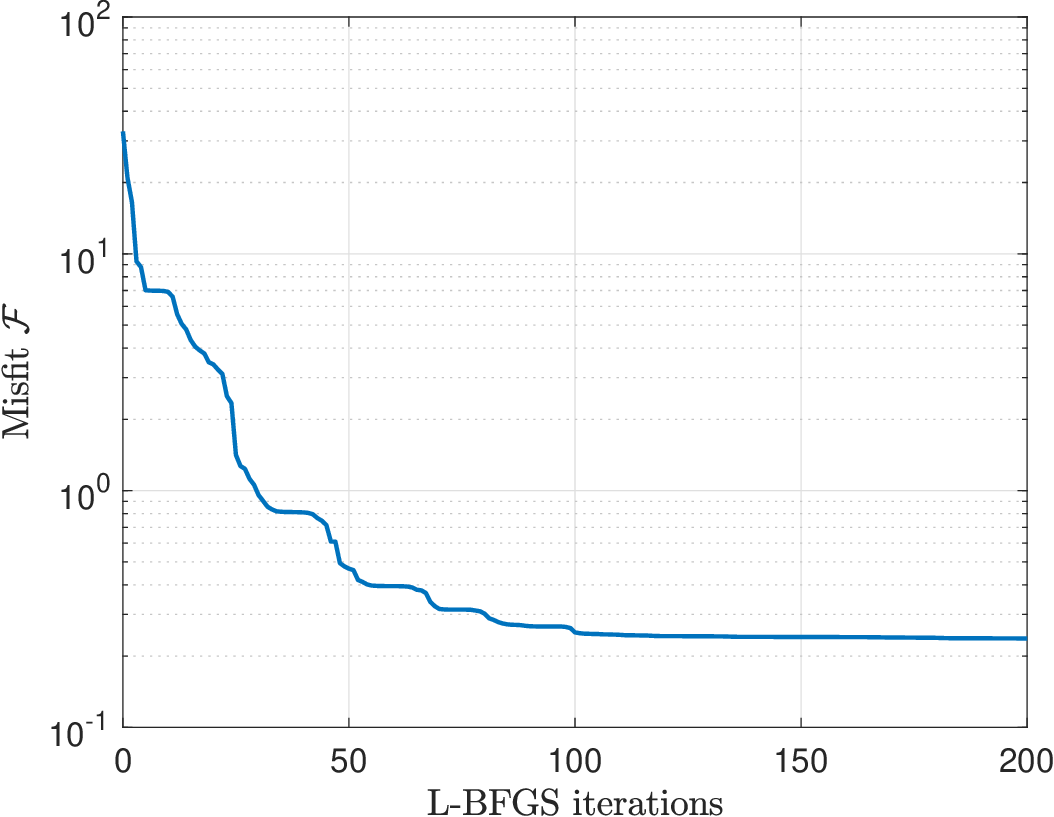}
      \caption{Misfit history}
      \label{fig:inverse_highres_data_a_misfit}
  \end{subfigure}
  \caption{Inversion for direct effect parameter $\mathbf{a}$ using high-resolution data and $N_{rec} = 88$.}
  \label{fig:inverse_highres_data_a}
\end{figure}

\begin{figure}[h!]
  \begin{subfigure}[b]{0.49\textwidth}
    \centering
    \includegraphics[width=\textwidth]{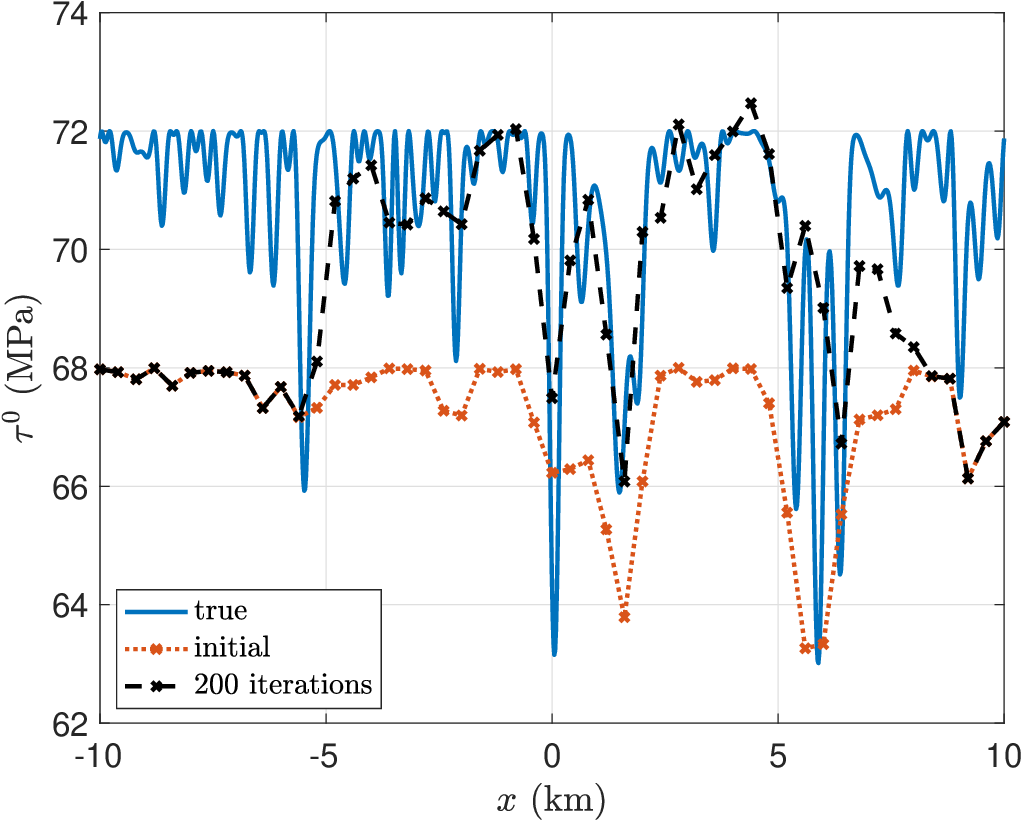}
    \caption{\changed{$\btau^0$} after 200 iterations}
  \end{subfigure}
  \hfill
  \begin{subfigure}[b]{0.49\textwidth}
      \centering
      \includegraphics[width=\textwidth]{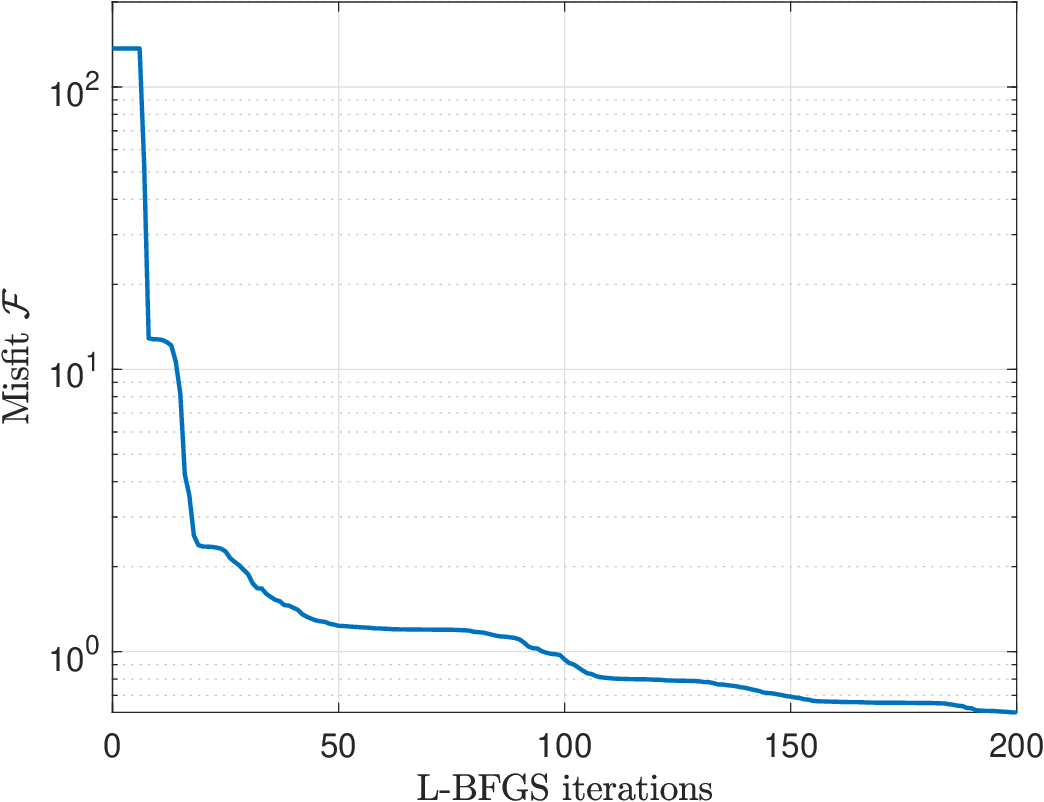}
      \caption{Misfit history}
  \end{subfigure}
  \caption{Inversion for shear traction \changed{$\btau^0$} using high-resolution data and $N_{rec} = 325$.}
  \label{fig:inverse_highres_data_tau0}
\end{figure}

\begin{figure}[h!]
  \begin{subfigure}[b]{0.49\textwidth}
    \centering
    \includegraphics[width=\textwidth]{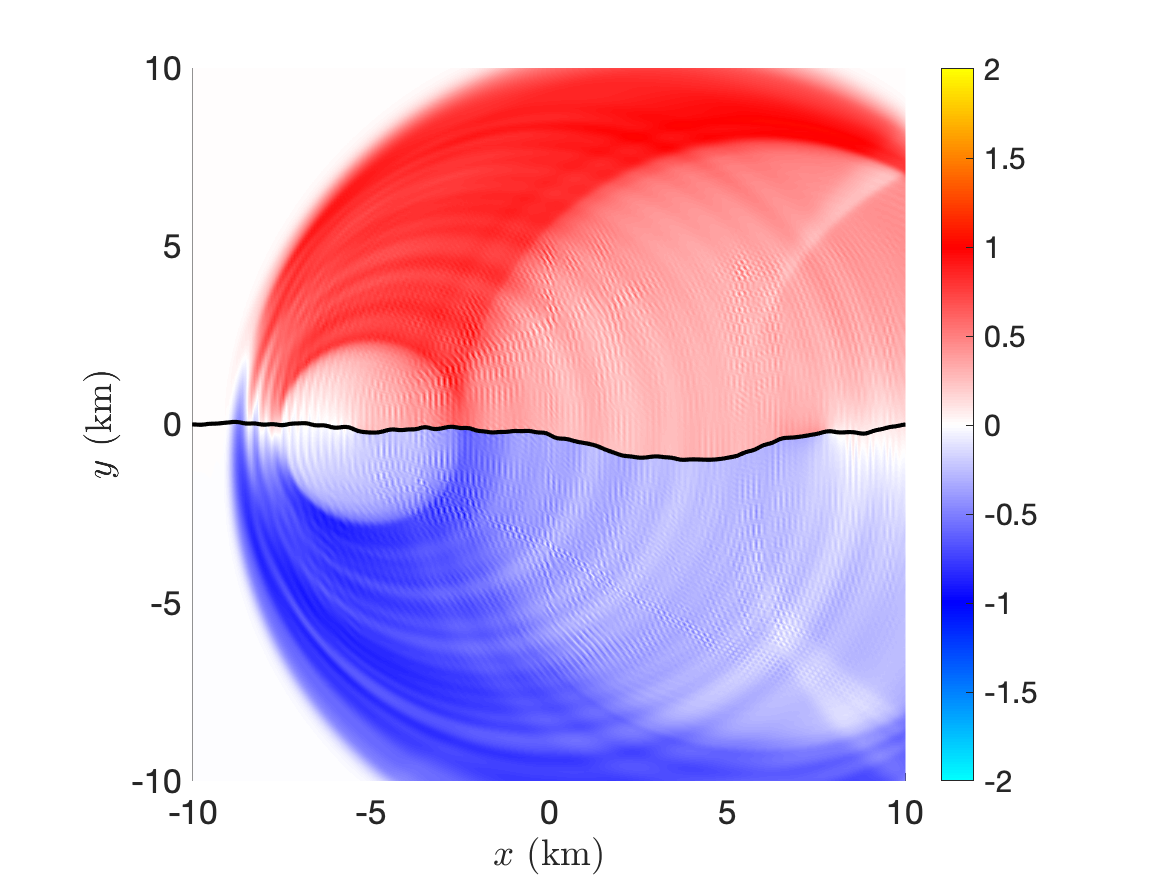}
    \caption{True parameters, computational grid}
    \label{fig:fractal_fault_v_lowres}
\end{subfigure}
\hfill
\begin{subfigure}[b]{0.49\textwidth}
    \centering
    \includegraphics[width=\textwidth]{fig/fractal_fault_v_t5_0005.eps}
    \caption{True parameters, high-resolution grid}
    \label{fig:fractal_fault_v_highres}
\end{subfigure}
  \begin{subfigure}[b]{0.49\textwidth}
    \centering
    \includegraphics[width=\textwidth]{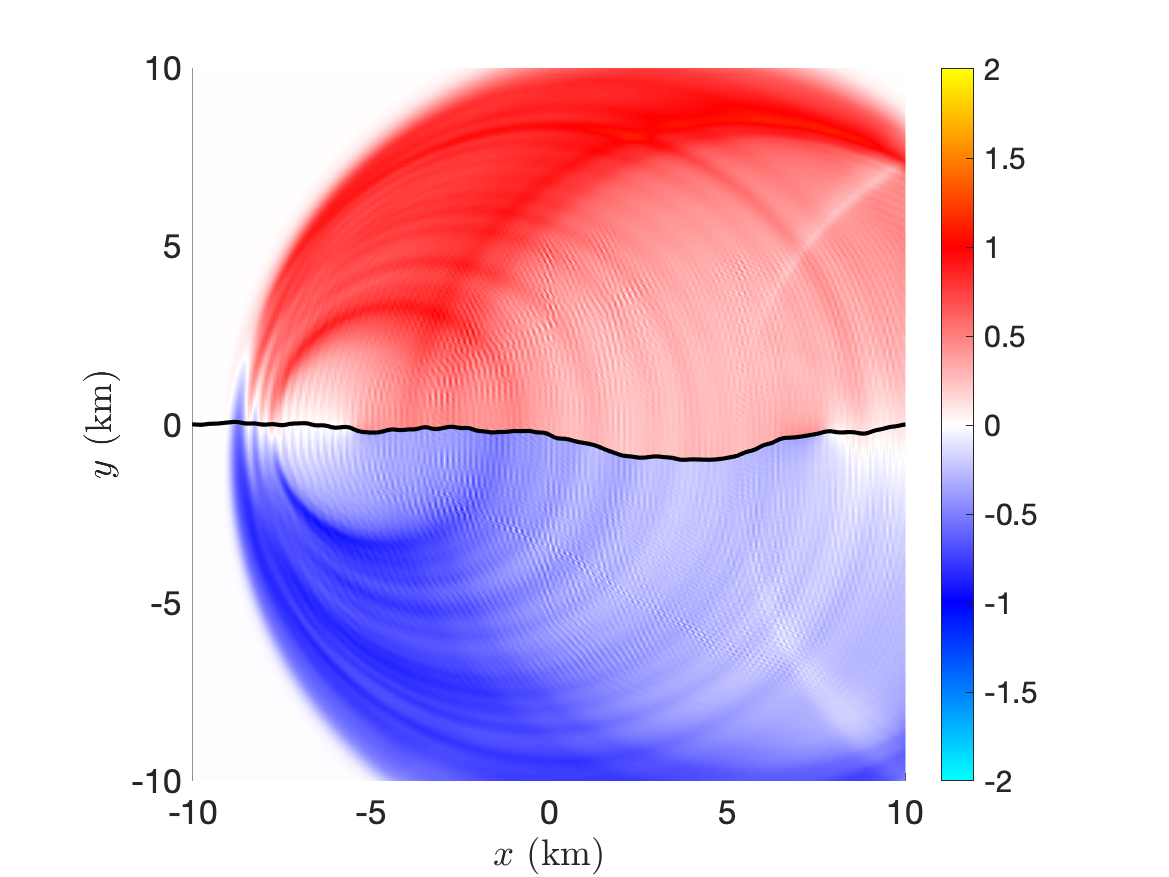}
    \caption{$\mathbf{a}$ 200 iterations}
    \label{fig:fractal_fault_v_inverted_a}
  \end{subfigure}
  \hfill
  \begin{subfigure}[b]{0.49\textwidth}
    \centering
    \includegraphics[width=\textwidth]{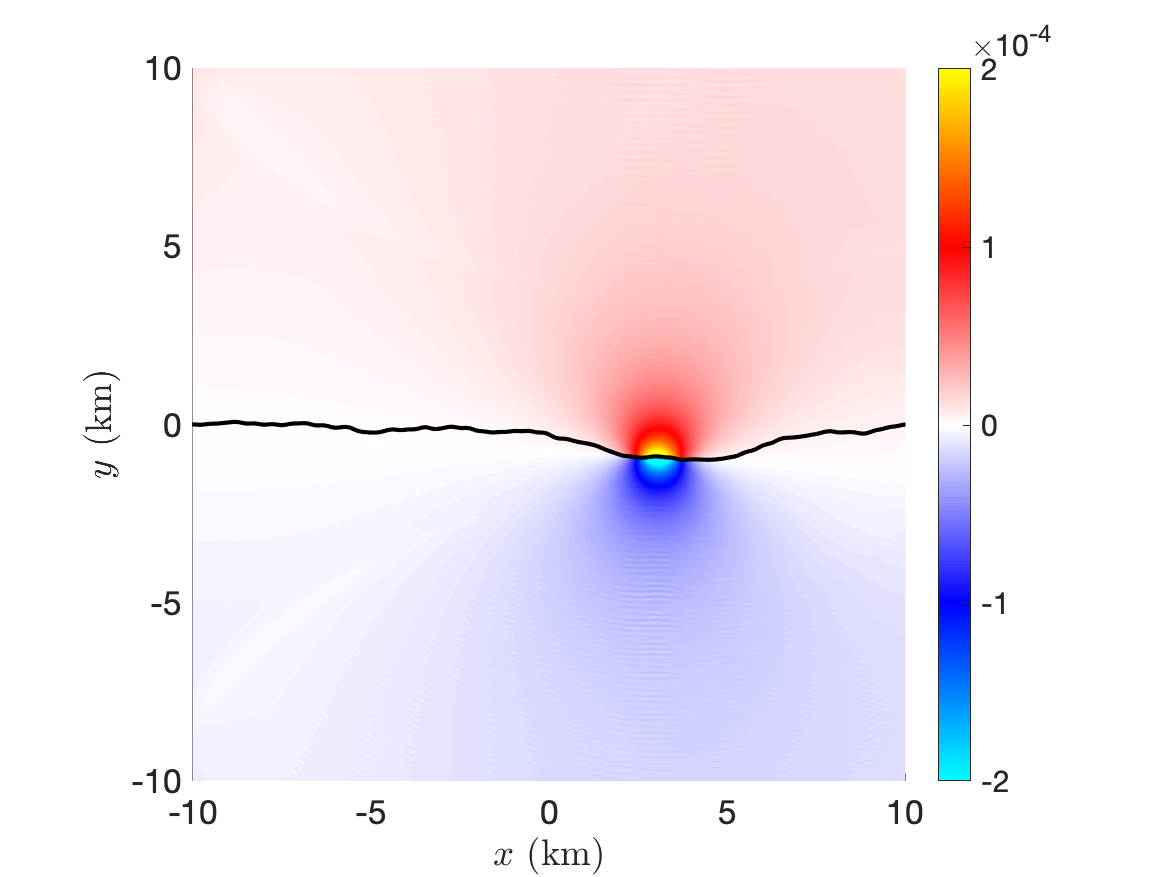}
    \caption{$\mathbf{a}$ initial parameters}
    \label{fig:fractal_fault_v_initial_a}
  \end{subfigure}
  \begin{subfigure}[b]{0.49\textwidth}
    \centering
    \includegraphics[width=\textwidth]{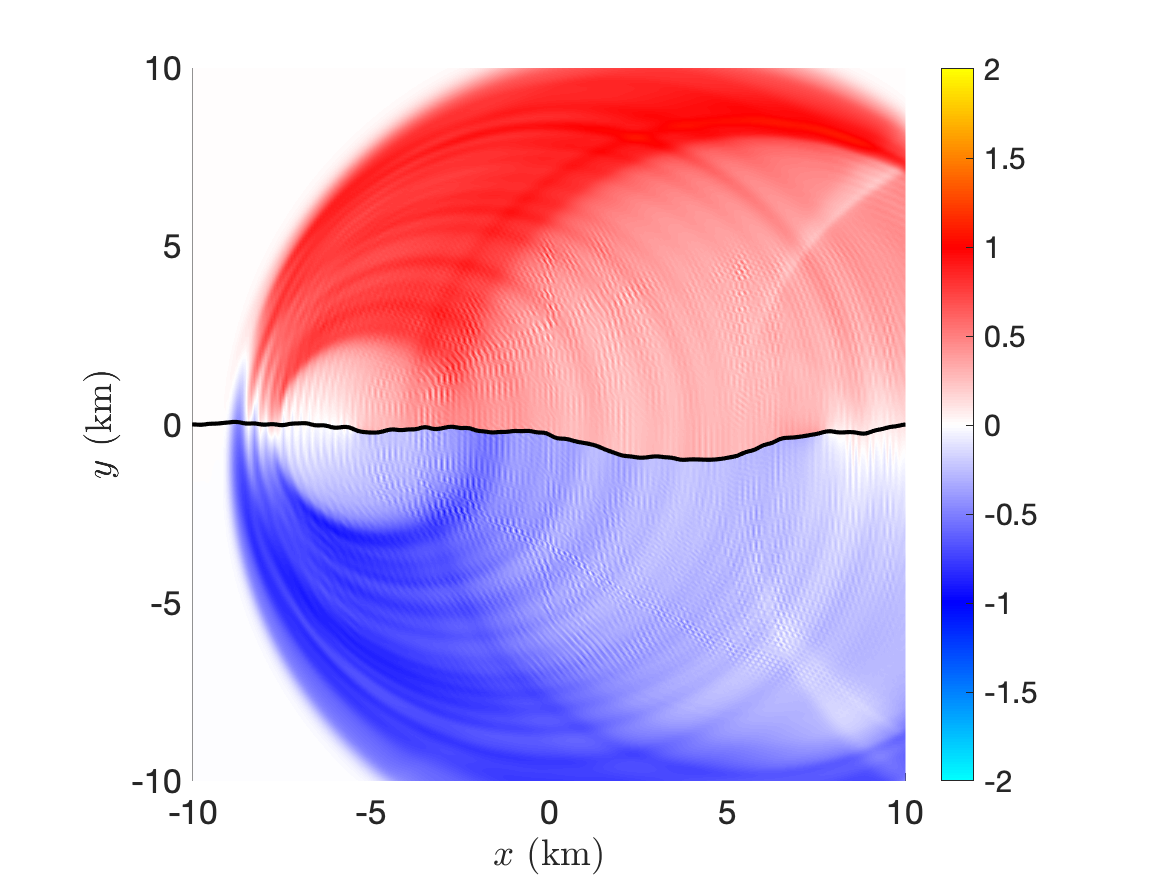}
    \caption{\changed{$\btau^0$} 200 iterations}
    \label{fig:fractal_fault_v_inverted_tau0}
  \end{subfigure}
  \hfill
  \begin{subfigure}[b]{0.49\textwidth}
    \centering
    \includegraphics[width=\textwidth]{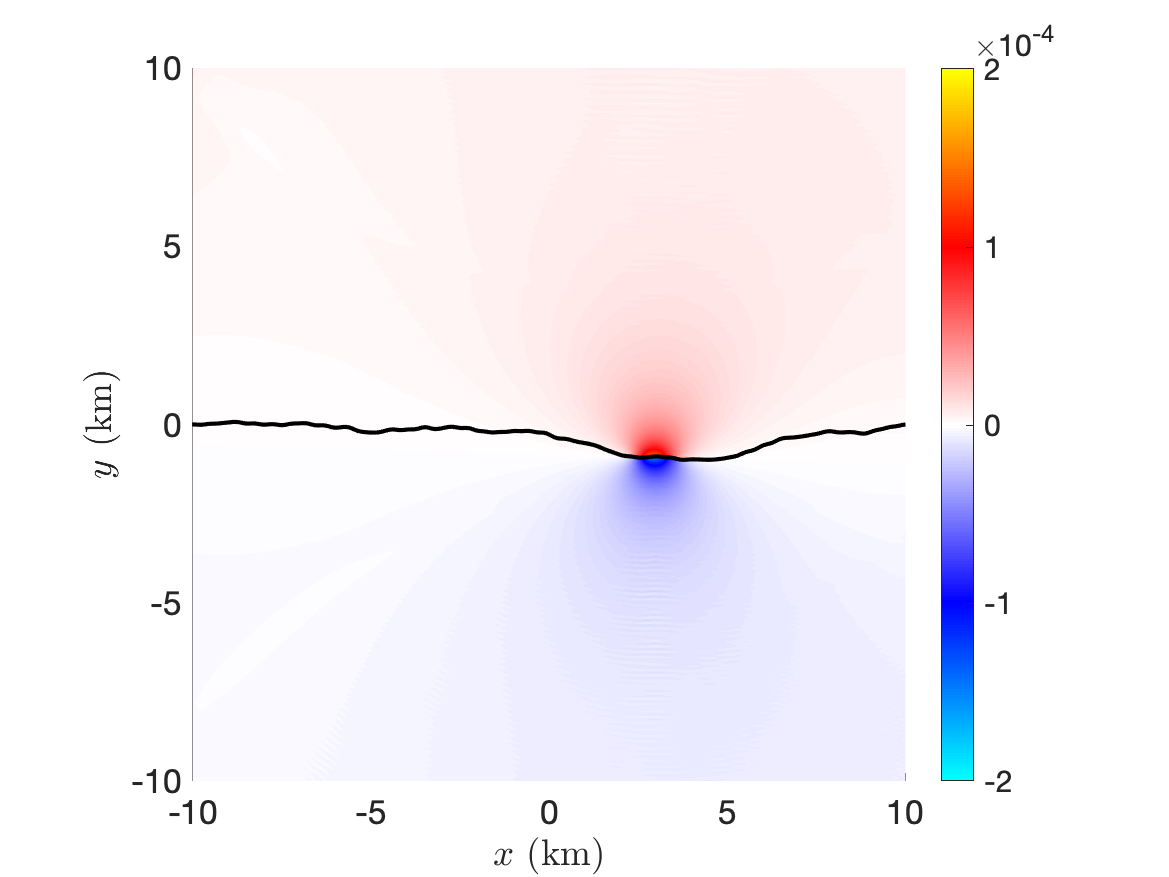}
    \caption{\changed{$\btau^0$} initial parameters}
    \label{fig:fractal_fault_v_initial_tau0}
  \end{subfigure}
  \caption{Comparison of velocity (m/s) at $t = 5$ using the true parameter sets on \changed{(\subref{fig:fractal_fault_v_lowres}) the computational grid, (\subref{fig:fractal_fault_v_highres}) the high-resolution grid}, \changed{(\subref{fig:fractal_fault_v_inverted_a}, \subref{fig:fractal_fault_v_inverted_tau0})} $\mathbf{a}$ and \changed{$\btau^0$} obtained after 200 L-BFGS iterations and \changed{(\subref{fig:fractal_fault_v_initial_a}, \subref{fig:fractal_fault_v_initial_tau0})} initial parameter values. Note the difference in scales using the initial parameter values.}
  \label{fig:fractal_fault_v_inversion}
\end{figure}

\section{Conclusion}\label{sec:conclusion}
This work presents an adjoint-based optimization framework for earthquake modeling with rate-and-state friction. The main contributions are two-fold. Firstly, the continuous adjoint equations to linear elasticity with rate-state frictional faults are derived, considering displacement and velocity misfit functionals. The adjoint equations satisfy the equations of linear elasticity in reversed time, together with an adjoint friction law and state evolution equation. The adjoint friction law and state evolution equation resemble those of linearized rate-and-state friction but with time-dependent variable coefficients from the forward problem. \changed{Additionally, for forward problems involving fault normal stress changes, the adjoint equations include a nonzero fault opening condition, also with time-dependent variable coefficients from the forward problem.} By convolving the adjoint fields with derivatives of the friction law and state evolution equation, the gradient of the misfit with respect to the inversion parameter is obtained.

Secondly, we present SBP-SAT finite difference discretizations of the forward and adjoint equations in antiplane shear, utilizing the non-stiff interface treatment in \cite{erickson_et_al_2022} and the high-order accurate boundary-optimized SBP operators of \cite{stiernstrom_et_al_2023}. SBP-preserving interpolation allows the use of a lower-dimensional parameter representation while maintaining dual consistency. In combination with RK4 time integration, a fully discrete gradient expression is presented. Due to the self-adjointness of RK4 \cite{sanz_serna_2016,matsuda_miyatake_2021}, the space-time discretization is dual consistent, meaning that the gradient expression is the true gradient of the discrete misfit and consistent with the continuous gradient. This claim is corroborated by numerical experiments in a dynamic rupture setting, presented in Section \ref{sec:gradient_verification}. \reviewerOne{Dual consistency has been shown to be beneficial in other applications, see \eg \cite{hartmann_2007,hicken_zingg_2014}, where dual consistent discretizations leads to superconvergent approximations of integral functionals in fluid dynamics. The practical benefits of dual consistency for the type of inverse modeling considered herein are yet to be determined.}

We then proceed with inversions of model problems in increasingly complex settings in Sections \ref{sec:inverse_crimes} - \ref{sec:high_resolution_data}. To start with, inverse-crime inversions are performed. Here, we are able to reconstruct the direct effect parameter and fault shear traction fairly accurately in the velocity-weakening region of the fault with appreciable slip. In the velocity-strengthening region that has negligible slip, the parameter sensitivity is close to zero, due to a lack of radiation from this region. Next, inversions utilizing synthetic high-resolution data are performed. For the direct effect parameter, a local minimum is found, which exhibits oscillations in the parameter around the true value. Inversions for shear traction interpolate the true value fairly well. We conclude that regularization, likely in combination with filtering of the misfit and adjoint source signals, is required to further improve the results.

An additional issue which will arise in real-world applications is the likely trade-off between model parameters when doing inversions for both friction parameters and initial stresses. The wavefield that is measured by the receivers, which the inversion attempts to match, is controlled by the fault slip history. Slip is controlled by the reduction in stress from the initial shear stress to the residual strength, the product of dynamic friction and normal stress. Parameter trade-offs will likely exist between $a-b$, which determines dynamic friction, and the initial shear stress. Additional trade-offs can be anticipated by examination of the crack tip equation of motion, which balances the energy release rate with the fracture energy \cite{freund1998dynamic}. \cite{garagash2021fracture} provides expressions for fracture energy in terms of rate and state frictional parameters and normal stress, and the energy release rate is a functional of the shear stress drop history. Parameter combinations that produce similar rupture histories will not be uniquely constrained.

The primary future research direction is thus to investigate suitable regularizations of the misfit functional, which might be designed by anticipating trade-offs through knowledge of dynamic fracture mechanics. In addition, we will extend the adjoint framework to include other types of measurements, \eg, strain rate from fiber optics cables, pressure from hydrophones, and InSAR measurements of surface displacement. 

\section*{Acknowledgments}
We sincerely thank St\'{e}phanie Chaillat (ENSTA Paris, Stanford visiting scholar), Laura Bagur and Alice Nassor (ENSTA Paris) for fruitful discussions during the fall of 2022, on the treatment of displacement misfits in Section \ref{sec:misfits}.

V. Stiernstr\"{o}m was supported by the Swedish Research Council (grant no. 2017-04626). The computations were enabled by resources provided by the National Academic Infrastructure for Supercomputing in Sweden (NAISS) at UPPMAX partially funded by the Swedish Research Council through grant agreement no. 2022-06725.

E. M. Dunham was supported by the U.S. National Science Foundation (EAR-1947448).

\clearpage
\appendix
\section{Summation-by-parts properties}\label{app:sbp}
Integration by parts (IBP) is key in deriving energy balances and showing conservation properties of the continuous equations. SBP operators allow us to mimic the analysis in the discrete setting. It is shown in \cite{almquist_dunham_2020,almquist_dunham_2021} that $\bD_{II}(\bmu)$ satisfies SBP properties on $\bOmega$, consistent with the continuous IBP property for $\partial_I \mu\partial_I$ on $\Omega$. The results are summarized below.

For scalar functions $u$ and $v$, the variable coefficient Laplace operator satisfies
\begin{equation}\label{eq:ibp_laplace}
    \ip{v}{\partial_I \mu\partial_I u}_{\Omega} = \ip{v}{\tau_u}_{\partial\Omega}-\ip{\partial_{I}v}{\mu\partial_{I}u}_{\Omega},
\end{equation}
where $\tau_u = \mu\partial_{\hat{n}} u$. Integrating \changed{by parts} once more yields
\begin{equation}\label{eq:ibp_laplace_twice}
  \ip{v}{\partial_I \mu\partial_I u}_{\Omega} = \ip{v}{\tau_u}_{\partial\Omega} - \ip{\tau_v}{u}_{\partial\Omega} +\ip{\partial_I \mu\partial_I v}{u}_{\Omega},
\end{equation}
In the discrete setting let $\bD_I \approx \partial_I$ on $\bOmega$ be a first-derivative SBP operator, based on the same SBP norm $\bH$ as $\bD_{II}$. Then, $\bD_{II}(\bmu)$ satisfies the SBP property
\begin{equation}\label{eq:sbp_laplace}
    \ip{\bv}{\bD_{II}(\bmu)\bu}_{\bOmega} = \ip{\bv}{\btau_u}_{\bm{\partial\Omega}}-\ip{\bD_{I}\bv}{\dbar{\bmu}\bD_{I}\bu}_{\bOmega} - \bv^T\bR_\Delta\bu,
\end{equation}
where $\btau_u = \bT\bu$, with $\bT \approx \mu\partial_{\hat{n}}$ and $\bR_\Delta = \bR_\Delta^T \ge 0$, $\bR_\Delta \approx 0$. Thus \eqref{eq:sbp_laplace} is a discrete counterpart to the IBP property \eqref{eq:ibp_laplace}. Applying the SBP property once more yields
\begin{equation}\label{eq:sbp_laplace_twice}
  \ip{\bv}{\bD_{II}(\bmu)\bu}_{\bOmega} = \ip{\bv}{\btau_u}_{\bm{\partial\Omega}} - \ip{\btau_v}{\bu}_{\bm{\partial\Omega}} + \ip{\bD_{II}(\bmu)\bv}{\bu}_{\bOmega},
\end{equation}
For details on how $\bD_I$ and $\bT$ are constructed see \cite{almquist_dunham_2020,almquist_dunham_2021,stiernstrom_et_al_2023}.

The SBP property \eqref{eq:sbp_laplace} facilitates the derivation of a discrete energy rate for scalar wave equations. This is achieved by replacing $\bv$ by $\dot{\bu}$, in which case the volume terms $\ip{\cdot}{\cdot}_{\bOmega}$ contribute to the discrete potential energy rate. The discrete energy rates are then useful for guiding the enforcement of boundary or interface conditions, such that an energy-stable scheme may be obtained. See, \eg, \cite{duru_et_al_2019,wang_petersson_2019,almquist_dunham_2020,almquist_dunham_2021,erickson_et_al_2022}. The property \eqref{eq:sbp_laplace_twice} states that $\bD_{II}(\bmu)$ augmented with suitable boundary or interface treatment is self-adjoint with respect to $\bH_{\bOmega}$.

\section{Semi-discrete adjoint sources}\label{app:semi_discr_receiver}
To derive the semi-discrete gradient we seek an adjoint source $\bQ^\dagger$ satisfying
\begin{equation}\label{eq:semi_disc_adj_source_rel}
  \ip{\bQ^\dagger}{\frac{\partial \dot{\bu}}{\partial\bp}}_{\bOmega \times \T} = \frac{\partial \bm{\F}}{\partial \bp }.
\end{equation}
Consider the gradient of the misfit \eqref{eq:anti_plane_shear_semi_discr_misfit} with a single receiver. For convenience, we drop the superscripts. By the chain rule, it follows that
\begin{equation}
  \begin{aligned}
  \frac{\partial \bm{\F}}{\partial \bp } &= \frac{\partial }{\partial \bp}\left(\frac{1}{2} \int_{\T} \mathbf{r}^2\mathrm{d}t\right)= \int_{\T} \changed{\mathbf{r}}\ip{\hat{\bm{\delta}}}{\frac{\partial\mathbf{m}}{\partial\bp}}_{\bOmega}\mathrm{d}t = \ip{\mathbf{r}\hat{\bm{\delta}}}{\frac{\partial\mathbf{m}}{\partial\bp}}_{\bOmega \times \T}.
\end{aligned}
\end{equation}
If $\mathbf{m} = \dot{\bu}$, we have $\frac{\partial \bm{\F}}{\partial \bp} = \ip{\mathbf{r}\hat{\bm{\delta}}}{\frac{\partial \dot{\bu}}{\partial\bp}}_{\bOmega \times \T}$, which is of the desired form \eqref{eq:semi_disc_adj_source_rel}. If $\mathbf{m} = \bu$, introduce $\hat{\mathbf{r}}(t) = \int_0^t \mathbf{r}(t')\mathrm{d}t' + \hat{\mathbf{r}}_0$ with $\hat{\mathbf{r}}_0 = -\int_0^T \mathbf{r}(t')\mathrm{d}t'$ and integrate $\ip{\mathbf{r}\hat{\bm{\delta}}}{\frac{\partial\bu}{\partial\bp}}_{\bOmega \times \T}$ by parts in time to obtain $\frac{\partial \bm{\F}}{\partial \bp} = \ip{-\hat{\mathbf{r}}\hat{\bm{\delta}}}{\frac{\partial\dot{\bu}}{\partial\bp}}_{\bOmega \times \T}$. Thus,
\begin{equation}\label{eq:semi_discr_misfit_grad}
  \frac{\partial \bm{\F}}{\partial \bp} = \ip{\bS}{\frac{\partial\dot{\bu}}{\partial\bp}}_{\bOmega \times \T}, \quad \bS(t) = 
  \begin{cases} 
    -\hat{\mathbf{r}}(t)\hat{\bm{\delta}}(\bar{\bx}-\bar{\bx}_r),& \mathbf{m} = \bu, \\
    \mathbf{r}(t)\hat{\bm{\delta}}(\bar{\bx}-\bar{\bx}_r),& \mathbf{m} = \dot{\bu}.
  \end{cases}
\end{equation}
Summing $\bS$ over all receivers results in $\bQ^\dagger$ in \eqref{eq:anti_plane_shear_discr_adj}.

\section{Derivation of the semi-discrete gradient}\label{app:proof_grad}
To derive the gradient of \eqref{eq:anti_plane_shear_semi_discr_misfit}, the procedure in Section \ref{sec:adjoint_eqs} is followed. Omitting external forcing (since it does not contribute to the gradient) the Lagrangian misfit functional to \eqref{eq:anti_plane_shear_semi_discr_misfit} reads
\begin{equation}\label{eq:semi_discr_lagrangian}
  \begin{aligned}
    \bm{\L} = \bm{\F} &+ \ip{\dot{\bu}^\dagger}{\dbar{\brho}\ddot{\bu} - \bL(\bu,\bu^*,\btau^*)}_{\bOmega_+ \times \T} \\
    &+ \ip{\dot{\bu}^\dagger}{\dbar{\brho}\ddot{\bu} - \bL(\bu,\bu^*,\btau^*)}_{\bOmega_- \times \T} \\
    &+ \ip{\bPsi^\dagger}{\dot{\bPsi} - \bG}_{\bGamma \times \T},
  \end{aligned}
\end{equation}
where we omit subscripts $\pm$ of the grid functions since they are implied by the inner products. Consider first $\ip{\dot{\bu}^\dagger}{\bL(\bu,\bu^*,\btau^*)}_{\bOmega_+ \times \T}$. By \eqref{eq:anti_plane_shear_char_SAT} it follows that
\begin{equation}\label{eq:weak_form_char_fwd}
\begin{aligned}
&\ip{\dot{\bu}^\dagger}{\bL(\bu,\bu^*,\btau^*)}_{\bOmega_+ \times \T} = \\
&\ip{\dot{\bu}^\dagger}{\bD_{II}(\bmu) \bu + \text{SAT}(\bu,\bu^*,\btau^*)}_{\bOmega_+ \times \T}  = \\
&\ip{\dot{\bu}^\dagger}{\bD_{II}(\bmu) \bu}_{\bOmega_+ \times \T} +\ip{\dot{\bu}^\dagger}{\btau^* - \btau}_{\bm{\partial\Omega}_+ \times \T} - \ip{\dot{\btau}^\dagger}{\bu^* - \bu}_{\bm{\partial\Omega}_+ \times \T}.
\end{aligned}
\end{equation}
Swapping forward and adjoint variables yields
\begin{equation}\label{eq:weak_form_char_adj}
\begin{aligned}
&\ip{\dot{\bu}}{\bL(\bu^\dagger,\bu^{\dagger*},\btau^{\dagger*})}_{\bOmega_+ \times \T} = \\
&\ip{\dot{\bu}}{\bD_{II}(\bmu) \bu^\dagger}_{\bOmega_+ \times \T} +
\ip{\dot{\bu}}{\btau^{\dagger*} - \btau^\dagger}_{\bm{\partial\Omega}_+ \times \T} - \ip{\dot{\btau}}{\bu^{\dagger*} - \bu^\dagger}_{\bm{\partial\Omega}_+ \times \T}.
\end{aligned}
\end{equation}
Next, using the SBP property \eqref{eq:sbp_laplace_twice} and IBP in time it follows that
\begin{equation}\label{eq:sbp_volume_term}
    \begin{aligned}
         &\ip{\dot{\bu}^\dagger}{\bD_{II}(\bmu) \bu}_{\bOmega_+ \times \T} = \\
         -&\ip{\bD_{II}(\bmu)\bu^\dagger}{\dot{\bu}}_{\bOmega_+ \times \T}
         + \ip{\dot{\bu}^\dagger}{\btau}_{\bm{\partial\Omega}_+ \times \T} - \ip{\dot{\btau}^\dagger}{\bu}_{\bm{\partial\Omega}_+ \times \T},
    \end{aligned}
\end{equation}
\changed{where terms evaluated at $t=0$ and $t=T$ vanish due to the initial and terminal conditions \eqref{eq:anti_plane_shear_discr_fwd} and \eqref{eq:anti_plane_shear_discr_adj}}. Due to the $p$-independence of data, we assumed homogeneous initial conditions for the forward problem, since such terms would vanish in the gradient expression regardless. Using \eqref{eq:sbp_volume_term} in \eqref{eq:weak_form_char_fwd} together with \eqref{eq:weak_form_char_adj} yields
\begin{equation}\label{eq:weak_form_char_rel}
\begin{aligned}
&\ip{\dot{\bu}^\dagger}{\bL(\bu,\bu^*,\btau^*)}_{\bOmega_+ \times \T} = \\
-&\ip{\bD_{II}(\bmu)\bu^\dagger}{\dot{\bu}}_{\bOmega_+ \times \T} + \ip{\dot{\bu}^\dagger}{\btau^*}_{\bm{\partial\Omega}_+ \times \T} - \ip{\dot{\btau}^\dagger}{\bu^*}_{\bm{\partial\Omega}_+ \times \T} =\\
-&\ip{\dot{\bu}^\dagger}{\bL(\bu,\bu^*,\btau^*)}_{\bOmega_+ \times \T} + BT
\end{aligned}
\end{equation}
where
\begin{equation}\label{eq:weak_form_char_IT}
\begin{aligned}
BT &= \ip{\btau^{\dagger*} - \btau^\dagger}{\dot{\bu}}_{\bm{\partial\Omega}_+ \times \T} - \ip{\bu^{\dagger*} - \bu^\dagger}{\dot{\btau}}_{\bm{\partial\Omega}_+ \times \T} \\
  &+\ip{\dot{\bu}^\dagger}{\btau^*}_{\bm{\partial\Omega}_+ \times \T} - \ip{\dot{\btau}^\dagger}{\bu^*}_{\bm{\partial\Omega}_+ \times \T} \\
&= \ip{\btau^{\dagger*} - \btau^\dagger}{\dot{\bu}}_{\bm{\partial\Omega}_+ \times \T} + \ip{\dot{\bu}^{\dagger*} - \dot{\bu}^\dagger}{\btau}_{\bm{\partial\Omega}_+ \times \T}\\
& +\ip{\dot{\bu}^\dagger}{\btau^*}_{\bm{\partial\Omega}_+ \times \T} + \ip{\btau^\dagger}{\dot{\bu}^*}_{\bm{\partial\Omega}_+ \times \T}
\end{aligned}
\end{equation}
with the last equality obtained by IBP in time. Next we perform a change of variables in the term $\ip{\dot{\bu}^{\dagger*} - \dot{\bu}^\dagger}{\btau}_{\bm{\partial\Omega}_+ \times \T}$ from $\btau$ to $\tilde{\btau}$ using \eqref{eq:modified_tau}
\begin{equation}
\begin{aligned}
&\ip{\dot{\bu}^{\dagger*} - \dot{\bu}^\dagger}{\btau}_{\bm{\partial\Omega}_+ \times \T} = \ip{\dot{\bu}^{\dagger*} - \dot{\bu}^\dagger}{\tilde{\btau} - \gamma(\bu^* - \bu)}_{\bm{\partial\Omega}_+ \times \T} = \\
 & \ip{\dot{\bu}^{\dagger*} - \dot{\bu}^\dagger}{\tilde{\btau}}_{\bm{\partial\Omega}_+ \times \T} + \ip{\gamma(\bu^{\dagger*} - \bu^\dagger)}{\dot{\bu}^*}_{\bm{\partial\Omega}_+ \times \T} -\ip{\gamma(\bu^{\dagger*} - \bu^\dagger)}{\dot{\bu}}_{\bm{\partial\Omega}_+ \times \T},
\end{aligned}
\end{equation}
where IBP in time was used to obtain the last equality. Combining the result with \eqref{eq:weak_form_char_IT} and substituting $\tilde{\btau}^\dagger$ results in
\begin{equation}\label{eq:IT_simplified_1}
  \begin{aligned}
  BT &= \ip{\btau^{\dagger*} - \tilde{\btau}^\dagger}{\dot{\bu}}_{\bm{\partial\Omega}_+ \times \T} + \ip{\dot{\bu}^{\dagger*} - \dot{\bu}^\dagger}{\tilde{\btau}}_{\bm{\partial\Omega}_+ \times \T} \\
  &+\ip{\dot{\bu}^\dagger}{\btau^*}_{\bm{\partial\Omega}_+ \times \T} + \ip{\tilde{\btau}^\dagger}{\dot{\bu}^*}_{\bm{\partial\Omega}_+ \times \T} .
  \end{aligned}
\end{equation}
Next we use \eqref{eq:preserve_outgoing} and \eqref{eq:preserve_outgoing_adj} to obtain 
\begin{equation}\label{eq:preserve_outgoing_mod_tau}
    \begin{aligned}
        \tilde{\btau} &= \btau^* + Z(\dot{\bu} - \dot{\bu}^*), \\
        \tilde{\btau}^\dagger &= \btau^{\dagger*} -Z(\dot{\bu}^\dagger - \dot{\bu}^{\dagger*}).
    \end{aligned}
\end{equation}
Substituting \eqref{eq:preserve_outgoing_mod_tau} into \eqref{eq:IT_simplified_1} and eliminating terms yields
\begin{equation}\label{eq:IT_simplified_2}
\begin{aligned}
    BT 
    &= \ip{\dot{\bu}^{\dagger*}}{\btau^*}_{\bm{\partial\Omega}_+ \times \T} + \ip{\btau^{*\dagger}}{\dot{\bu}^*}_{\bm{\partial\Omega}_+ \times \T}.
\end{aligned}
\end{equation}
Since the target values satisfy the boundary conditions, $BT$ vanishes on all boundaries but the fault. This is straightforwardly seen on \changed{rigid-wall} or \reviewerOne{traction-free} boundaries, where $\bu^*$ and $\bu^{\dagger*}$, or $\btau^*$ and $\btau^{\dagger*}$, respectively, are zero. On \changed{non-reflecting} boundaries, substituting \eqref{eq:tau_star_bc}, \eqref{eq:ode_ustar_bc}, \eqref{eq:tau_star_bc_adj}, \eqref{eq:ode_ustar_bc_adj} with $\bR = \mathbf{0}$ into \eqref{eq:IT_simplified_2} yields the result. Adding contributions of \eqref{eq:IT_simplified_2} from $\bOmega_-$ yields
\begin{equation} \label{eq:IT_simplified_3}
IT = BT_+ + BT_- = \ip{\btau^{\dagger*}_+}{ \bV^*}_{\bGamma \times \T}  -
\ip{\bV^{\dagger *}}{\btau^*_+}_{\bGamma \times \T}.
\end{equation}
where it is assumed that $\bV^* = \jump{\dot{\bu}^*}$ and $\bV^{\dagger*} = -\jump{\dot{\bu}^{\dagger*}}$.

By \eqref{eq:semi_discr_lagrangian} and \eqref{eq:weak_form_char_rel} it then follows that
\begin{equation}
\begin{aligned}
\bm{\L} = \bm{\F} &- \ip{\dbar{\brho}\ddot{\bu}^\dagger - \bL(\bu^\dagger,\bu^{\dagger*},\btau^{\dagger*})}{\dot{\bu}}_{\bOmega_+ \times \T} \\
                  &- \ip{\dbar{\brho}\ddot{\bu}^\dagger - \bL(\bu^\dagger,\bu^{\dagger*},\btau^{\dagger*})}{\dot{\bu}}_{\changed{\bOmega_- \times \T}} \\
                  &- ST - IT,
\end{aligned}
\end{equation}
where 
\begin{equation}
    ST = \ip{\dot{\bPsi}^\dagger}{\bPsi}_{\bGamma \times \T} + \ip{\bPsi^\dagger}{\bG}_{\bGamma \times \T} - \ip{\bPsi_0^\dagger}{\bPsi_0}_{\bGamma},
\end{equation}
is obtained from IBP in time, where the term at $t = T$ vanishes due to the terminal condition in \eqref{eq:anti_plane_shear_discr_adj}.

We will now derive the gradient $\frac{\partial \bm{\L}}{\partial \bp}$. Combining \eqref{eq:semi_disc_adj_source_rel} with the adjoint scheme \eqref{eq:anti_plane_shear_discr_adj}, the gradient of the volume terms vanishes such that  $\frac{\partial \bm{\L}}{\partial \bp} = -\frac{\partial ST }{\partial \bp} - \frac{\partial IT}{\partial \bp}$. First, consider the inversion for a parameter in either $F$ or $G$. Since 
\begin{equation}
    \frac{\partial \btau_+^*}{\partial \bp} = -\left(\dbar{\bF}_V \frac{\partial \bV^*}{\partial \bp} + \dbar{\bF}_\Psi \frac{\partial \bPsi}{\partial \bp} + \dbar{\bF}_p \frac{\partial \bp}{\partial \bp}\right),
\end{equation}
and
\begin{equation}
        \frac{\partial ST}{\partial \bp}  = 
        \ip{\dot{\bPsi}^\dagger}{\frac{\partial \bPsi}{\partial \bp}}_{\bGamma \times \T} + \ip{\bPsi^\dagger}{\dbar{\bG}_V\frac{\partial \bV^*}{\partial \bp} + \dbar{\bG}_\Psi\frac{\partial \bPsi}{\partial \bp} + \dbar{\bG}_p\frac{\partial \bp}{\partial \bp}}_{\bGamma \times \T},
\end{equation}
we obtain
\begin{equation}
\begin{aligned}
    \frac{\partial IT}{\partial \bp} &= \ip{\btau^{\dagger*}_+}{ \frac{\partial \bV^*}{\partial \bp} }_{\bGamma \times \T} + \ip{\bV^{\dagger *}}{\dbar{\bF}_V \frac{\partial \bV^*}{\partial \bp}  + \dbar{\bF}_\Psi \frac{\partial \bPsi}{\partial \bp}  + \dbar{\bF}_p \frac{\partial \bp}{\partial \bp}}_{\bGamma \times \T} 
\end{aligned}
\end{equation}
and
\begin{equation}
\begin{aligned}
  \frac{\partial ST}{\partial \bp} + \frac{\partial IT}{\partial \bp} 
    = &\ip{\btau^{\dagger*}_+ + \dbar{\bF}_V \bV^{\dagger *} + \dbar{\bG}_V \bPsi^{\dagger}}{ \frac{\partial \bV^*}{\partial \bp} }_{\bGamma \times \T} \\
    + &\ip{\dot{\bPsi}^{\dagger}+\dbar{\bF}_\Psi \bV^{\dagger *} + \dbar{\bG}_\Psi \bPsi^{\dagger}}{  \frac{\partial \bPsi}{\partial \bp} }_{\bGamma \times \T} \\
    + &\ip{\dbar{\bG}_p \bPsi^{\dagger} + \dbar{\bF}_p \bV^{\dagger *}}{\frac{\partial \bp}{\partial \bp} }_{\bGamma \times \T}. 
\end{aligned}
\end{equation}
Using the discrete friction law \eqref{eq:tau_star_fault_adj} and state evolution equation \eqref{eq:state_evolution_discr_adj}, we arrive at
\begin{equation}\label{eq:grad_fault_p}
  \frac{\partial ST}{\partial \bp} + \frac{\partial IT}{\partial \bp} 
    = \ip{\dbar{\bG}_p \bPsi^{\dagger} + \dbar{\bF}_p \bV^{\dagger *}}{\frac{\partial \bp}{\partial \bp} }_{\bGamma \times \T}. 
\end{equation}
Similarly, if inverting for the initial state $\bPsi_0$
\begin{equation}
  \frac{\partial \btau_+^*}{\partial \bPsi_{0}} = -\left(\dbar{\bF}_V \frac{\partial \bV^*}{\partial \bPsi_{0}} + \dbar{\bF}_\Psi \frac{\partial \bPsi}{\partial \bPsi_{0}}\right),
\end{equation}
while
\begin{equation}
  \begin{aligned}
  \frac{\partial ST}{\partial \bPsi_{0}}  &= 
  \ip{\dot{\bPsi}^\dagger}{\frac{\partial \bPsi}{\partial \bPsi_{0}}}_{\bGamma \times \T} + \ip{\bPsi^\dagger}{\dbar{\bG}_V\frac{\partial \bV^*}{\partial \bPsi_{0}} + \dbar{\bG}_\Psi\frac{\partial \bPsi}{\partial \bPsi_{0}}}_{\bGamma \times \T} \\
  &- \ip{\bPsi_0^\dagger}{\frac{\partial \bPsi_0}{\partial \bPsi_{0}}}_{\bGamma},
  \end{aligned}
\end{equation}
such that
\begin{equation}\label{eq:grad_fault_psi0}
  \frac{\partial ST}{\partial \bPsi_{0}} + \frac{\partial IT}{\partial \bPsi_{0}} 
    = - \ip{\bPsi_0^\dagger}{\frac{\partial \bPsi_0}{\partial \bPsi_{0}}}_{\bGamma} .
\end{equation}
Since (in \changed{index} notation) $\frac{\partial \bp_j}{\partial \bp_i} = \delta_{ij}$ (and similarly for $\bPsi_0$) \eqref{eq:anti_plane_shear_semi_discr_grad} follows.

\begin{remark}
To reach \eqref{eq:grad_fault_p} and \eqref{eq:grad_fault_psi0} we have assumed  $\bV^* = \jump{\dot{\bu}^*}$ and $\bV^{\dagger*} = -\jump{\dot{\bu}^{\dagger*}}$. This is true for $\bV^{\dagger*}$, since the adjoint scheme satisfies \eqref{eq:V_star_nonlin_adj} exactly, due to $F^{\dagger}$ being linear in $V^\dagger$ such that the solve is trivial. However, due to the non-linear solve required in \eqref{eq:V_star_nonlin},
\begin{equation}  
  \jump{\dot{\bu}^*} = \bV^* + \changed{\bm{\varepsilon}}.
\end{equation}
When adding the contributions from the interface terms, this leads to
\begin{equation} \label{eq:IT_simplified_with_error}
  IT = \ip{\btau^{\dagger*}_+}{\bV^* + \changed{\bm{\varepsilon}}}_{\bGamma \times \T}  -
\ip{\bV^{\dagger *}}{\btau^*_+}_{\bGamma \times \T} ,
\end{equation}
such that
\begin{equation}
  \frac{\partial \bm{\F}^{(\varepsilon)}}{\partial \bp_i} = \frac{\partial \bm{\F}}{\partial \bp_i} + \ip{\btau^{\dagger*}_+}{\frac{\partial \changed{\bm{\varepsilon}}}{\partial \bp_i}}_{\bGamma \times \T}
\end{equation}
is the resulting gradient. Thus, the error in the non-linear solve results in an error in the semi-discrete gradient. However, \changed{since $F$ is invertible and monotone on the interval $[V_{min}, V_{max}]$, the root can bracketed, \ie, $\bV^* \in [\bV_{min}, \bV_{max}]$} \cite{erickson_et_al_2022}. Therefore we can control \changed{$\bm{\varepsilon}$} arbitrarily through the tolerance in a bracketed root finding algorithm, \eg bisection.
\end{remark}

\bibliographystyle{abbrvnat}
\bibliography{bibfile}

\end{document}